\documentclass[11pt]{amsart}
     \usepackage{pstricks, pst-node, pst-tree,pstricks-add,pst-eucl}
     \usepackage{multido}
     \usepackage{verse}
     \usepackage{pst-text}
     \usepackage{pst-bar}
     \usepackage{pst-bezier}
     \usepackage{amssymb}
     \usepackage{pst-plot}
     \usepackage{comment}
     \usepackage{graphicx}
     \usepackage{csquotes}
     \usepackage{palatino}
     \usepackage{mathtools}

     \usepackage[margin=2.5cm, vmargin={1.5cm}]{geometry}

     \usepackage{hyperref}
     \hypersetup{
       colorlinks=true,
       linkcolor=blue,
       filecolor=magenta,
       urlcolor=cyan,
     }
\makeatletter
\let\save@mathaccent\mathaccent
\newcommand*\if@single[3]{%
  \setbox0\hbox{${\mathaccent"0362{#1}}^H$}%
  \setbox2\hbox{${\mathaccent"0362{\kern0pt#1}}^H$}%
  \ifdim\ht0=\ht2 #3\else #2\fi
  }
\newcommand*\rel@kern[1]{\kern#1\dimexpr\macc@kerna}
\newcommand*\widebar[1]{\@ifnextchar^{{\wide@bar{#1}{0}}}{\wide@bar{#1}{1}}}
\newcommand*\wide@bar[2]{\if@single{#1}{\wide@bar@{#1}{#2}{1}}{\wide@bar@{#1}{#2}{2}}}
\newcommand*\wide@bar@[3]{%
  \begingroup
  \def\mathaccent##1##2{%
    \let\mathaccent\save@mathaccent
    \if#32 \let\macc@nucleus\first@char \fi
    \setbox\z@\hbox{$\macc@style{\macc@nucleus}_{}$}%
    \setbox\tw@\hbox{$\macc@style{\macc@nucleus}{}_{}$}%
    \dimen@\wd\tw@
    \advance\dimen@-\wd\z@
    \divide\dimen@ 3
    \@tempdima\wd\tw@
    \advance\@tempdima-\scriptspace
    \divide\@tempdima 10
    \advance\dimen@-\@tempdima
    \ifdim\dimen@>\z@ \dimen@0pt\fi
    \rel@kern{0.6}\kern-\dimen@
    \if#31
      \overline{\rel@kern{-0.6}\kern\dimen@\macc@nucleus\rel@kern{0.4}\kern\dimen@}%
      \advance\dimen@0.4\dimexpr\macc@kerna
      \let\final@kern#2%
      \ifdim\dimen@<\z@ \let\final@kern1\fi
      \if\final@kern1 \kern-\dimen@\fi
    \else
      \overline{\rel@kern{-0.6}\kern\dimen@#1}%
    \fi
  }%
  \macc@depth\@ne
  \let\math@bgroup\@empty \let\math@egroup\macc@set@skewchar
  \mathsurround\z@ \frozen@everymath{\mathgroup\macc@group\relax}%
  \macc@set@skewchar\relax
  \let\mathaccentV\macc@nested@a
  \if#31
    \macc@nested@a\relax111{#1}%
  \else
    \def\gobble@till@marker##1\endmarker{}%
    \futurelet\first@char\gobble@till@marker#1\endmarker
    \ifcat\noexpand\first@char A\else
      \def\first@char{}%
    \fi
    \macc@nested@a\relax111{\first@char}%
  \fi
  \endgroup
}
\makeatother

\newcounter{Theorem}
\numberwithin{Theorem}{section}

\theoremstyle{plain}
\newtheorem{prop}[Theorem]{Proposition}

\newtheorem{thm}[Theorem]{Theorem}
\newtheorem{lem}[Theorem]{Lemma}
\newtheorem{cor}[Theorem]{Corollary}

\theoremstyle{remark}
\newtheorem{rem}[Theorem]{\bfseries\itshape Remark}
\newtheorem{ques}[Theorem]{\bfseries\itshape Question}

\theoremstyle{definition}
\newtheorem{defn}[Theorem]{Definition}
\newtheorem{exm}[Theorem]{Example}

\DeclareMathOperator{\id}{id}

\newrgbcolor{darkgreen}{0.0 0.39215686 0}
\newrgbcolor{LemonChiffon}{1.0 0.98 0.8}
\newrgbcolor{magenta}{1.0 0 1.0}
\newrgbcolor{mistyrose}{1.0 0.89 0.88}
\newrgbcolor{deeppink}{1.0 0.08 0.58}
\newrgbcolor{lightsalmon}{1.0 0.63 0.48}
\newrgbcolor{lightred}{0.7 0.0 0.0}
\newrgbcolor{lightblue}{0.68 0.85 0.90}
\newrgbcolor{indianred}{0.80  0.36  0.36}
\newrgbcolor{lightgreen}{0.56 0.93 0.56}
\newrgbcolor{darkred}{0.89928057554   0   0}
\newrgbcolor{orange}{1   .5   0}

 \def\co{\colon\thinspace}

\begin{document}

   \title[Mind-Body Duality]{A duality for labeled graphs and
     factorizations with applications to graph embeddings and Hurwitz
     enumeration} \author{Nikos Apostolakis} \address{Bronx Community
     College of The City University of New York}
   \email{nikolaos.apostolakis@bcc.cuny.edu} \date{\today}

\begin{abstract}
  The set of factorizations of permutations in to $m$ transpositions
  of some symmetric group $\mathcal{S}_n$ is naturally in bijection
  with the set of graphs of order $n$ and size $m$ with both edges and
  vertices labeled.  We define a notion of duality (the
  \emph{mind-body duality}) for factorizations and such labeled graphs
  and interpret it in terms of Properly Embedded Graphs, a class of
  graphs embedded in a bounded compact oriented surface with all the
  vertices lying in the boundary, and show a close connection of this
  duality with the Hurwitz action of the Braid Group.  Connections
  with the theory of Cellularly Embedded Graphs are highlighted and
  hints of possible applications are given.  In this paper we focus on
  developing the necessary theory, leaving specific applications and
  further developments for future projects.
\end{abstract}

\keywords{ Properly Embedded Graph ; Factorizations Edge-Labeled
  Graphs ; Mind-Body Duality ; Hurwitz action ; Garside Element of
  Braid Group ; Cellularly Embedded Graph ; Perfect Trail Double Cover
  ; Medial Digraph ; Perfect Chain Decomposition ; Self-Dual
  Embeddings ; Up-Down Numbers}

\thanks{Part of the work in Sections~\ref{sec:duality}
  and~\ref{sec:gy} was developed jointly with Kerry Ojakian. The
  Computer Algebra Systems Sage~\cite{sagemath}, and GAP~\cite{GAP4}
  were used extensively to confirm calculations and check conjectures
  at several stages of this project.  It is my pleasure to thank
  Nichole McDaniel for help with drawing some of the figures.  Finally
  I gratefully acknowledge the partial support of PSCCUNY Research
  Award TRADA-48-526.}

\maketitle
\tableofcontents

\section{Introduction}
\label{sec:intro}

The research in this paper originated with the author
reading~\cite{GouldenYong} and attempting to understand the notion of
duality implicit in the construction of the ``\emph{structural
  bijection}'' defined there.  That duality is closely related to the
``duality''\footnote{The quotation marks are there because this
  ``duality'' is not involutory, a natural requirement in order to
  call any bijection a duality.} for non-crossing trees defined
in~\cite{Hernando1999}, and the structural properties of the bijection
follow from the properties of that duality, namely from the fact that
the neighborhood of a vertex is transformed to a path that, together
with an arc of the boundary circle, forms the boundary of one of the
region that the disk of the non-crossing tree is divided into by the
tree.  An exposition of the contents of~\cite{GouldenYong} from this
point of view is given in Section~\ref{sec:trees}.  A question posed
in~\cite{GouldenYong} is to find generalizations of the structural
bijection, defined there for minimal factorizations of cycles, to more
general classes of factorizations.  We answer this question by making
explicit, and clarifying, the implicit duality, which we call
``Mind-Body Duality'' for reasons explained in
Section~\ref{sec:factorizations}, and generalizing it so that it
applies to factorizations of any permutation with any number of
factors. We leave specific applications to bijective enumerations to
future projects, see Section~\ref{sec:future} for a sample of such
projects planned for the near future.

By \emph{factorization} we mean an expression of a permutation as a
product of a sequence of transpositions. As observed by D\'enes
in~\cite{Denes1959}, factorizations of permutations of a finite set
$V$ with $m$ factors, are in bijection with graphs with vertex set $V$
and $m$ edges labeled by $[m] := \left\{ 1,\ldots,m \right\}$. This
bijection assigns to a factorization a graph that has an edge
connecting $v$ and $u$ labeled by $i$, if and only if and only if, the
$i$-th factor interchanges $v$ and $u$.  The mind body duality is
first defined in Section~\ref{sec:egraphdual} for \emph{edge-labeled
  graphs} using the fact that a labeling of the edges of a graph
induces two dual structures on the graph: a \emph{Local Edge Order}
(\emph{leo}), and a \emph{Perfect Trail Double Cover} (\emph{PTDC}),
see Definitions~\ref{defn:leo} and~\ref{defn:ptdc}. The leo is simply
the linear orders induced by the edge-labels on the set of edges
incident at each vertex, while, in terms of factorizations, the PTDC is
the set of trails formed by the trajectories of $V$ under the
successive application of the factors.  The mind-body dual of an
edge-labeled graph is then defined as the \emph{intersection graph} of
the family of these trajectories viewed as sets of edges, see
Definition~\ref{defn:mbdualgraph}.

The study of factorizations, branched coverings, and their enumeration
is a classical topic that goes back to Hurwitz
(see.~\cite{Hurwitz1902}). The Braid Group, as a group of
automorphisms of the Free Group seems to have appeared for the first
time in that paper as well.  It turns out that Mind-Body Duality has a
simple interpretation in terms of the Hurwitz action on the set of
factorizations: it is simply the reverse of the action of the Garside
element of the Braid Group (see Theorem~\ref{thm:Hurw}), and this
allows us to get explicit formulas for the dual of a
factorization. Furthermore by exploiting a certain ``operadic''
property of the Garside element (see Theorem~\ref{thm:Doperad}) we are
able to calculate the dual of a factorization ``locally'' by relating
the dual of a concatenation of factorizations to the concatenation of
their duals. The mind-body duality is also closely related to the
duality in the braid group defined by changing under-crossings to
over-crossings and vice-versa (see Theorem~\ref{thm:dubr}).

If one interprets the mind-body duality in the context of Topological
Graph Theory, it turns out that it is a generalization of the duality
of graphs embedded in closed oriented surfaces.  In order to define
that generalization we define the notion of a \emph{proper embedding
  of a graph in a surface with boundary} (see
Definition~\ref{defn:peg}). A Properly Embedded Graph (\emph{peg}) is
a graph embedded in a bounded surface in such a way that all the
vertices lie in the boundary, with some technical conditions that
ensure that the dual graph is also properly embedded.  Namely we
require every boundary component to contain at least one vertex, and
that each region to contain exactly one arc in its boundary.  The
prototype for this concept, that in this generality appears to be new,
is a non-crossing tree. The class of Cellularly Embedded Graphs
(\emph{cegs}) is contained in the class of pegs: a graph cellularly
embedded in a closed oriented surface can be construed as a graph
properly embedded in a surface with boundary in such a way that every
boundary component of the surface contains exactly one vertex.  In
this context, leos are the analogue of \emph{rotation systems}, and
PTDCs the analogue of cycle double covers.  A factorization (or a
vertex-and-edge-labeled graph) gives an oriented surface with boundary
with the graph pegged in it, via a straightforward generalization of
the correspondence between rotation systems and cegs. The dual of a
peg (defined in the obvious way, see Definition~\ref{defn:dualpeg}) 
is then isomorphic to the peg obtained by the mind-body dual of
the factorization, and in the special case where the peg can also be
considered a ceg, it coincides with the standard notion of duality for
cegs (see Theorem~\ref{thm:dualpegceg}).

We also give an alternative construction of the peg associated with a
factorization, and of the mind-body duality, via the theory of
branched coverings over the two dimensional disk.  In this context,
the peg corresponding to the factorization is essentially the lifting
of a certain graph in the disk, a so-called \emph{Hurwitz system}, and
its dual is the lifting of a ``dual'' Hurwitz system, see
Theorem~\ref{thm:pegthroubrcov}.  This interpretation of the peg
associated to a factorization is a generalization of a similar
construction for cegs given by Arnold in~\cite{Arnold1996}, however
our interpretation of duality via branched coverings appears to be
new.

We hint at possible applications of this point of view to the theory
of cegs by reproving some baby cases on the existence of self-dual
complete graphs, and giving examples of self-dual embedding of the
complete \emph{digraph} on six vertices, where it is known that no
self-dual embeddings of the corresponding complete (undirected) graph
exist. We further remark that the theory of pegs is a refinement of
the theory of cegs, that is more attuned to the graph theoretic
properties of the graph: whether a graph can be pegged on to a given
surface is not invariant under subdivisions; in fact (see
Proposition~\ref{prop:pegsub}) any ceg admits a subdivision that
renders it the completion of a peg.  We hope that this refinement will
have future applications in explaining known, as well as discovering
new, enumerative coincidences between various classes of Hurwitz
numbers.

Even though in this work pegs are mainly used as a tool for a
topological understanding of the mind-body duality and are not studied
much as a topic on their own right, we do touch upon some natural
questions that arise.  One such natural question is the analogue of
the genus question from the theory of cegs, namely: given a graph
$\Gamma$, what can one say about the Euler characteristic and the
number of boundary components of the pegs that arise from all possible
edge-labelings of $\Gamma$? We consider that question in
Section~\ref{sec:chib} and give a complete answer in the case of
complete graphs, see Theorem~\ref{thm:bsconv}.  The general case is an
interesting open question.

The main tool we develop is the \emph{medial digraph} of an
edge-labeled graph, or more generally a peg.  This is a digraph that
has vertices in bijection with the edges of the peg and there is an
arc from a vertex $a$ to a vertex $b$ if and only if, the edge
corresponding to $b$ immediately follows the edge corresponding to
$a$ in the local edge order of some vertex of the peg, see
Definition~\ref{defn:medialdigaph} and the paragraph following
Definition~\ref{defn:muofpeg}.  This notion is the analogue of
the \emph{medial graph} from the theory of cegs, which in the case of
embeddings in oriented surfaces, also admits a natural digraph
structure coming from the orientation of the surface.  A peg can be
encoded as a \emph{Perfect Chain Decomposition} (\emph{PCD}) of its
medial digraph (see Definition~\ref{defn:binarydg}) and the mind body
duality corresponds to a natural duality on the set of PCDs (see
Theorem~\ref{thm:dualofpcd}).

We use the medial digraph to characterize the class of pegs obtained
by factorizations in Proposition~\ref{prop:leodag}: a peg comes from
a factorization if and only if its medial ditree is Directed Acyclic
Graph (dag).  A peg can be ``completed'' into a ceg by gluing disks
along the boundary components of its surface, and we use the medial
digraph to characterize the class of cegs that are obtained as closures
of pegs: a ceg is the closure of a peg if and only if its medial
digraph admits a Feedback Arc System of cardinality the size of the
ceg (see Proposition~\ref{prop:real}).

In the penultimate section, we examine the case of minimal transitive
factorizations of a cycle, or equivalently edge-and-vertex-labeled
trees, in the light of the developed theory.  We first give an
exposition of the results in~\cite{GouldenYong}, and we show that the
mind-body duality at the level of factorizations can, in this case, be
expressed via the mind-body duality at the level of rooted
edge-labeled trees.  In general, the mind-body dual of a factorization
of a permutation $\pi$ is a factorization of $\pi^{-1}$, but using
rooted edge-labeled trees we show that one can define a duality
between factorizations of the same cycle, enjoying the same structural
properties as mind-body duality. This is a topic that will be further
explored in~\cite{Apostolakis2018a}.  As an application of the medial
digraph (which in this case is a directed tree) we show that the set
of self-dual edge-labeled trees is equinumerous with the set of
alternating permutations (see Corollary~\ref{cor:zigzag}).

In this work we develop the theory of pegs in enough detail to be able
to treat the case of factorizations, postponing the fully developed
theory for a future paper~\cite{Apostolakis2018b}.  This and other
forthcoming future directions of this project are outlined in the
final section.

\subsection{Conventions and Terminology}
\label{sec:conventions}
All graphs we consider are finite. We view graphs as one-dimensional
complexes, with a set of $0$-cells called \emph{vertices} and a set of
$1$-cells called \emph{edges}.  Digraphs are graphs where every edge
has been endowed with an orientation.  We emphasize that every edge,
\emph{including loops}, admits two distinct orientations.  In general
we work with \emph{loopless} graphs, i.e. all $1$-cells are
attached to two distinct $0$-cells, except in Section~\ref{sec:pegs},
where digraphs are allowed to have loops.  The edges of a digraph are
sometimes called \emph{darts} or \emph{arcs} depending on the context.
We use more or less standard graph theoretical terminology, with a few
exceptions; in particular:
\begin{itemize}
\item We use $\Gamma$ to denote a graph, sometimes endowed with extra
  structure.  The set of vertices of a graph is denoted by $V$, and
  its set of edges by $E$.  The \emph{order} of $\Gamma$ (i.e.
  $\left| V \right|$) is typically denoted by $n$, and its \emph{size}
  (i.e. $\left| E \right|$) by $m$.
\item We use $\chi$ to denote the Euler characteristic so that for
  a graph $\Gamma$, we have $\chi(\Gamma) = n -m$.
\item The \emph{neighborhood of a vertex $v$ of $\Gamma$}, denoted $\nu(v)$ is
  the subgraph of $\Gamma$ defined by all the edges that are incident
  to $v$. The \emph{star of a vertex $v$ of $\Gamma$} is the neighborhood
  of $v$ in the first barycentric subdivision of $\Gamma$, the edges of
  the first barycentric subdivision are the \emph{half-edges} of $\Gamma$.
\item A \emph{trail} is a walk without repeated edges.  Every trail has
   a beginning and an ending vertex that may coincide.
 \item By a \emph{cycle} we mean an equivalence class of closed trails,
   where two trails are equivalent if they have the same edges, i.e. the
   endpoint is not important.  In the context of digraphs a cycle is
   always a \emph{directed cycle}.
\end{itemize}

For a finite set $V$ we denote the \emph{symmetric group of $V$} by
$\mathcal{S}_V$, if $V=[n] := \left\{ 1,\ldots,n \right\}$ the
symmetric group of $V$ is denoted by $\mathcal{S}_n$.  We use the
usual terminology e.g. \emph{permutation}, \emph{transposition},
\emph{cycles}, etc even if $V$ is not $[n]$.  For a transposition
$\tau = (s,t)$ we call $s,t$ the moved points of $\tau$.  We multiply
permutations from left to right, so that $(1\,2)(2\,3) = (3\,2\,1)$.
We use left and right exponential notation for conjugation in a group,
i.e.  $g^{h} := h^{-1} g h$ and $\prescript{h}{}g = h g h^{-1}$.

\section{Mind-Body Duality}
\label{sec:duality}

In this section we define the \emph{mind-body duality} first in a
graph theoretical context and then in the more algebraic context of
factorizations.  We start by defining the main objects of study and
their basic equivalence.

\begin{defn}
  \label{defn:fact}
  A \emph{factorization in} $\mathcal{S}_V$ is a sequence of
  transpositions $\rho = \left( \tau_i \right)_{1\le i \le m}$, with
  $\tau_i \in \mathcal{S}_V$.  The product of $\rho$ is called its
  \emph{total monodromy} or simply its \emph{monodromy} and denoted by
  $\mu(\rho)$.  We also say that $\rho$ is a factorization \emph{of}
  $\mu(\rho)$, and sometimes we'll call $\tau_i$ the $i$-th monodromy
  of $\rho$\footnote{This terminology comes from the theory of
    \emph{branched coverings}, see  Subsection~\ref{sec:brcov}}.

  The \emph{reverse} of a factorization $\rho = (\tau_i)_{1\le i \le m}$
  is the factorization $\rho^{\intercal} := (\tau_{m+1-1})_{1\le i \le m}$.

  The \emph{concatenation} of two factorization
  $\rho_1 = \left( \tau_i \right)_{1\le i \le m_1}$ and
  $ \rho_{2} = \left( \tau_i' \right)_{1\le i \le m_2}$ is the
  factorization
  $\left( \tau_1, \ldots, \tau_{m_1}, \tau_1',\ldots,\tau_{m_2}'
  \right)$.
  The concatenation of two factorizations $\rho_1, \rho_2$ will be
  denoted by $\rho_1\rho_2$.

  For a factorization $\rho = (\tau_i)$ and an element
  $\tau \in \mathcal{S}_V$, we use the notation $\rho^{\tau}$
  (resp. $\prescript{\tau}{}\rho$) for the factorization $(\tau_i^{\tau})$
  (resp. $(\prescript{\tau}{}\tau_i)$).

  An \emph{edge-labeled graph} (\emph{e-graph} for short) is a graph
  with edges labeled with elements of $[m]$ where $m$ is the order of
  the graph, or equivalently a graph with a total order in the set of
  its edges. A \emph{vertex-labeled graph} (or \emph{v-graph} for
  short) is a graph whose vertices are labeled by $[n]$ where $n$
  is its order. An \emph{edge-and-vertex-labeled graph} (or
  \emph{e-v-graph} for short) has both vertices and edges labeled.

  The \emph{reverse} of an e-graph $\Gamma$, is the graph
  $\Gamma^{\intercal}$, with the same vertices and edges as $\Gamma$
  and its edges relabeled according to $i \mapsto m+1-i$.

  The \emph{concatenation} of two e-graphs $\Gamma_1$ and $\Gamma_2$
  of sizes $m_1$ and $m_2$, respectively, is the graph
  $\Gamma_1 \Gamma_2$ with
  $V(\Gamma_1\Gamma_2) = V(\Gamma_1) \cup V(\Gamma_2)$ and
  $E(\Gamma_1\Gamma_2)$ consisting of the edges of $\Gamma_1$ with
  their labels unchanged, and the edges of $\Gamma_2$ with their
  labels increased by $m_1$.

  There is an obvious bijection between the set of factorizations of
  $\mathcal{S_{V}}$ and the set of e-graphs with vertex set $V$: for a
  factorization $\rho$ define the \emph{associated e-graph of $\rho$}
  to be the graph $\Gamma(\rho)$ that has an edge with endpoints $u,v$,
  and labeled by $i$, if and only if the $i$-th monodromy of $\rho$ is $(u,v)$;
  conversely for an e-graph $\Gamma$ the \emph{associated
    factorization of $\Gamma$}, $\rho\left( \Gamma \right)$ has the
  $i$-th monodromy interchanging the endpoints of the edge labeled
  $i$.  This bijection specializes to a bijection between factorization
  of $\mathcal{S}_n$ of length $m$ and e-v-graphs or order $n$ and size $m$.
\end{defn}

Clearly the bijection $\rho \mapsto \Gamma(\rho)$ preserves the notions of
reverse and concatenation.

\begin{exm} \label{exm:trans_seq} The sequence
  $\rho = (3\,4), (1\,3), (1\,2), (3\,4), (2\,3)$ is a factorization
  of the cycle $(4\, 3\, 2\, 1)$ in $\mathcal{S}_4$.  The associated
  graph $\Gamma(\rho)$ is shown in Figure~\ref{fig:grassoc}. We draw
  two versions of it, a standard planar drawing in the left, and one
  that the cyclic order of the edges at every vertex is consistent with
  the order induced by their labels, see Definition~\ref{defn:leo}, on
  the right.  The colors of the vertices are used later, see
  Subsection~\ref{sec:medial}.
\end{exm}

\begin{figure}[htbp]
  \centering
  \begin{pspicture}(-4.5,-2)(8.5,2.6)
    \rput(-1.5,0){%
      \psset{unit=.8}
      \begin{pspicture}(-.3,-.5)(6.5,4.5)
        \rput(0,4){\rnode{n1}{\psdot[dotscale=1.2,linecolor=blue](0,0)}}
        \uput[90](0,4){$1$}
        \rput(0,0){\rnode{n2}{\psdot[dotscale=1.2,linecolor=red](0,0)}}
        \uput[-90](0,0){$2$}
        \rput(3.5,2){\rnode{n3}{\psdot[dotscale=1.2,linecolor=green](0,0)}}
        \uput[90](3.5,2){$3$}
        \rput(6,2){\rnode{n4}{\psdot[dotscale=1.2,linecolor=orange](0,0)}}
        \uput[0](6,2){$4$}
        \ncline{n1}{n2}
        \nbput{\raisebox{.5pt}{\textcircled{\raisebox{-.9pt} {\small $3$}}}}
        \ncline{n1}{n3}
        \naput{\raisebox{.5pt}{\textcircled{\raisebox{-.9pt} {\small $2$}}}}
        \ncline{n2}{n3}
        \nbput{\raisebox{.5pt}{\textcircled{\raisebox{-.9pt} {\small $5$}}}}
        \ncarc[arcangle=40]{n3}{n4}
        \naput{\raisebox{.5pt}{\textcircled{\raisebox{-.9pt} {\small $1$}}}}
        \ncarc[arcangle=40]{n4}{n3}
        \naput{\raisebox{.5pt}{\textcircled{\raisebox{-.9pt} {\small $4$}}}}
      \end{pspicture}}
    \rput(6,.5){%
      \psset{unit=.8}
      \begin{pspicture}(-2,2)(5,7.2)
        \psset{linewidth=.05,dotsize=.25}
        \psbezier(1,5)(1.5,3)(3.5,3)(4,5)
        \psbezier[border=2.5pt](1,5)(4,5.5)(3,2.5)(1,2.5)
        \psbezier(1,5)(1.5,7)(3.5,7)(4,5)
        \psbezier(1,2.5)(-2,2)(-2,3)(-1.5,4)
        \psbezier(-1.5,4)(-2,5)(0,5.5)(1,5)
        \psdot[linecolor=orange](4,5)
        \uput[0](4,5){$4$}
        \psdot[linecolor=green](1,5)
        \uput[-135](1,5){$3$}
        \psdot[linecolor=red](1,2.5)
        \uput[-90](1,2.5){$2$}
        \psdot[linecolor=blue](-1.5,4)
        \uput[0](-1.5,4){$1$}
        \uput[90](2.5,6.6){\raisebox{.5pt}{\textcircled{\raisebox{-.9pt} {\small $1$}}}}
        \uput[90](-1,5.1){\raisebox{.5pt}{\textcircled{\raisebox{-.9pt} {\small $2$}}}}
        \uput[-90](-1.5,2.4){\raisebox{.5pt}{\textcircled{\raisebox{-.9pt} {\small $3$}}}}
        \uput[-45](2.2,2.8){\raisebox{.5pt}{\textcircled{\raisebox{-.9pt} {\small $5$}}}}
        \uput[0](3.7,4){\raisebox{.5pt}{\textcircled{\raisebox{-.9pt} {\small $4$}}}}
      \end{pspicture}}
  \end{pspicture}
  \caption{The graph associated with the factorization $\rho$ of Example~\ref{exm:trans_seq}.}
  \label{fig:grassoc}
\end{figure}
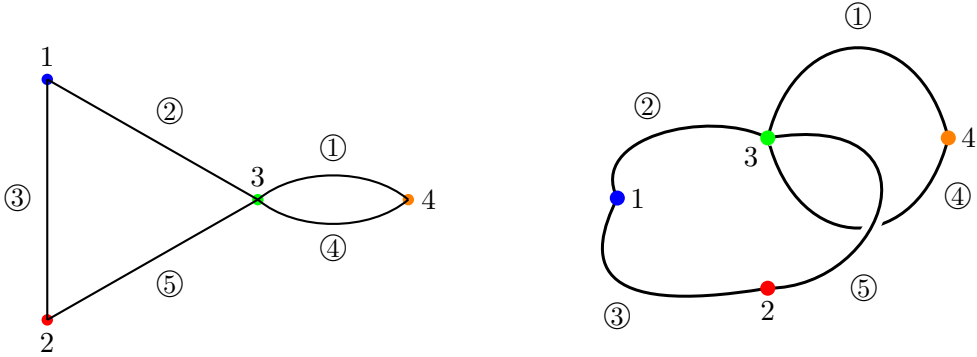

\subsection{Mind-Body duality for e-graphs}
\label{sec:egraphdual}

Let $\Gamma$ be an e-graph.  The edge labels induce two dual structures on $\Gamma$,
a  \emph{Local Edge Ordering} and a \emph{Perfect Trail Double  Cover}.

\begin{defn}
  \label{defn:leo} A \emph{Local Edge Ordering} (\emph{leo} for short)
  of a graph $\Gamma$ is an assignment of a linear order at the
  neighborhood of each vertex of $\Gamma$. We draw leos in such a way
  that the cyclic order induced by the standard (counterclockwise)
  orientation of the plane is consistent with the ordering of the
  edges at every vertex.
\end{defn}

The edge labels of an e-graph induce a leo in the obvious way: the
order of the edges around a vertex is given by the natural order of
their labels.

A leo on a graph $\Gamma$ induces a decomposition of the darts of
$\Gamma$ into chains. We define the relevant structure in more
generality than is strictly needed for this paper in view of future
planned work, see~\cite{Apostolakis2018b}.

\begin{defn}
  \label{defn:ptdc} A \emph{Perfect Trail Double Cover} (PTDC) of a
  graph $\Gamma$ is a family $\mathcal{T}$ of positive length trails
  such that each edge of $\Gamma$ belongs to exactly two trails of
  $\mathcal{T}$, and each vertex is the endpoint of two trails.  We
  emphasize that trivial paths of length zero are not allowed, but
  \emph{closed trails are allowed}; for a closed trail its endpoint is
  counted twice.

  A PTDC is called \emph{orientable} if each trail can be endowed with
  an orientation such that every edge is traversed once in each
  direction, in other words, $\mathcal{T}$ induces a decomposition of
  the darts of $\Gamma$ into chains; a PTDC endowed with such a choice
  of orientations is called \emph{oriented}.  In that case every
  vertex is the start of exactly one trail (we denote that trail by
  $\overrightarrow{v}$) and the end of exactly one trail (we denote
  that trail by $\overleftarrow{v}$).

  A PTDC is called \emph{non-singular} if for every interior (i.e. non-leaf)
  vertex the first edge of $\overrightarrow{v}$ and the last edge of
  $\overleftarrow{v}$ are distinct.
\end{defn}

For the most part of this paper we will be dealing with oriented
non-singular PTDCs, and from now on, barring explicit mention to the
contrary, \emph{we will use PTDC to mean an oriented non-singular
  PTDC}.

Given a leo on $\Gamma$ define the \emph{Minimally Increasing Greedy
  Trail (migt) starting at $v$} to be the trail $\overrightarrow{v}$
obtained as follows: the first edge of $\overrightarrow{v}$ is the
smallest (with respect to the leo) edge in the neighborhood
$\nu(v)$.  We proceed inductively: once we have added an edge $e$ from
$\nu(u)$ to the trail, in the next step we add the smaller edge in the
neighborhood of $w$, the other vertex of $e$, that is larger than $e$
in the leo of $w$, provided that such edge exists.  We stop when $e$
is the last edge at the leo of $w$.\footnote{Recall that our graphs
  are loopless.  Loops could be treated by considering half-edges but
  such generality is not needed in this paper.}

As expected from the notation we have:
\begin{lem}
  \label{lem:leoeqptdc} The migts of a leo give a (non-singular,
  oriented) PTDC. Conversely, a (non-singular, oriented) PTDC
  $\mathcal{T}$ gives a leo on its underlying graph, whose migts are
  the trails of $\mathcal{T}$.
\end{lem}

\begin{proof}
  A vertex $v$ of a graph $\Gamma$ endowed with a leo is the endpoint
  of exactly two migts: $\overrightarrow{v}$ and $\overleftarrow{v}$.
  Now if the first edge $e$ of $\overrightarrow{v}$ is the same as the
  last edge of $\overleftarrow{v}$, then $e$ is both the first and the
  last edge in the leo of $v$ and therefore is the only edge incident
  to $v$. So for an interior vertex $v$, the migts $\overleftarrow{v}$
  and $\overrightarrow{v}$ are distinct.

  Let $e$ be an edge of $\Gamma$ with endpoints $v,u$, and let $e'$
  (resp. $e''$) be the edge of $\Gamma$ incident to $v$ (resp. $u$)
  and immediately preceding $e$ in the leo at $v$ (resp. $u$), if such
  an edge exists.  Then by the definition of migts, $e$ belongs to two
  migts, $m_1 = \ldots, e',e, \ldots$ that transverses $e$ from $v$ to
  $u$, and $m_2 = \ldots, e'',e, \ldots$ that transverses $e$ from $u$
  to $v$; of course if $e'$ (resp. $e''$) does not exist then $m_1$
  (resp. $m_2$) is simply $\overleftarrow{v}$
  (resp. $\overleftarrow{u}$). If a migt does not pass through $v$ or
  $u$ it clearly can't contain $e$, and the other migts that pass
  through $v$ or $u$ do not contains $e$ because migts are minimally
  increasing. So every edge belongs to exactly two migts that
  transverse it in opposite orientations.

  The above can be summarized by saying that the local configuration
  of the migts in a neighborhood of a vertex $v$ is as in
  Figure~\ref{fig:localmigts}: we have a vertex of degree $4$, there
  are four edges incident at a vertex with their order in the leo is
  indicated by the subscripts, and there are five migts that pass
  trough $v$.

  Conversely, the trails of a non-singular oriented PTDC $\mathcal{T}$
  that go through a vertex $u$ are as in Figure~\ref{fig:localmigts}.
  A leo at $v$ can then be defined by the rule that an edge $e$ is
  less than an edge $e'$ if $e$ precedes $e'$ in some trail.  Clearly
  the migts of that leo are exactly the trails of $\mathcal{T}$.
\end{proof}

For example the PTDC induced by the e-v-labeled graph in
Figure~\ref{fig:grassoc} is shown in Figure~\ref{fig:migts}.

\begin{figure}[htbp]
  \centering
  \begin{pspicture}(-3,-3)(3,3)
    \rput(0,0){
      \includegraphics[scale=0.45]{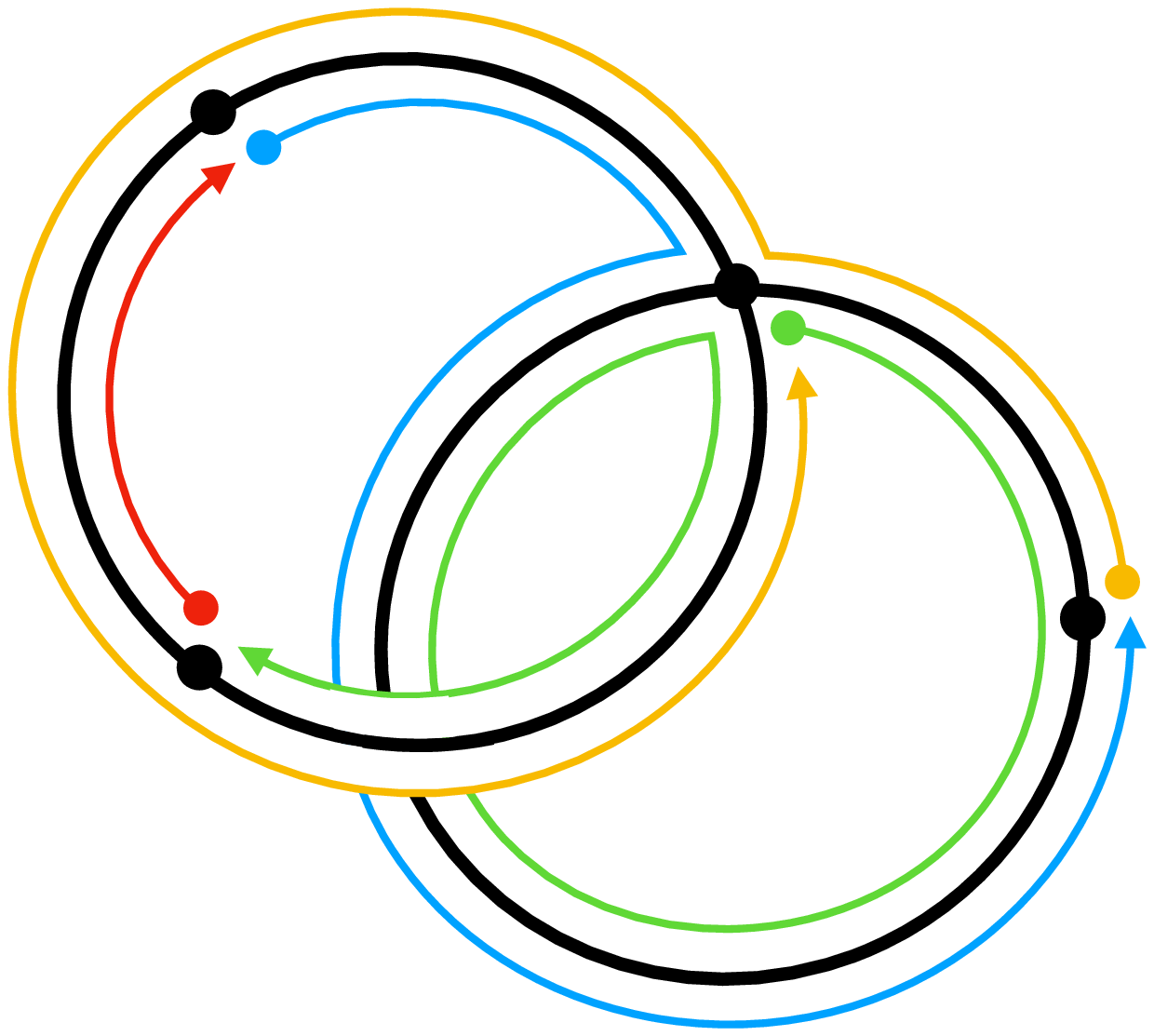}}
    \uput[90](-2,2.2){\blue $1$}
    \uput[0](2.8,-.4){\orange $4$}
    \uput[45](1,1.4){\green $3$}
    \uput[-135](-2,-1){\red $2$}
    \uput[90](2.1,.9){\raisebox{.5pt}{\textcircled{\raisebox{-.9pt} {\small $1$}}}}
    \uput[90](-.5,2.5){\raisebox{.5pt}{\textcircled{\raisebox{-.9pt} {\small $2$}}}}
    \uput[180](-2.8,.6){\raisebox{.5pt}{\textcircled{\raisebox{-.9pt} {\small $3$}}}}
    \uput[-90](.1,-2.5){\raisebox{.5pt}{\textcircled{\raisebox{-.9pt} {\small $5$}}}}
    \uput[-90](1,-.6){\raisebox{.5pt}{\textcircled{\raisebox{-.9pt} {\small $4$}}}}
    \psdot[linecolor=blue,dotsize=.25](-1.8,2.02)
    \psdot[linecolor=red,dotsize=.25](-1.88,-.7)
    \psdot[linecolor=green,dotsize=.25](.8,1.2)
    \psdot[linecolor=orange,dotsize=.25](2.48,-.45)
  \end{pspicture}
  \caption{The graph from Figure~\ref{fig:grassoc} with its migts}
  \label{fig:migts}
\end{figure}

If we start with a factorization $\rho$ its monodromy can be recovered
from the migts of the associated e-v-labeled graph.  Recall, that all
PTDCs of graphs are assumed non-singular and orientable.

\begin{defn}
  \label{defn:ptdcmu} The \emph{monodromy digraph} of a PTDC is the
  digraph that has the same vertices as $\Gamma$ and for each trail an
  arc from its beginning to its end, in other words there is an arc
  from $v$ to $u$ if and only if
  $\overleftarrow{v} = \overrightarrow{u}$.
\end{defn}

\begin{prop}
  \label{prop:graphmu} The monodromy digraph of a PTDC is a functional
  digraph of a permutation in $\mathcal{S}_{V}$.  If the PTDC comes
  from (the associated e-graph of) a factorization then that
  permutation is the monodromy of the factorization.
\end{prop}

\begin{proof}
  The bidegree of every vertex of $\mu \left( \Gamma \right)$ is
  $(1,1)$ because each vertex has exactly one incoming and one
  outgoing trail, so $\mu \left( \Gamma \right)$ is the functional
  digraph of a permutation.

  To prove the second statement we need to prove that if
  $\overrightarrow{u} = \overleftarrow{v}$ then $\mu(u) = v$.  In fact
  it's easy to see that the vertices of $\overrightarrow{u}$ form the
  trajectory of $u$ under the successive applications of the elements
  of $\rho$.  For, let $e_1$ be the first edge of $\overleftarrow{u}$
  and $v_1$ its other endpoint, then $e_1$ is the first edge of
  $\nu(u)$ and therefore $(u\,v_1)$ is the first transposition in
  $\rho$ that moves $u$.  The next time $u$ is moved, is when there is
  a monodromy $(v_1\,v_2)$ in $\rho$ with index larger than the index
  of $e_1$, and this monodromy will correspond to the second edge
  $e_2$ of $\overrightarrow{u}$, and so on until will reach the last
  edge of $\overrightarrow{u} = \overleftarrow{v}$, which is the last
  edge of $\nu(v)$, and no further monodromies move $v$.  It follows then
  that in the product $\mu$ of $\rho$ we have $\mu(u) = v$.
\end{proof}

The above can be verified in Figure~\ref{fig:migts},
$\mu\left(\rho\right) = (4\,3\,2\,1)$, and indeed
$\overrightarrow{4} = \overleftarrow{3}$,
$\overrightarrow{3} = \overleftarrow{2}$,
$\overrightarrow{2} = \overleftarrow{1}$, and
$\overrightarrow{1} = \overleftarrow{4}$.

The terminology PTDC and the monodromy digraph where inspired
by~\cite{Bondy1990}.

Now we can define the \emph{mind-body dual} of an e-labeled graph.
\begin{defn}
  \label{defn:mbdualgraph}
  Let $\Gamma$ be an e-graph.  The \emph{mind-body dual} e-graph is
  the graph $\Gamma^{*}$ that has vertices the migts of $\Gamma$ and
  an edge labeled $i$ connecting two trails $t_{1}$ and $t_2$ if and
  only if the two trails share the edge of $\Gamma$ labeled $i$.
\end{defn}

There is a one-to-one correspondence between the edges of $\Gamma$ and
$\Gamma^{*}$, and corresponding edges have the same label, when needed
we will denote the edge of $\Gamma^{*}$ corresponding to the edge $e$
of $\Gamma$ by $e^{*}$.  Since there are as many migts as vertices,
$\Gamma^{*}$ has as many vertices as $\Gamma$ and there are two
canonical ways to set up a correspondence between the vertices of
$\Gamma$ and the vertices of $\Gamma^{*}$: we can choose $v^{*}$ to be
$\overrightarrow{v}$ or $\overleftarrow{v}$.  We choose the former,
i.e. $v^{*} = \overrightarrow{v}$ but not much of what follows depends
on that choice.  We will comment when the choice makes a difference,
see Remark~\ref{rem:convention} and Theorem~\ref{thm:Hurwbar}.

Notice that if $\Gamma$ is e-v-labeled, $\Gamma^{*}$ is also
e-v-labeled via the correspondence $v \mapsto v^{*}$.  When no
confusion is likely we will abuse language by talking as if $\Gamma$
and $\Gamma^{*}$ have the same vertices and edges.

For example the mind-body dual of the e-v-graph of
Figure~\ref{fig:grassoc} together with its migts is shown in
Figure~\ref{fig:dualgrassoc}.

\begin{figure}[htbp]
  \centering
  \begin{pspicture}(-3.3,-2.3)(3.3,2.3)
    \rput(0,0){\includegraphics[scale=.4]{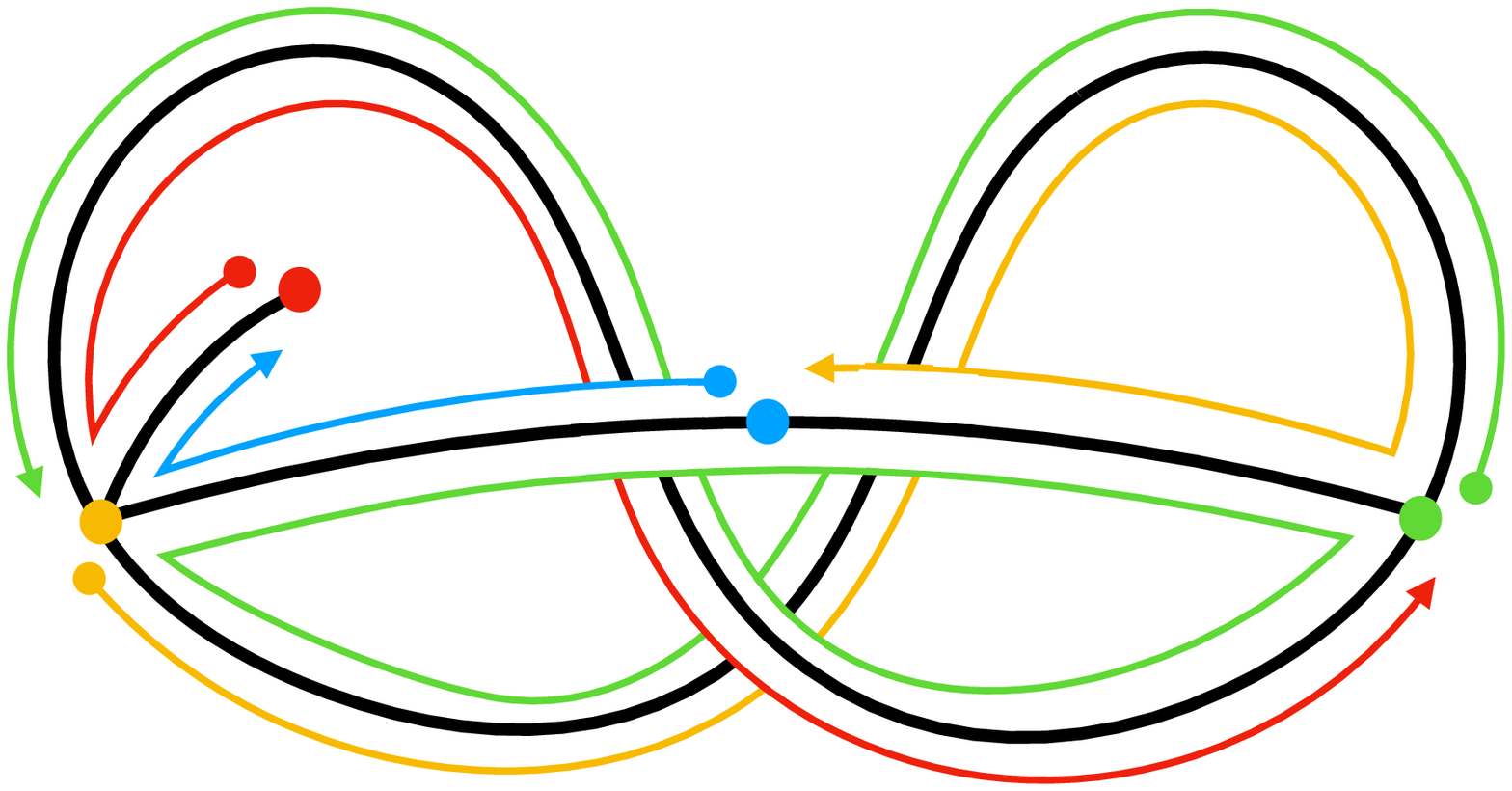}}
    \uput[90](0,0){\blue $1$}
    \uput[90](-2,.6){\red $2$}
    \uput[0](3.1,-.6){\green $3$}
    \uput[180](-3.2,-.6){\orange $4$}
    \uput[-90](-1.5,-.4){\raisebox{.5pt}{\textcircled{\raisebox{-.9pt} {\small $2$}}}}
    \uput[-90](1.5,-.4){\raisebox{.5pt}{\textcircled{\raisebox{-.9pt} {\small $5$}}}}
    \uput[90](2,1.9){\raisebox{.5pt}{\textcircled{\raisebox{-.9pt} {\small $1$}}}}
    \uput[90](-2,1.9){\raisebox{.5pt}{\textcircled{\raisebox{-.9pt} {\small $4$}}}}
    \uput[90](-2,-.1){\raisebox{.5pt}{\textcircled{\raisebox{-.9pt} {\small $4$}}}}
  \end{pspicture}
  \caption{The mind-body dual of the graph in Figure~\ref{fig:grassoc}}
  \label{fig:dualgrassoc}
\end{figure}

\begin{thm}
  \label{thm:dualproperties}
  The following hold:
  \begin{enumerate}
  \item \label{item:startrail} If $v$ is a vertex of $\Gamma$ then the
    neighborhood of $v^{*}$ in $\Gamma^{*}$ consists of (the duals of)
    the edges of $\overrightarrow{v}$.  The migt $\overrightarrow{v^{*}}$
    consists of (the duals of) the edges of $\nu(v)$.
  \item \label{item:involutory}
$\left( \Gamma^{*} \right)^{*} = \Gamma$
  \item \label{item:mu}
    $\mu(\Gamma^{*}) = \mu(\Gamma)^{-1}$
  \end{enumerate}
\end{thm}

\begin{proof}
  The arguments will be easier to follow if the reader refers to
  Figure~\ref{fig:localmigts}.

\begin{figure}[htbp]
  \centering
  \begin{pspicture}(-3,-3)(3,3)
       \rput(0,0){\includegraphics[scale=.75]{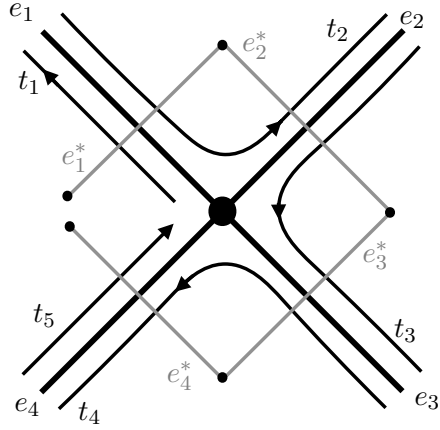}}
       \rput(-1.6,-2.7){$t_4$}
       \rput(2.6,-1.6){$t_3$}
       \rput(1.7,2.4){$t_2$}
       \rput(-2.4,1.7){$t_1$}
       \rput(-2.2,-1.4){$t_5$}
       \rput(-1.8,.7){\gray $e_1^{*}$}
       \rput(.6,2.2){\gray $e_2^{*}$}
       \rput(2.2,-.6){\gray $e_3^{*}$}
       \rput(-.4,-2.2){\gray $e_4^{*}$}
       \uput[135](-2.2,2.4){ $e_1$}
       \uput[45](2.3,2.3){ $e_2$}
       \uput[-45](2.5,-2.3){ $e_3$}
       \uput[125](-2.2,-2.9){ $e_4$}
  \end{pspicture}
  \caption{The local structure of migts}
  \label{fig:localmigts}
\end{figure}

The first statement of Item~\ref{item:startrail} is obvious. To see
the second let $e_1,\ldots,e_d$ be the edges in $\nu(v)$ in the local
ordering, then $e_i$ and $e_{i+1}$ are in some trail $t_{i}$ and thus
$e_i^{*}$ connects $t_i$ and $t_{i+1}$, and $e_{i+1}^{*}$ is the
smallest edge in $t_{i+1}$ greater than $e_i^{*}$.  So in the migt of
$v^{*}$, $e_{i+1}^{*}$ follows $e_{i}^{*}$.

Item~\ref{item:involutory} follows from Item~\ref{item:startrail}.

The proof of Item~\ref{item:mu} is illustrated in
Figure~\ref{fig:mustar}: let $\mu = \mu(\Gamma)$ and
$\mu^{*} = \mu(\Gamma^{*})$, and assume that $m(v) = u$, we need to
show that the migt of $\overrightarrow{u}$ in $\Gamma^{*}$ ends in
$\overrightarrow{v}$.  The last edge of $\overrightarrow{v}$ is the
last edge of $\nu(u)$, Now in $\Gamma^{*}$ the migt of $u^{*}$
consists of the edges dual to the edges in $\nu(u)$ so the last edge
of this migt ends in $\overrightarrow{v}$.

\begin{figure}[htbp]
  \centering
  \begin{pspicture}(-3,-1)(3,1.6)
       \rput(0,0){\includegraphics[scale=.3]{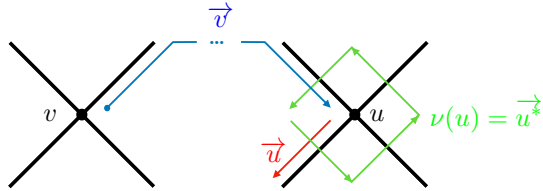}}
       \uput[180](-2,0){\small $v$}
       \uput[0](1.8,0){\small $u$}
       \uput[90](0,1){\blue \small $\overrightarrow{v}$}
       \uput[90](.7,-.8){\red \small $\overrightarrow{u}$}
       \uput[0](2.6,0){\green \small $\nu(u) = \overrightarrow{u^{*}}$}
  \end{pspicture}
  \caption{$\mu(\Gamma^{*}) = \mu(\Gamma)^{-1}$}
  \label{fig:mustar}
\end{figure}

\end{proof}

We end this subsection by noticing that one could more generally
define the dual of a leo, and mutatis mutandis, almost all of the
above would go through.  This will indeed be done
in~\cite{Apostolakis2018b}.

\subsection{Medial digraphs}
\label{sec:medial}

If we put together all the Hasse diagrams of the edge orders of a leo
on $\Gamma$ we obtain the \emph{medial digraph} of the leo. For
general definitions and terminology on digraphs we refer the reader
to~\cite{bang2002diagraphs}.

\begin{defn}
  \label{defn:medialdigaph} The \emph{medial digraph} of of a leo on
  $\Gamma$ is the digraph $\mathcal{M}\left( \Gamma \right)$ with
  vertices the edges of $\Gamma$ and an arc from edge $a$ to edge $b$
  if and only if $a$ immediately proceeds $b$ in the local order
  around a vertex.  A leo is called \emph{e-realizable} if its medial
  digraph is a \emph{dag}, in other words it has no (oriented) cycles.
\end{defn}

For example the medial digraph $\mathcal{M}$ of the e-v-graph in
Figure~\ref{fig:grassoc} is shown in the left side of
Figure~\ref{fig:daggrassoc}, the edges of $\mathcal{M}$ are colored
according to the vertex of $\Gamma$ that they come from.  The right
side of Figure~\ref{fig:daggrassoc} shows the medial digraph of
$\Gamma^{*}$, notice that, edge colors aside, the two digraphs
coincide.  In general, it follows from Item~\ref{item:startrail} of
Theorem~\ref{thm:dualproperties} that

\begin{thm}
  \label{thm:dualmedial} For any factorization $\rho$ we have
    $$\mathcal{M} \left( \Gamma(\rho)^{*} \right) = \mathcal{M}\left( \Gamma(\rho) \right)$$
\end{thm}


\begin{figure}[htbp]
  \centering
  \psset{unit=1.5,arrowsize=.15}
  \begin{pspicture}(1,-.3)(6,3.2)
    \rput(2,1.6){%
  \begin{pspicture}(0,0)(1,3.2)
    \rput(1,0){\circlenode{1}{\tiny $1$}}
    \rput(.5,1){\circlenode{2}{\tiny $2$}}
    \rput(0,2){\circlenode{3}{\tiny $3$}}
    \rput(1,2){\circlenode{4}{\tiny $4$}}
    \rput(.5,3){\circlenode{5}{\tiny $5$}}
    \ncline[linecolor=green]{->}{1}{2}
    \ncline[linecolor=blue]{->}{2}{3}
    \ncline[linecolor=red]{->}{3}{5}
    \ncline[linecolor=orange]{->}{1}{4}
    \ncline[linecolor=green]{->}{4}{5}
    \ncline[linecolor=green]{->}{2}{4}
  \end{pspicture}}
\rput(5,1.6){%
   \begin{pspicture}(0,0)(1,3.2)
    \rput(1,0){\circlenode{1}{\tiny $1$}}
    \rput(.5,1){\circlenode{2}{\tiny $2$}}
    \rput(0,2){\circlenode{3}{\tiny $3$}}
    \rput(1,2){\circlenode{4}{\tiny $4$}}
    \rput(.5,3){\circlenode{5}{\tiny $5$}}
    \ncline[linecolor=orange]{->}{1}{2}
    \ncline[linecolor=orange]{->}{2}{3}
    \ncline[linecolor=orange]{->}{3}{5}
    \ncline[linecolor=green]{->}{1}{4}
    \ncline[linecolor=green]{->}{4}{5}
    \ncline[linecolor=blue]{->}{2}{4}
  \end{pspicture}}
  \end{pspicture}
  \caption{The medial digraph of the e-v-graph in Figure~\ref{fig:grassoc}.}
  \label{fig:daggrassoc}
\end{figure}
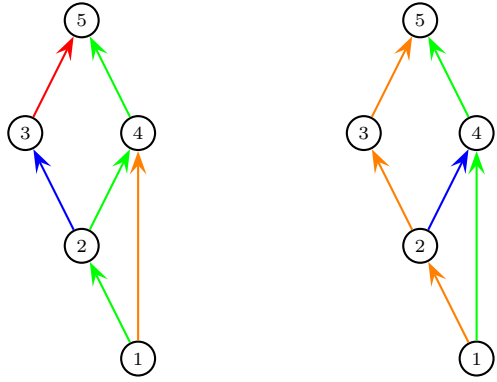

Notice also that the medial digraph in Figure~\ref{fig:daggrassoc} is
a dag, and furthermore the edge labels endow it with a
\emph{topological sort}.  We recall the definition:

\begin{defn}
  \label{defn:topsort} A \emph{topological sort} of a digraph is a total order of its
  vertices such that for two vertices $u,v$, we have that if there is an edge
  from $u$ to $v$ then $u < v$ in that order.  Clearly a digraph admits a topological
  sort if and only if it is a dag.

  A \emph{topsorted dag} is a vertex-labeled dag such that the order of the vertices
  induced by their labeling is topological sort.
\end{defn}

We note the somewhat subtle distinction of the two notions defined
above.  Clearly a topological sort of a dag gives a topsorted dag, but
two different topological sorts of the same dag may give the same
topsorted dag.  For example the updown ditree with three vertices (see
Definition~\ref{defn:zigzag}) admits two topological sorts, but
because of the order two automorphism, there is only one topsorted
dag whose underlying (unlabeled) dag is the updown ditree on three
vertices.

Now we can prove:

\begin{prop}
  \label{prop:leodag} A given leo is induced by an edge labeling of
  $\Gamma$ if and only if it is e-realizable.  Furthermore the edge
  labels induce a topological sort of the medial digraph.
\end{prop}
\begin{proof}
  Clearly the medial digraph of an e-v-labeled tree contains no cycles
  since the local orders come from a global order. Conversely any dag
  admits a topological sort, that is a global order compatible with
  all the local orders thus giving a total order in the edges of
  $\Gamma$.
\end{proof}

\begin{exm}
  \label{exm:notrealizableptdc}
  The left side of Figure~\ref{fig:nonrealizableptdc} show the (PTDC
  of) a non-e-realizable leo on a graph.  Its (obviously non-cyclic)
  medial digraph is shown on the right.
\end{exm}

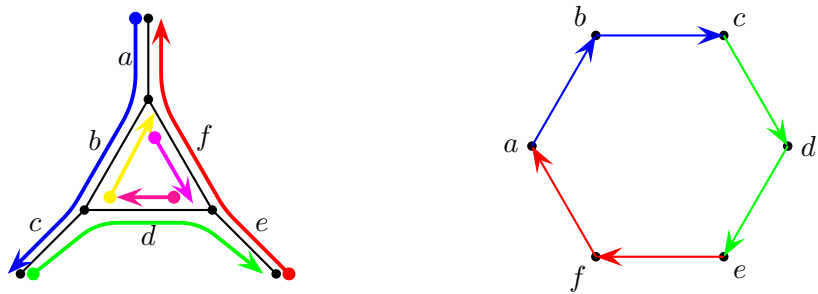
\begin{figure}[htbp]
\begin{center}
{\psset{unit=1.7,arrowsize=.15}
\begin{pspicture}(1.5,-.5)(4,2.1)
\rput(.5,0){%
  \rput(0,0){\rnode{n1}{\psdot(0,0)}}
  \rput(.5,0.866025403784439){\rnode{n2}{\psdot(0,0)}}
  \rput(1,0){\rnode{n3}{\psdot(0,0)}}
  \rput(.5,1.5){\rnode{n4}{\psdot(0,0)}}
  \rput(-.5,-.5){\rnode{n5}{\psdot(0,0)}}
  \rput(1.5,-.5){\rnode{n6}{\psdot(0,0)}}
  \ncline{n1}{n2}
  \naput{$b$}
  \ncline{n2}{n3}
  \naput{$f$}
  \ncline{n1}{n3}
  \nbput{$d$}
  \ncline{n4}{n2}
  \nbput{$a$}
  \ncline{n5}{n1}
  \naput{$c$}
  \ncline{n6}{n3}
  \nbput{$e$}
  \psline[linewidth=.03,linecolor=blue,linearc=.75]{*->}(.4,1.5)(.4,0.866025403784439)(-.1,0)(-.6,-.5)
  \psline[linewidth=.03,linecolor=red,linearc=.75]{<-*}(.6,1.5)(.6,0.866025403784439)(1.1,0)(1.6,-.5)
  \psline[linewidth=.03,linecolor=green,linearc=.45]{*->}(-.4,-.5)(.1,-.1)(.9,-.1)(1.4,-.5)
  \psline[linewidth=.03,linecolor=yellow]{*->}(.2,.1)(.55,0.766025403784439)
  \psline[linewidth=.03,linecolor=magenta]{*->}(.55,0.566025403784439)(.85,.05)
  \psline[linewidth=.03,linecolor=deeppink]{*->}(.7,.1)(.25,.1)}

  \rput(5,.5){\begin{pspicture}(-1.3,-1.3)(1.3,1.3)
    \rput(1.00000000000000, 0.000000000000000){\rnode{d}{\psdot(0,0)}}
    \uput[0](1.00000000000000,0.000000000000000){$d$}
    \rput(0.500000000000000, 0.866025403784439){\rnode{c}{\psdot(0,0)}}
    \uput[45](0.500000000000000, 0.866025403784439){$c$}
    \rput(-0.500000000000000, 0.866025403784439){\rnode{b}{\psdot(0,0)}}
    \uput[135](-0.500000000000000, 0.866025403784439){$b$}
    \rput(-1.00000000000000, 0.000000000000000){\rnode{a}{\psdot(0,0)}}
    \uput[180](-1.00000000000000, 0.000000000000000){$a$}
    \rput(-0.500000000000000, -0.866025403784439){\rnode{f}{\psdot(0,0)}}
    \uput[-135](-0.500000000000000, -0.866025403784439){$f$}
    \rput(0.500000000000000, -0.866025403784439){\rnode{e}{\psdot(0,0) }}
    \uput[-45](0.500000000000000, -0.866025403784439){$e$}
    \ncline[linecolor=blue]{->}{a}{b}
    \ncline[linecolor=blue]{->}{b}{c}
    \ncline[linecolor=green]{->}{c}{d}
    \ncline[linecolor=green]{->}{d}{e}
    \ncline[linecolor=red]{->}{e}{f}
    \ncline[linecolor=red]{->}{f}{a}
  \end{pspicture}}
\end{pspicture}}
\end{center}
\caption{A non-e-realizable PTDC and its medial digraph}
\label{fig:nonrealizableptdc}
\end{figure}

These constructions were inspired by the ideas in section 3
of~\cite{DulPen1993}.  The terminology \emph{medial digraph} is meant
to suggest an analogy with the medial graphs in the theory of graph
embeddings, see for example~\cite{Archdeacon1992}.  This analogy will
be made precise in Section~\ref{sec:pegs}.

Before proceeding we prove a lemma:

\begin{lem}
  \label{lem:leochi} The underlying graph of the medial digraph of a
  leo $\Gamma$ has the same Euler characteristic as $\Gamma$.
\end{lem}
\begin{proof}
  If $\Gamma$ has $m$ edges then $\mathcal{M}$ has $m$ vertices.  Furthermore
  a vertex $v$ of degree $d_v$ contributes $d_{v}-1$ edges.  So $\mathcal{M}$ has
  $\Sigma (d_v - 1) = 2m - n$ edges.  So $\chi(\mathcal{M}) = m - (2m-n) = n -m$.
\end{proof}

Given a leo on $\Gamma$ the local orders of every vertex induce a
decomposition of the edges of $\mathcal{M}(\Gamma)$ into chains.  For
example the graph in Figure~\ref{fig:grassoc} induce a chain
decomposition of its medial digraph that is indicated by the coloring
of the edges.  We formalize this idea in the following definition.

\begin{defn}
  \label{defn:binarydg}
  A digraph $M$ is a \emph{binary digraph} if the in and out degree
  of every vertex is at most $2$.  A vertex of a binary digraph is
  \emph{internal} if its in and out degree is at least $1$.

  A \emph{Perfect Chain Decomposition} (\emph{PCD} for short) of a
  binary digraph $M$ is a decomposition $\mathcal{C}$ of the edges of
  $M$ into chains such that every vertex of $M$ belongs to exactly two
  chains of $\mathcal{C}$, chains of length $0$ are allowed.
\end{defn}

One can now prove:
\begin{lem}
  \label{lem:int}
  The following hold:
  \begin{enumerate}
  \item   \label{item:int1}
    The number of chains in any PCD of a binary digraph $M$ is
    $2m-l$, where $m$ is the number of vertices and $l$ the number of
    edges.
  \item \label{item:int2}
    Given a binary digraph with $\iota$ internal vertices there
    are $2^{\iota}$ PCDs on $M$.  In particular, every medial digraph
    admits a PCD.
  \item \label{item:int3}
    Given any binary \emph{dag} $M$, there is an e-graph whose
    medial digraph is $M$. In fact if $\iota$ stands for the number
    of internal vertices, $\tau$ for the number of topological
    sorts, and $\alpha$ for the number of automorphisms of $M$,
    then, up to isomorphism, there are
    $$ \frac{2^{\iota}\tau}{\alpha} $$
    e-graphs that have $M$ as medial digraph.
  \end{enumerate}
\end{lem}

\begin{proof}
 Item~\ref{item:int1} follows from the Handshaking Lemma:  If there are
$k$ chains $c_1,\ldots, c_k$ with lengths $l_1,\ldots,l_k$, then $c_i$
will have $l_i + 1$ vertices, so $\sum_{i=1}^k(l_1 +1) = 2m$.  On the other
hand $\sum_{i=1}^k(l_1 +1) = l + k$.

For Item~\ref{item:int2} we note that there are $4$ possible bidegrees
for an internal vertex $v$: $(1,1)$, $(1,2)$, $(2,1)$, and $(2,2)$,
and for each of these bidegrees there are two choices for joining the
edges into $2$ chains. If $v$ has bidegree $(1,1)$, one choice is to
join the two edges together for one of the chains and have the other
chain to be the trivial chain $v$, while the other choice is to have
one of the edges in the first chain and the other edge in the second.
If $v$ has bidegree $(1,2)$ with incoming edge $(x,v)$ and outgoing
edges $(v,y)$ and $(v,z)$ the fist choice is to join $(x,v)$ and
$(v,y)$ and have $(v,z)$ by itself, and the other choice is to join
$(x,v)$ and $(v,z)$ and have $(v,y)$ by itself.  The case of bidegree
$(2,1)$ is entirely similar.  Finally for bidegree $(2,2)$ with
incoming edges $(x,v)$ and $(y,v)$, and outgoing edges $(v,z)$ and
$(v,w)$ the one choice is to join $(x,v)$ with $(v,z)$ and $(y,v)$
with $(v,w)$, and the other is to join $(x,v)$ with $(v,w)$ and
$(y,v)$ with $(v,z)$.  (See the first and third columns of
Figure~\ref{fig:pcdduality}.)

For a non-internal vertex $v$ there is only one choice: if the vertex
is a leaf then one of the chains is the trivial chain $v$ and the
other contains the unique edge, while if $v$ is a minimum
(resp. maximum) each of the outgoing (resp. incoming) edges goes in to
a different chain.

Making a choice in each of the vertices gives a PCD, and there are $2^{\iota}$
such choices.

For Item~\ref{item:int3} we note that any PCD $\mathcal{C}$ of a
binary digraph $M$ defines a graph $\Gamma$ with a leo, as follows:
the vertices of $\Gamma$ are the chains of $\mathcal{C}$, and for each
vertex $a$ of $M$ there is an edge in $\Gamma$ joining the vertices of
$\Gamma$ that correspond to the two chains that $a$ belongs
to. Clearly the neighborhood of a vertex $c$ of $G$ consists of the
edges that correspond to the vertices of that chain in $M$, and the
chain defines a total order on that neighborhood.  By definition, the
medial digraph of $G$ is $M$ and the PCD induced by the leo is
$\mathcal{C}$.

If $M$ is a dag, any topological sorting of $M$, gives an e-labeling
for each of the $2^{\iota}$ graphs $\Gamma$ constructed above.  Taking
into account the action of the automorphism group of $M$ gives the
formula for the number of e-graphs that have $M$ as medial digraph.
\end{proof}

The proof of Item~\ref{item:int3} identifies the set of (isomorphism
classes of) e-graphs with medial dag $M$, with the set of (isomorphism
classes of) PCDs of $M$, and ~\ref{item:int2} identifies PCDs of $M$
with a set of binary choices, one choice for each internal vertex.  It
turns out that under these identifications the mind-body dual of an
e-graph is identified with the PCD obtained by making the opposite
choice at every internal vertex. We make this precise below.

For any given binary digraph one could identify the two choices of
connecting arcs at each internal vertex with $0$ and $1$.  This
identification can be done canonically for vertices of bidegree
$(1,1)$, say the choice that connects the two arcs is $0$, and the
choice that doesn't is $1$.  For the other types of internal vertices
an identification has to be made arbitrarily at each vertex.  One way
to accomplish a uniform encoding is to draw $M$ in the plane and then
use the orientation of the plane to say, for example, that for
vertices of bidegree $(1,2)$ (resp $(2,1)$) choice $0$ is to connect
the single incoming (resp. outgoing) edge with the left outgoing
(resp.  incoming) edge, while for vertices of bidegree $(2,2)$ choice
$0$ means to connect the left incoming to the left outgoing edge.
After such an identification has been made, the proof of
Item~\ref{item:int2} constructs a bijection from the set of all PCDs
on $M$ to the set of all function $s\co I \to \{0,1\}$, where $I$ is
the set of internal vertices of $M$.  (See the first and third column
of Figure~\ref{fig:pcdduality})

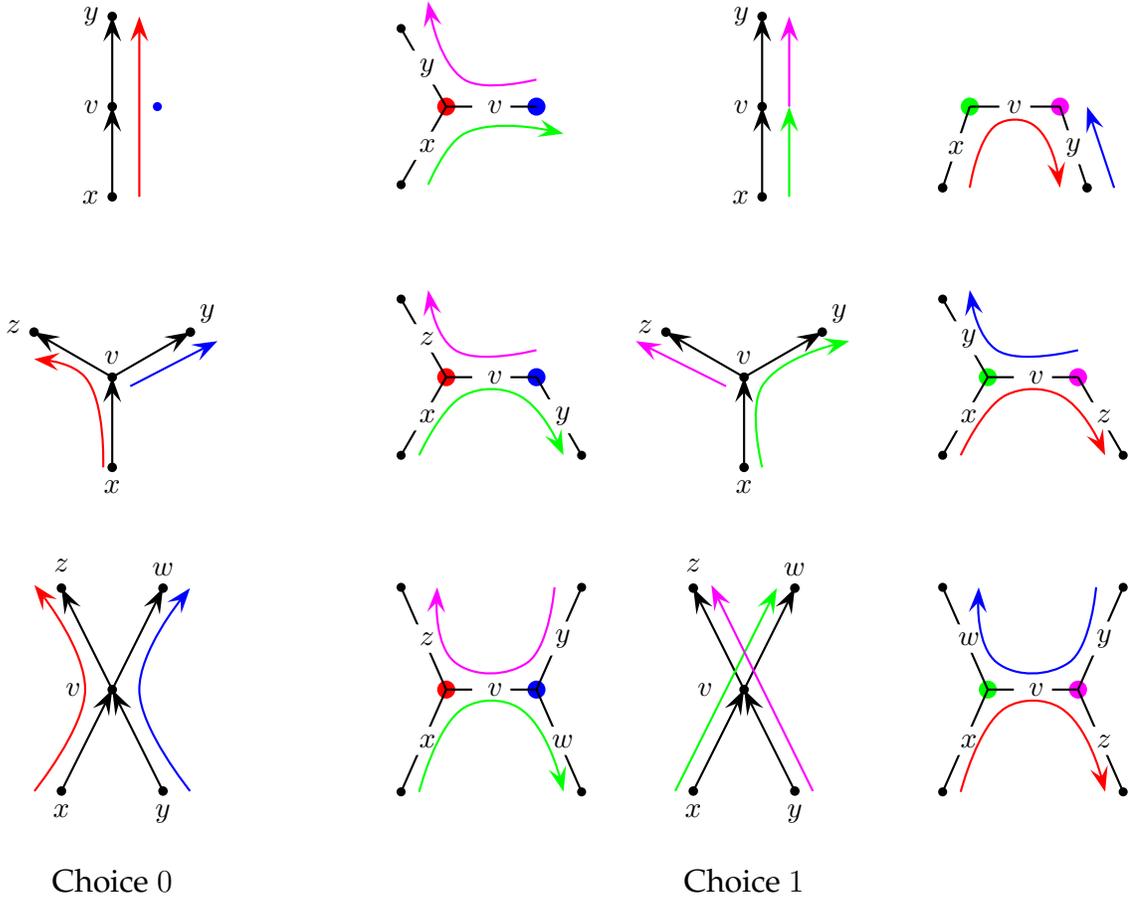
\begin{figure}[htbp]
  \centering
  \psset{unit=1.2,arrowsize=.2}
  \begin{pspicture}(1,-3)(8.7,7)
    \rput(.3,5 ){\rnode{x}{\psdot(0,0)}}
    \uput[180](.3,5 ){$x$}
    \rput(.3,6 ){\rnode{v}{\psdot(0,0)}}
    \uput[180](.3,6 ){$v$}
    \rput(.3,7 ){\rnode{y}{\psdot(0,0)}}
    \uput[180](.3,7 ){$y$}
    \ncline{->}{x}{v}
    \ncline{->}{v}{y}
    \psline[linecolor = red]{->}(.6,5)(.6,7)
    \psdot[linecolor=blue,dotsize=.1](.8,6)
    \rput(4,6){\rnode{a}{\psdot[linecolor=red,dotsize=.2](0,0)}}
    \rput(5,6){\rnode{b}{\psdot[linecolor=blue,dotsize=.2](0,0)}}
    \rput(3.5,6.8660254){\rnode{c}{\psdot[dotsize=.1](0,0)}}
    \rput(3.5,5.1339746){\rnode{d}{\psdot[dotsize=.1](0,0)}}
    \ncline{a}{b}
    \ncput*{ $v$}
    \ncline{a}{c}
    \ncput*{ $y$}
    \ncline{a}{d}
    \ncput*{ $x$}
    \pscurve[linecolor=green]{->}(3.8,5.1339746)(4.2,5.7)(5.3,5.7)
    \pscurve[linecolor=magenta]{->}(5,6.3)(4.2,6.3)(3.8,7.1660254)
    \rput(7.5,5){\rnode{x1}{\psdot(0,0)}}
    \uput[180](7.5,5){$x$}
    \rput(7.5,6){\rnode{v1}{\psdot(0,0)}}
    \uput[180](7.5,6 ){$v$}
    \rput(7.5,7){\rnode{y1}{\psdot(0,0)}}
    \uput[180](7.5,7){$y$}
    \ncline{->}{x1}{v1}
    \ncline{->}{v1}{y1}
    \psline[linecolor = green]{->}(7.8,5)(7.8,6)
    \psline[linecolor = magenta]{->}(7.8,6)(7.8,7)
    \rput(9.5,5.1){\rnode{a1}{\psdot(0,0)}}
    \rput(9.8,6){\rnode{b1}{\psdot[linecolor=green,dotsize=.2](0,0)}}
    \rput(10.8,6){\rnode{c1}{\psdot[linecolor=magenta,dotsize=.2](0,0)}}
    \rput(11.1,5.1){\rnode{d1}{\psdot(0,0)}}
    \ncline{a1}{b1}
    \ncput*{$x$}
    \ncline{b1}{c1}
    \ncput*{$v$}
    \ncline{c1}{d1}
    \ncput*{$y$}
    \pscurve[linecolor=red]{->}(9.8,5.1)(10.1,5.8)(10.5,5.8)(10.8,5.1)
    \psline[linecolor=blue]{->}(11.4,5.1)(11.1,6)
    \rput(.3,3){\rnode{v}{\psdot(0,0)}}
    \uput[90](.3,3){$v$}
    \rput(0.3,2){\rnode{x}{\psdot(0,0)}}
    \uput[-90](0.3,2){$x$}
    \rput(1.16602540378444,3.5){\rnode{y}{\psdot(0,0)}}
    \uput[45](1.16602540378444,3.5){$y$}
    \rput(-0.566025403784438,3.5){\rnode{z}{\psdot(0,0)}}
    \uput[165](-0.566025403784438,3.5){$z$}
    \ncline{->}{x}{v}
    \ncline{->}{v}{y}
    \ncline{->}{v}{z}
    \pscurve[linecolor=red]{->}(.2,2)(0,3)(-0.566025403784438,3.2)
    \psline[linecolor=blue]{->}(.5,2.9)(1.46602540378444,3.4)
    \rput(4,3){\rnode{a3}{\psdot[linecolor=red,dotsize=.2](0,0)}}
    \rput(5,3){\rnode{b3}{\psdot[linecolor=blue,dotsize=.2](0,0)}}
    \rput(3.5,3.8660254){\rnode{c3}{\psdot[dotsize=.1](0,0)}}
    \rput(3.5,2.1339746){\rnode{d3}{\psdot[dotsize=.1](0,0)}}
    \rput(5.5,2.1339746){\rnode{e3}{\psdot[dotsize=.1](0,0)}}
    \ncline{a3}{b3}
    \ncput*{ $v$}
    \ncline{a3}{c3}
    \ncput*{ $z$}
    \ncline{a3}{d3}
    \ncput*{ $x$}
    \ncline{b3}{e3}
    \ncput*{ $y$}
    \pscurve[linecolor=green]{->}(3.7,2.1339746)(4.2,2.8)(4.8,2.8)(5.3,2.1339746)
    \pscurve[linecolor=magenta]{->}(5,3.3)(4.1,3.3)(3.8,3.9660254)
    \rput(7.3,3){\rnode{v}{\psdot(0,0)}}
    \uput[90](7.3,3){$v$}
    \rput(7.3,2){\rnode{x}{\psdot(0,0)}}
    \uput[-90](7.3,2){$x$}
    \rput(8.16602540378444,3.5){\rnode{y}{\psdot(0,0)}}
    \uput[45](8.16602540378444,3.5){$y$}
    \rput(6.43397459621556,3.5){\rnode{z}{\psdot(0,0)}}
    \uput[165](6.43397459621556,3.5){$z$}
    \ncline{->}{x}{v}
    \ncline{->}{v}{y}
    \ncline{->}{v}{z}
    \psline[linecolor=magenta]{->}(7.1,2.9)(6.1,3.4)
    \pscurve[linecolor=green]{->}(7.5,2)(7.5,2.9)(8.46602540378444,3.4)
    \rput(10,3){\rnode{a3}{\psdot[linecolor=green,dotsize=.2](0,0)}}
    \rput(11,3){\rnode{b3}{\psdot[linecolor=magenta,dotsize=.2](0,0)}}
    \rput(9.5,3.8660254){\rnode{c3}{\psdot[dotsize=.1](0,0)}}
    \rput(9.5,2.1339746){\rnode{d3}{\psdot[dotsize=.1](0,0)}}
    \rput(11.5,2.1339746){\rnode{e3}{\psdot[dotsize=.1](0,0)}}
    \ncline{a3}{b3}
    \ncput*{ $v$}
    \ncline{a3}{c3}
    \ncput*{ $y$}
    \ncline{a3}{d3}
    \ncput*{ $x$}
    \ncline{b3}{e3}
    \ncput*{ $z$}
    \pscurve[linecolor=red]{->}(9.7,2.1339746)(10.2,2.8)(10.8,2.8)(11.3,2.1339746)
    \pscurve[linecolor=blue]{->}(11,3.3)(10.1,3.3)(9.8,3.9660254)
    \rput(.3, -0.46375){\rnode{v}{\psdot(0,0)}}
    \uput[180](.1, -0.46375){$v$}
    \rput(-0.2625,-1.58875){\rnode{x}{\psdot(0,0)}}
    \uput[-90](-0.2625,-1.58875){$x$}
    \rput(0.8625,-1.58875){\rnode{y}{\psdot(0,0)}}
    \uput[-90](0.8625,-1.58875){$y$}
    \rput(-0.2625,0.66125){\rnode{z}{\psdot(0,0)}}
    \uput[90](-0.2625,0.696084051562){$z$}
    \rput(0.8625,0.66125){\rnode{w}{\psdot(0,0)}}
    \uput[90](0.8625,0.66125){$w$}
    \ncline{->}{x}{v}
    \ncline{->}{y}{v}
    \ncline{->}{v}{z}
    \ncline{->}{v}{w}
    \pscurve[linecolor=red]{->}(-0.5625,-1.58875)(0, -0.46375)(-0.5625,0.696084051562)
    \pscurve[linecolor=blue]{->}(1.1625,-1.58875)(.6, -0.46375)(1.1625,0.66125)
    \rput(4,-0.46375){\rnode{a3}{\psdot[linecolor=red,dotsize=.2](0,0)}}
    \rput(5,-0.46375){\rnode{b3}{\psdot[linecolor=blue,dotsize=.2](0,0)}}
    \rput(3.5,-1.5977246){\rnode{c3}{\psdot[dotsize=.1](0,0)}}
    \rput(3.5,0.6702246){\rnode{d3}{\psdot[dotsize=.1](0,0)}}
    \rput(5.5,0.6702246){\rnode{e3}{\psdot[dotsize=.1](0,0)}}
    \rput(5.5,-1.5977246){\rnode{f3}{\psdot[dotsize=.1](0,0)}}
    \ncline{a3}{b3}
    \ncput*{ $v$}
    \ncline{a3}{c3}
    \ncput*{ $x$}
    \ncline{a3}{d3}
    \ncput*{ $z$}
    \ncline{b3}{e3}
    \ncput*{ $y$}
    \ncline{b3}{f3}
    \ncput*{ $w$}
    \pscurve[linecolor=green]{->}(3.7,-1.5977246)(4.2,-0.66375)(4.8,-0.66375)(5.3,-1.5977246)
    \pscurve[linecolor=magenta]{->}(5.2,0.6702246)(4.9,-0.16375)(4.1,-0.16375)(3.9,0.6702246)

    \rput(7.3, -0.46375){\rnode{v}{\psdot(0,0)}}
    \uput[180](7.1, -0.46375){$v$}
    \rput(6.7375,-1.58875){\rnode{x}{\psdot(0,0)}}
    \uput[-90](6.7375,-1.58875){$x$}
    \rput(7.8625,-1.58875){\rnode{y}{\psdot(0,0)}}
    \uput[-90](7.8625,-1.58875){$y$}
    \rput(6.7375,0.66125){\rnode{z}{\psdot(0,0)}}
    \uput[90](6.7375,0.696084051562){$z$}
    \rput(7.8625,0.66125){\rnode{w}{\psdot(0,0)}}
    \uput[90](7.8625,0.66125){$w$}
    \ncline{->}{x}{v}
    \ncline{->}{y}{v}
    \ncline{->}{v}{z}
    \ncline{->}{v}{w}
    \psline[linecolor=green]{->}(6.5375,-1.58875)(7.6625,0.66125)
    \psline[linecolor=magenta]{->}(8.0625,-1.58875)(6.9375,0.696084051562)
    \rput(10,-0.46375){\rnode{a3}{\psdot[linecolor=green,dotsize=.2](0,0)}}
    \rput(11,-0.46375){\rnode{b3}{\psdot[linecolor=magenta,dotsize=.2](0,0)}}
    \rput(9.5,-1.5977246){\rnode{c3}{\psdot[dotsize=.1](0,0)}}
    \rput(9.5,0.6702246){\rnode{d3}{\psdot[dotsize=.1](0,0)}}
    \rput(11.5,0.6702246){\rnode{e3}{\psdot[dotsize=.1](0,0)}}
    \rput(11.5,-1.5977246){\rnode{f3}{\psdot[dotsize=.1](0,0)}}
    \ncline{a3}{b3}
    \ncput*{ $v$}
    \ncline{a3}{c3}
    \ncput*{ $x$}
    \ncline{a3}{d3}
    \ncput*{ $w$}
    \ncline{b3}{e3}
    \ncput*{ $y$}
    \ncline{b3}{f3}
    \ncput*{ $z$}
    \pscurve[linecolor=red]{->}(9.7,-1.5977246)(10.2,-0.66375)(10.8,-0.66375)(11.3,-1.5977246)
    \pscurve[linecolor=blue]{->}(11.2,0.6702246)(10.9,-0.16375)(10.1,-0.16375)(9.9,0.6702246)
    \uput[-90](.3,-2.3){\large Choice $0$}
    \uput[-90](7.3,-2.3){\large Choice $1$}
\end{pspicture}
  \caption{From PCDs to e-graphs and duality}
  \label{fig:pcdduality}
\end{figure}

\begin{defn}
  \label{defn:dualofpcd}
  The \emph{dual} of a function $s\co I \to \left\{ 0,1 \right\}$ is
  the function $s^{*}\co I \to \left\{ {0,1} \right\}$ defined by
  $s^{*}(v) = 1-s(v)$.

  The \emph{dual} $\mathcal{C}^{*}$ of a PCD on a binary digraph $M$
  constructed using the function $s$, is the PCD $\mathcal{C}$
  constructed using $s*$.
\end{defn}

\begin{lem}
  \label{lem:pcdofptdc}
  The PTDC of an e-graph $\Gamma$ also induces a PCD on
  $\mathcal{M}\left( \Gamma \right)$. The PCD induced from the PTDC of
  $\Gamma$ is the dual of the PCD induced by the leo of $\Gamma$.
\end{lem}
\begin{proof}
  The first statement follows from~\ref{thm:dualmedial}.  The proof of
  the second is in Figure~\ref{fig:pcdduality}.  The left column shows
  for each type of internal vertex with the PCD imposed by choice $0$,
  the second shows the e-graph constructed from that PCD, the third
  column the PCD induced by the migts.  Notice that in each case the
  PCD in the third column is exactly the PCD that corresponds to
  choice $1$.  The fourth column shows the graph constructed from
  choice $1$ with it's migts.  One can readily verify that the PCD
  imposed by those migts is exactly the PCD imposed by Choice $0$.
\end{proof}

An immediate corollary is:

\begin{thm}
  \label{thm:dualofpcd}
  For an e-graph $\Gamma$, the leo of $\Gamma^{*}$ induces on
  $\mathcal{M}\left( \mu \right)$ the PCD dual to the PCD induced by
  the leo of $\Gamma$.
\end{thm}

For example see Figure~\ref{fig:daggrassoc} that shows the medial
digraph of the graph of Figure~\ref{fig:grassoc} in the left, and of
its dual (the graph in Figure~\ref{fig:dualgrassoc}) in the right.
The different chains of the PCDs are indicated by the different colors
of the edges, and chains of length $0$ are not shown since their
presence can be deduced.

\subsection{Mind-Body dual of a factorization}
\label{sec:factorizations}

Now we can transfer this notion of duality to factorizations.

\begin{defn}
  Let $\rho = (\tau_1,\ldots, \tau_{m})$ be a factorization
in $\mathcal{S}_n$.  Then its mind-body dual $\rho^{*}$ is defined
to be the factorization associated to the mind-dual e-v-graph associated
with $\rho$.
\end{defn}

The term mind-body duality comes from the following amusing
interpretation of a transposition introduced on the episode\emph{ The
  prisoner of Benda} (sixth season, episode 10) of the animated sitcom
\emph{Futurama} and elaborated on, for example
in~\cite{EvansHuang2014}.  In this scenario there is a machine that
interchanges the minds of any two bodies that enter its two booths and
each such exchange can be encoded by a transposition.  A sequence of
transpositions can then be interpreted as a series of mind exchanges
from the point of view of the bodies. The dual sequence is then the
series of body exchanges that the corresponding minds experience.
This follows from the fact that $\nu(i)$, the neighborhood of a vertex
$i$ with its leo, stands for the sequence of mind exchanges that the
body $i$ experiences, while the trail $\overrightarrow{i}$ is the
trajectory of the mind $i$, and so it describes the sequence of body
exchanges that the mind $i$ experiences.

Pursuing this interpretation a bit, we have a set of minds $M$ and a
set of bodies $B$, of the same cardinality, and an initial assignment
of minds to bodies $\alpha_0 \co M \to B$, say each mind is assigned
to the body it's born in.  Choosing an identification of $M$ with
$[n]$, and pushing it forward via $\alpha_0$ to an identification of
$B$ with $[n]$ we can consider $\alpha_0$ to be the identity
permutation in $\mathcal{S}_n$ and any other mind-body assignment as a
permutation $\alpha \in \mathcal{S}_n$.  That way $\mathcal{S}_n$ acts
on the set of mind-body assignments on the left by permuting the minds
and on the right by permuting the bodies.  The dual of a permutation
of minds $\pi$, with respect to a mind-body assignment $\alpha$ is the
permutation of bodies $\pi^{*_{\alpha}}$ that has the same effect in
$\alpha$ as $\pi$.  In other words we want
$\pi\, \alpha = \alpha\,\pi^{*_{\alpha}}$, and it follows that
\begin{equation}
  \label{eq:mbper}
  \pi^{*_{\alpha}} = \pi^{\alpha}.
\end{equation}

Now given a factorization $\rho = \tau_1,\ldots,\tau_n$, and
considering it as a sequence of mind exchanges, it's mind-body dual is
the factorization $\rho^{*}$ which when considered as a sequence of
body exchanges, has the same effect in the mind-body assignment as
$\rho$, \emph{at every step}. The mind-body assignments we obtain by
applying $\rho$ to $\alpha_0$ are,
$\alpha_1 = \tau_1 \alpha_0 = \tau_1$,
$\alpha_2 = \tau_{2} \alpha_1 = \tau_2 \tau_1$, \ldots,
$\alpha_n = \tau_n\cdots \tau_2 \tau_1$.

Taking into account Equation~\ref{eq:mbper} we have the following explicit
formula for the mind-body dual of a factorization:

\begin{thm}
  \label{thm:exform}
     For a factorization $\rho = \tau_1,\ldots,\tau_n$ we have:
  \begin{equation}
    \label{eq:exform}
    \rho^{*} = \tau_1, \tau_2^{\tau_1},\ldots, \tau_n^{\tau_{n-1}\ldots\tau_1}
  \end{equation}
\end{thm}

We note that this formula is an expanded version of
Theorem~\ref{thm:Hurw} in the next section.

It is amusing to explain the properties of mind-body duality
described in Theorem~\ref{thm:dualproperties} in terms of the
mind-exchange interpretation.  For example the monodromy of the dual
is the inverse of the monodromy of the original, because from the
point of view of the minds, there are body-mind assignments and the
initial body-mind assignment is of course $\alpha_0^{-1}$.

The author would like to stress that despite the use of this
terminology, he does not subscribe to the philosophically untenable
position of Cartesian dualism that is implicitly assumed.  \footnote{The
  author, after long deliberations, decided to not use the term
  \emph{husband-wife duality} alluding to a more risqu\`e
  interpretation, and to leave such an interpretation to the reader if
  (s)he is so inclined.}



\section{The Hurwitz action}
\label{sec:braid}

The braid group on $m$ strands $B_m$ is the group generated
by $m-1$ generators $\sigma_i$ for $i\in [m-1]$ and relations:
$\sigma_i\sigma_j = \sigma_j\sigma_i$ if  $|i-j| > 1$, and
$\sigma_i\sigma_{i+1}\sigma_i = \sigma_{i+1}\sigma_i\sigma_{i+1}$.
For details about the braid groups we refer the reader
to~\cite{Birman1974} and~\cite{TuraevKassel2008}.  We view
the braid group as the \emph{Mapping Class Group} of a $2$-dimensional
disc $\mathbb{D}_m^2$, with $m$ distinguished points called the punctures;
that is, $B_m$ is the group of isotopy classes of orientation preserving
self homeomorphisms of $\mathbb{D}_m^2$ that fix the boundary circle
pointwise and permute the $m$ punctures.  For details about mapping
class groups of surfaces and this interpretation of the braid groups
we refer the reader to~\cite{Birman1974} and~\cite{FBMprim2012}.
We will represent braids graphically and our convention is that the positive
generator $\sigma_i$ is represented diagrammatically by the $i$-th
strand going \emph{over} the $(i+1)$-th and that multiplication in the
braid group happens from top to bottom, see Figure~\ref{fig:brgen}.

\begin{figure}[htbp]
  \centering
  \begin{pspicture}(0,-1.3)(8,2)
  \psline[linearc=.15](1,1)(1,.9)(0,.1)(0,0)
  \psline[linearc=.15,border=.1](0,1)(0,.9)(1,.1)(1,0)
  \uput[-90](.5,0){$\sigma_i$}
  \rput(3,0){
    \psline[linearc=.15](0,1)(0,.9)(1,.1)(1,0)
    \psline[linearc=.15,border=.1](1,1)(1,.9)(0,.1)(0,0)
    \uput[-90](.5,0){$\sigma_i^{-1}$}}
  \rput(5,0){
    \psline[linearc=.35](2,2)(1,1)(1,0)(2,-1)
    \psline[linearc=.35,border=.1](3,2)(3,1)(1,-1)
    \psline[linearc=.35,border=.1](1,2)(3,0)(3,-1)
    \uput[-90](2,-1){$\sigma_1 \sigma_2\sigma_1^{-1}$}}
\end{pspicture}
  \caption{Generators and multiplication in the braid group}
  \label{fig:brgen}
\end{figure}
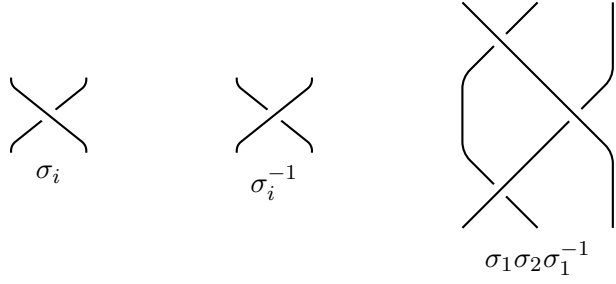

Since the fundamental group of a disc with $m$ punctures is
$\mathsf{F}_m$, the free group on $m$-generators, the interpretation
of elements of $B_{m}$ as self homeomorphisms of $\mathbb{D}_m^2$,
induces a left action of $B_m$ by automorphisms on $\mathsf{F}_m$.  If
$x_1,\ldots,x_m$ are the free generators of $\mathsf{F}_m$ then the
action of the generator $\sigma_i$ is given, on the generators of
$F_m$ by, $\sigma_i\, x_j = x_j$ for $j\ne i,i+1$, while
$\sigma_i \, x_i = x_i x_{i+1} x_i^{-1}$ and
$\sigma_i \, x_{i+1} = x_i$.  It follows that $B_m$ acts on the
right on the set of homomorphisms $\mathsf{F}_m \to G$, for any group
$G$ and in particular for $G$ a symmetric group.  A factorization
$\rho$ is a sequence of elements in a symmetric group, and therefore
can be construed as a representation of $\mathsf{F}_m$ to that group.
So we have a right action of $B_m$ on the set of all factorizations in
any symmetric group, this action is called the \emph{Hurwitz action}.
The action of a generator $\sigma_i$ on the factorization
$\rho = \tau_1,\ldots,\tau_m$ is given by\footnote{Recall that $\prescript{h}{}g$
  stands for $h g h^{-1}$.  Since transpositions are involutions, we
  could have used $\tau_{i+1}^{\tau_{i}}$ in the formula, but we
  choose to write the formula in a way that applies for any elements
  of any group.}:

\begin{equation}
  \label{eq:huract}
  \left(\rho\,  \sigma_i \right)_k =
\begin{cases}
  \tau_k & \text{if } k\neq i,i+1 \\
  \prescript{\tau_i}{}\tau_{i+1} & \text{if } k = i \\
  \tau_i & \text{if } k = i+1
\end{cases}
\end{equation}

The Hurwitz action can be described diagrammatically as in
figure~\ref{fig:hurfac}, where $i,j,k,l$ are distinct elements of
$[n]$ and $i\,j$ stands for the transposition $(i\,j)$. For more details
about the Hurwitz action see~\cite{Apos2003} and references therein.

\begin{figure}[htbp]
  \centering
  \begin{pspicture}(0,-2.8)(4,3.8)
    \rput(0,2.5){
      \psline[linearc=.15](1,1)(1,.9)(0,.1)(0,0)
      \psline[linearc=.15,border=.1](0,1)(0,.9)(1,.1)(1,0)
      \uput[90](0,1){$i\,j$}
      \uput[90](1,1){$i\,j$}
      \uput[-90](0,0){$i\,j$}
      \uput[-90](1,0){$i\,j$}}
    \rput(3,2.5){
      \psline[linearc=.15](0,1)(0,.9)(1,.1)(1,0)
      \psline[linearc=.15,border=.1](1,1)(1,.9)(0,.1)(0,0)
      \uput[90](0,1){$i\,j$}
      \uput[90](1,1){$i\,j$}
      \uput[-90](0,0){$i\,j$}
      \uput[-90](1,0){$i\,j$}}
  \rput(0,0){
      \psline[linearc=.15](1,1)(1,.9)(0,.1)(0,0)
      \psline[linearc=.15,border=.1](0,1)(0,.9)(1,.1)(1,0)
      \uput[90](0,1){$i\,j$}
      \uput[90](1,1){$j\,k$}
      \uput[-90](0,0){$i\,k$}
      \uput[-90](1,0){$i\,j$}}
  \rput(3,0){
    \psline[linearc=.15](0,1)(0,.9)(1,.1)(1,0)
    \psline[linearc=.15,border=.1](1,1)(1,.9)(0,.1)(0,0)
    \uput[90](0,1){$i\,j$}
    \uput[90](1,1){$j\,k$}
    \uput[-90](0,0){$j\,k$}
    \uput[-90](1,0){$i\,j$}}
    \rput(0,-2.5){
      \psline[linearc=.15](1,1)(1,.9)(0,.1)(0,0)
      \psline[linearc=.15,border=.1](0,1)(0,.9)(1,.1)(1,0)
      \uput[90](0,1){$i\,j$}
      \uput[90](1,1){$k\,l$}
      \uput[-90](0,0){$i\,j$}
      \uput[-90](1,0){$k\,l$}}
  \rput(3,-2.5){
    \psline[linearc=.15](0,1)(0,.9)(1,.1)(1,0)
    \psline[linearc=.15,border=.1](1,1)(1,.9)(0,.1)(0,0)
    \uput[90](0,1){$i\,j$}
    \uput[90](1,1){$k\,l$}
    \uput[-90](0,0){$i\,j$}
    \uput[-90](1,0){$k\,l$}}
  \end{pspicture}
  \caption{The Hurwitz action on factorizations}
  \label{fig:hurfac}
\end{figure}
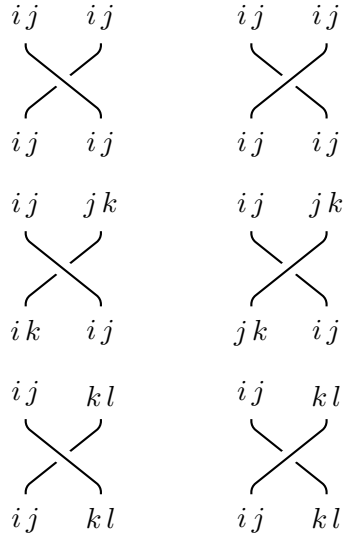

Using the bijection between factorizations of $\mathcal{S}_n$ and
e-graphs on $[n]$ (see Definition~\ref{defn:fact}), we can transfer
this to a $B_m$ action on the set of e-labeled graphs on $[n]$ with
$m$ edges.  It is easily seen that if $\Gamma$ is such an e-v-graph
then $\Gamma \sigma_i$ is obtained from $\Gamma$ by interchanging the
labels of the $i$-th and $(i+1)$-th edge and then ``sliding'' the
$(i+1)$-th edge along the $i$-th, while $\Gamma \sigma_i^{-1}$ is
obtained by interchanging the $i$-th and $(i+1)$-th labels and then
sliding the $i$-th edge along the $(i+1)$-th.  We interpret a slide of
an edge along a non-adjacent edge to have no effect.  This action on
e-v-labeled graphs, which we'll also call the \emph{Hurwitz action},
is shown in figure~\ref{fig:hurfacg}, only the edges labeled $i$ and
$i+1$ are shown since the other edges are not affected.

\begin{figure}[htbp]
  \centering
  \psset{unit=2.5}
  \begin{pspicture}(-.2,-.15)(3.85,3.8)
    \rput(-.1,3.5){
      \rput(0, 0){\rnode{a}{\psdot(0,0)}}
      \uput[180](0,0){$a$}
      \rput(1, 0){\rnode{b}{\psdot(0,0)}}
      \uput[0](1,0){$b$}
      \ncarc[arcangle=25]{a}{b}
      \ncput*{\small $i$}
      \ncarc[arcangle=-25]{a}{b}
      \ncput*{\small $i+1$}}
    \psline[arrowsize=.09]{->}(1.3,3.6)(2.45,3.6)
    \psline[arrowsize=.09]{<-}(1.3,3.4)(2.45,3.4)
    \uput[90](1.875,3.6){$\sigma_{i}$}
    \uput[-90](1.875,3.4){$\sigma_i^{-1}$}
    \rput(2.8,3.5){
      \rput(0, 0){\rnode{a}{\psdot(0,0)}}
      \uput[180](0,0){$a$}
      \rput(1, 0){\rnode{b}{\psdot(0,0)}}
      \uput[0](1,0){$b$}
      \ncarc[arcangle=25]{a}{b}
      \ncput*{\small $i$}
      \ncarc[arcangle=-25]{a}{b}
      \ncput*{\small $i+1$}}
    \rput(-.1,2.5){
      \rput(0, 0){\rnode{a}{\psdot(0,0)}}
      \uput[90](0,0){$a$}
      \rput(1, 0){\rnode{b}{\psdot(0,0)}}
      \uput[90](1,0){$b$}
      \rput(.5, -.7){\rnode{c}{\psdot(0,0)}}
      \uput[-90](.5,-.7){$c$}
      \ncline{a}{c}
      \ncput*{\small $i$}
      \ncline{b}{c}
      \ncput*{\small $i+1$}}
    \psline[arrowsize=.09]{->}(1.3,2.2)(2.45,2.2)
    \psline[arrowsize=.09]{<-}(1.3,2)(2.45,2)
    \uput[90](1.875,2.2){$\sigma_{i}$}
    \uput[-90](1.875,2){$\sigma_i^{-1}$}
    \rput(2.8,2.5){
      \rput(0, 0){\rnode{a}{\psdot(0,0)}}
      \uput[90](0,0){$a$}
      \rput(1, 0){\rnode{b}{\psdot(0,0)}}
      \uput[90](1,0){$b$}
      \rput(.5, -.7){\rnode{c}{\psdot(0,0)}}
      \uput[-90](.5,-.7){$c$}
      \ncline{a}{c}
      \ncput*{\small $i+1$}
      \ncline{a}{b}
      \ncput*{\small $i$}}
    \rput(-.1,1){
      \rput(0, 0){\rnode{a}{\psdot(0,0)}}
      \uput[90](0,0){$a$}
      \rput(1, 0){\rnode{b}{\psdot(0,0)}}
      \uput[90](1,0){$b$}
      \rput(0, -1){\rnode{c}{\psdot(0,0)}}
      \uput[-90](0,-1){$c$}
      \rput(1, -1){\rnode{d}{\psdot(0,0)}}
      \uput[-90](1,-1){$d$}
      \ncline{a}{c}
      \ncput*{\small $i$}
      \ncline{b}{d}
      \ncput*{\small $i+1$}}
    \psline[arrowsize=.09]{->}(1.3,.6)(2.45,.6)
    \psline[arrowsize=.09]{<-}(1.3,.4)(2.45,.4)
    \uput[90](1.875,.6){$\sigma_{i}$}
    \uput[-90](1.875,.4){$\sigma_i^{-1}$}
    \rput(2.8,1){
      \rput(0, 0){\rnode{a}{\psdot(0,0)}}
      \uput[90](0,0){$a$}
      \rput(1, 0){\rnode{b}{\psdot(0,0)}}
      \uput[90](1,0){$b$}
      \rput(0, -1){\rnode{c}{\psdot(0,0)}}
      \uput[-90](0,-1){$c$}
      \rput(1, -1){\rnode{d}{\psdot(0,0)}}
      \uput[-90](1,-1){$d$}
      \ncline{a}{c}
      \ncput*{\small $i+1$}
      \ncline{b}{d}
      \ncput*{\small $i$}}

  \end{pspicture}
  \caption{The Hurwitz action on e-v-graphs}
  \label{fig:hurfacg}
\end{figure}
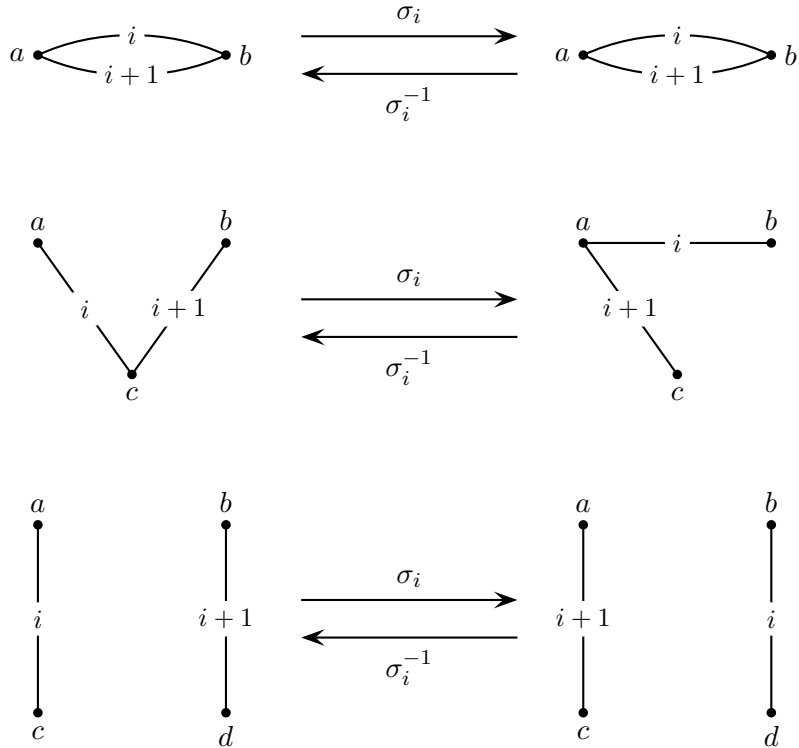

Notice that this action descends at the level of e-labeled graphs
(just forget the v-labels in Figure~\ref{fig:hurfacg}).  We will still
call it the Hurwitz action since no confusion is likely to arise.  We
remark that this $B_{m}$ action on the set of e-graphs, was also noted
in~\cite{CataWaj1991}, for the case of e-trees.

\subsection{The duality in terms of the Hurwitz action}
\label{sec:dualhurw}  We first define the notions that we need
in order to provide the promised characterization.

\begin{defn}
  \label{defn:gar}
  For $i< j \le m $ define the braid $\delta_{i,j}$ to be the braid
  that takes the $j$-th point, and moves it to the $i$-place going
  under all in between strands, leaving all the other strands
  unchanged, i.e. $\delta_{i,j} := \sigma_{j-1}\sigma_{j-2}\ldots \sigma_i$.
  We will write simply $\delta_m$ for $\delta_{1,m}$.

  The braid $\lambda_{i,j}$ is defined to be the braid that takes the
  $i$-th point, and moves it to the $j$-th place going over all the in
  between strands, i.e.
  $\lambda_{i,j} := \sigma_i\ldots \sigma_{j-1}$.  We simply write
  $\lambda_m$ for $\lambda_{1,m}$.

  We also define $\Delta_{i,j} := \delta_{i,j}\delta_{i+1,j}\ldots\delta_{j-1,j}$. We
  simply write $\Delta_m$ for $\Delta_{1,n}$ and call it the \emph{Garside element}
  of $B_m$.
\end{defn}

We summarize some of the properties of the Garside element in the following
proposition.  All of these properties are either well known or follow easily
from the definitions.

\begin{prop}
  \label{prop:garprop}  The following hold:
  \begin{enumerate}
  \item \label{item:DeltaactionBm} For all $i\in [m-1]$ we have
    $\sigma_i^{\Delta_m} = \sigma_{m-i}$.
  \item \label{item:central} $\Delta_m^2$ is central in $B_m$.  In
    fact it generates the center of $B_m$.
  \item\label{item:twist} As an element of the mapping class group of
    $\mathbb{D}_m^2$, $\Delta_m$ is represented by a homeomorphism
    that, leaving the boundary circle fixed, twists a smaller disk
    that contains all the punctures by $\pi$.
  \item \label{item:induction} $\Delta_m = \delta_m \Delta_{2,m}$.
  \item \label{item:deltalambda} $\Delta_m = \Delta_{2,m}\lambda_{m}$.
  \end{enumerate}
\end{prop}

If $e$ is an edge of a graph $\Gamma$ and $t$ a trail in $\Gamma$ ending
in a vertex incident to $e$ then we refer to the operation of detaching
$e$ from the end vertex of $t$ and attaching it to the beginning vertex
as \emph{sliding the edge $e$ along the trail $t$}.

\begin{lem}
  \label{lem:daction}
  Let $\Gamma$ be an e-v-graph with $m$ edges.  If the $m$-th edge of
  $\Gamma$ has endpoints $(v_1,v_2)$ then $\Gamma \delta_m$ is
  obtained from $\Gamma$ by sliding edge $m$ along the migts
  $\overleftarrow{v_1}$ and $\overleftarrow{v_2}$ and relabeling its
  edges according to $i \mapsto (i+1) \mod m$.
\end{lem}

\begin{proof}
  Let $i_1 < i_2< \ldots < i_k < i_{k+1}= m$ be the edges of the union
  of the trails $\overleftarrow{v_1}$ and $\overleftarrow{v_2}$. Then
  by the definition of migts as minimally increasing, the edges
  with labels $l$ with $i_k < l < m$ are not adjacent to edge $m$, the
  edges with labels $l$ in the range $i_{k-1}<l<i_k$ are not adjacent to
  the edge $i_k$, and so on.  So if we write
  $$\delta_m = (\sigma_{m-1}\ldots \sigma_{i_k+1})
  (\sigma_{i_k}\ldots\sigma_{i_{k-1}+1})\ldots (\sigma_{i_1}\ldots
  \sigma_1)$$
  then the action of the first factor has the effect of
  increasing the labels of the edges in the range $i_k<l<m$ by one and
  relabeling edge $m$ as edge $i_k+1$ without changing the underlying
  graph. The action of the second factor on the resulting e-v-labeled
  graph is then to slide the edge $i_{k+1}$ along edge $i_k$, increase
  the labels in the range $i_{k-1}<l<i_{k+1}$ by one, and relabel
  $i_{k+1}$ as $i_{k-1}+1$. This pattern continues so that the
  overall effect of the action by $\delta_{m}$ is to increase all the
  labels in the range $1\le l < m$ by one, relabel the edge originally
  labeled $m$ as $1$ and slide it along all the edges in the trails
  leading to $v_1$ or $v_2$.
\end{proof}

We can now give the characterization of mind-body duality in terms of
the Hurwitz action:
\begin{thm}
  \label{thm:Hurw}
   Let $\Gamma$ be an e-v-graph of size $m$.  Then
  \begin{equation}
    \label{eq:dudel}
     \Gamma^{*} = \left( \Gamma\Delta_m  \right)^{\intercal}
  \end{equation}
\end{thm}

\begin{proof}
  We proceed by induction on the number of edges $m$. For $m=1$ the
  theorem is obvious.  Let the edge labeled $m$ be incident to
  vertices $v_1$ and $v_2$, $v_1'$ be the starting vertex of
  $\overleftarrow{v_1}$, and $v_s'$ be the starting vertex of
  $\overleftarrow{v_2}$. Consider the e-v-labeled graph $\Gamma_1 =
  \Gamma\setminus m$ obtained by deleting edge $m$. After we attach
  the edge $m$ to $\Gamma_{1}$ the trail that ended in $v_1$ gets
  augmented by $m$ and ends in $v_2$ while the trail that ended in
  $v_2$ gets augmented by $m$ and ends in $v_1$, and all the other
  trails are the same.  It follows that $\Gamma^{*}$ is obtained from
  $\Gamma_1^{*}$ by attaching a new edge labeled $m$ to the vertices
  $v_1'$ and $v_2'$.

  On the other hand since $\Delta_m = \delta_m \Delta_{2,n}$, by
  Lemma~\ref{lem:daction}, $\Gamma \Delta_m$ is obtained by
  $\Gamma_1\Delta_{m-1}$ by increasing all edge labels by one and
  attaching an edge labeled $1$ to $v_1'$ and $v_2'$.  Now by
  induction we have that
  $\Gamma_1 \Delta_{m-1} = \left(  \Gamma_1^{*})^{\intercal} \right)$ and so it
  follows that $\Gamma \Delta_m$ is obtained from
  $\left( \Gamma_1^{*} \right)^{\intercal}$ by increasing all edge
  labels by one and adding an edge labeled $1$ attached to the
  vertices $v_1'$ and $v_2'$.  Taking the reverse we conclude that
  $\left( \Gamma\Delta_m\right)^{\intercal} $ is obtained by
  $\Gamma_1^{*}$ by attaching a new edge labeled $m$ to the vertices
  $v_1'$ and $v_2'$.
\end{proof}

Using this we can get the following formula for the mind-body dual
of a factorization:
\begin{cor}
  \label{cor:exform}
  For a factorization $\rho = \tau_1,\ldots,\tau_m$ we have:
  \begin{equation}
    \label{eq:exformleft}
    \rho^{*} = \tau_1,\prescript{\tau_1}{}\tau_2,\ldots, \prescript{\tau_1\ldots\tau_{m-1}}{}\tau_m
  \end{equation}
\end{cor}
\begin{proof}
  One can prove by induction that
$\rho\,\Delta_m = \prescript{\tau_1\ldots\tau_{m-1}}{}\tau_m, \ldots, \prescript{\tau_1}{}\tau_2, \tau_1 $
using Item~\ref{item:induction} of Proposition~\ref{prop:garprop}.
\end{proof}

Notice that since all transpositions are involutions this formula is
the same as Formula~\eqref{eq:exform} in Theorem~\ref{thm:exform}.

The Garside element plays a central role in the theory of braid groups
and can be written in terms of the generators in many interesting
ways, and each of these ways gives some information for the mind body
dual of a factorization.  We give a very general description of many
of these properties, using the language of operads.

All the braid groups can be ``put together'' into an algebraic
structure called the \emph{Braid Operad}.  We refer the reader
to~\cite{MarklShniderStasheff2002} and the references there for (some)
details.  The composition in the braid operad is given by \emph{cabling},
that is the composition
$$
B_m\times B_{i_1} \times \ldots \times B_{i_m} \to B_{i_1 + \cdots + i_m}
$$
sends $(\beta, \beta_1,\ldots,\beta_{m})$ to
$\beta \left[ \beta_1,\ldots,\beta_{m} \right]$ defined
informally\footnote{We won't give the formal definition since it would
  take us far afield.  We hope this informal description is enough for
  the reader to fill the details if (s)he wishes.} as follows: think
of the strands of $\beta$ as ``cables'' where several strands are
weaved together: the cable that corresponds to the $k$-th strand is
weaved according to the braid $\beta_{k}\in B_{i_k}$.  The braid
$\beta \left[ \beta_1,\ldots,\beta_{m} \right]$ is then the braid that
results if we forget the ``cable structure'' and view all the strands
of all the cables as strands of new bigger braid in
$B_{i_1 + \cdots + i_m}$.  See for example Figure~\ref{fig:D4oper},
for cabling using Garside elements.  With that notation in place we
can now state the following property of the Garside element, which the
author feels it should be well known but wasn't able to find a
reference in the literature.  For the statement of the following, we
take $\Delta_1 = 1$, the unique one strand braid; we also remark that
$\Delta_2 = \sigma_1$.

\begin{thm}
  \label{thm:Doperad}
  The family $\left( \Delta_k \right)_{k\ge 1}$ is a suboperad of the
  Braid Operad isomorphic to the Associative Operad\footnote{Thanks to
    Najib Idrissi for observing this in
    \href{https://mathoverflow.net/questions/289968/question-about-terminology-and-reference-request-related-to-the-braid-operad?noredirect=1\#comment718861_289968}{this
      comment in MathOverflow.}}. Indeed, for all positive integers
  $i_{1},\ldots,i_m$ we have:
$$
\Delta_m \left[ \Delta_{i_1}, \ldots, \Delta_{i_m} \right] = \Delta_{i_1 + \cdots+i_m}
$$
\end{thm}
\begin{proof}
  We proceed by induction on $m$.  For $m=1$ the result is obvious.
  For $m=2$, i.e. proving
  $\Delta_2 \left[ \Delta_{k_1}, \Delta_{k_2} \right] = \Delta_{k_1 +
    k_2}$
  we first observe that if $k_2 = 1$, this is simply a restatement of
  the definition of
  $\Delta_{k_2 +1} := \delta_{k_2 + 1} \Delta_{2,k_2 + 1} $ (see
  Definition~\ref{defn:gar}).  This can be seen in the top of
  Figure~\ref{fig:operadic}.  In the diagrammatic calculations we use
  the convention that a cable weaved according to $\Delta_k$ is
  denoted by a thick strand carrying a box labeled $k$.  If we assume
  that the result has been proved for $k_2$, then the bottom of
  Figure~\ref{fig:operadic}, and the definition of
  $\Delta_{k_1 + k_{2} +1}$ as
  $\delta_{k_1 + k_2 +1} \Delta_{2,k_1+k_2 + 1}$ proves it for
  $k_2+1$.

  Assuming now that the result has been proved for $m$ we use
  induction on $k_m$ to prove it for $m+1$.  For $k_m = 1$, it is
  again a restatement of the definition of the Garside element.
  Assuming that it has been proved for $k_{m+1}$
  Figure~\ref{fig:operadic2}, proves it for $k_{m+1}+1$, using the case
  $m=2$ that was proved above.  This concludes the induction and the
  proof.

\begin{figure}[htbp]
  \centering
  \begin{pspicture}(-1,1)(9,10)
    \psline[linearc=.35](2,10)(2,9)(1,8)(1,7.7)
    \psline[linearc=.35,linewidth=0.05,border=3pt](1,9)(1,8.7)(2,8)(2,7.7)
    \psline[linewidth=0.05](1,10)(1,9.7)
    \psline(.7,9.7)(1.3,9.7)(1.3,9.3)(.7,9.3)(.7,9.7)
    \psline[linewidth=0.05](1,9.3)(1,9)
    \rput(1,9.5){\small $k$}
    \psline[linearc=.35](4.5,10)(4.5,9.7)(3.5,9)(3.5,7.7)
    \psline[linearc=.35,linewidth=0.05,border=3pt](3.5,10)(3.5,9.7)(4.5,9)(4.5,8.4)
    \psline(4.2,8.4)(4.8,8.4)(4.8,8)(4.2,8)(4.2,8.4)
    \psline[linewidth=0.05](4.5,8)(4.5,7.7)
    \rput(4.5,8.2){\small $k$}
    \rput(-.7,8.85){$\Delta_2 \left[ \Delta_k, \Delta_1 \right] = $}
    \rput(2.6,8.85){$=$}
    \rput(6.8,8.85){$= \delta_{k+1}\Delta_{2,k+1} = \Delta_{k+1}$}
\rput(-1.2,5){$\Delta_2 \left[ \Delta_k,\Delta_{l+1} \right] = $}
    \rput(0,5){
      \psline[linewidth=0.05](2.3,1.1)(2.3,.8)
      \psline[linewidth=0.05](1,1.1)(1,.8)
      \psline[linewidth=0.05,linearc=.35](2.3,.4)(2.3,.1)(1,-.4)(1,-.7)
      \psline[linewidth=0.05,border=3pt,linearc=.35](1,.4)(1,.1)(2.3,-.4)(2.3,-.7)
      \psline(1.9,.4)(2.7,.4)(2.7,.8)(1.9,.8)(1.9,.4)
      \psline(.7,.4)(1.3,.4)(1.3,.8)(.7,.8)(.7,.4)
      \rput(2.3,.6){\tiny $l+1$}
      \rput(1,.6){\tiny $k$}}
\rput(3.5,5){$=$}
\rput(3.6,5.4){
  \psline[linearc=.35](3,1.1)(3,.8)(2,.2)(2,-.1)(1,-.6)(1,-1.9)
  \psline[linewidth=0.05,border=3pt,linearc=.35](2,1.1)(2,.8)(3,.2)(3,-.1)(3,-.3)
  \psline[linewidth=0.05,border=3pt,linearc=.35](3,-.7)(3,-1)(2,-1.6)(2,-1.9)
  \psline[linewidth=0.05,border=3pt,linearc=.35](1,.4)(1,.1)(3,-1.6)(3,-1.9)
  \psline(.7,.4)(1.3,.4)(1.3,.8)(.7,.8)(.7,.4)
  \psline[linewidth=0.05](1,1.1)(1,.8)
  \rput(1,.6){\tiny $k$}
  \psline(2.7,-.3)(3.3,-.3)(3.3,-.7)(2.7,-.7)(2.7,-.3)
  \rput(3,-.5){\tiny $l$}
}
\rput(8,5){$=$}
\rput(8.3,5.4){
  \psline[linearc=.35](3,1.1)(3,.8)(1,.2)(1,-1.9)
  \psline[linewidth=0.05,border=3pt,linearc=.35](2,1.1)(2,.8)(3,.2)(3,-.1)(3,-.3)
 \psline[linewidth=0.05,border=3pt,linearc=.35](3,-.7)(3,-1)(2,-1.6)(2,-1.9)
  \psline[linewidth=0.05,border=3pt,linearc=.35](1,1.1)(1,.8)(2,.1)(2,-.3)
  \psline[linewidth=0.05,border=3pt,linearc=.35](2,-.7)(2,-1)(3,-1.6)(3,-1.9)
  \psline(1.7,-.3)(2.3,-.3)(2.3,-.7)(1.7,-.7)(1.7,-.3)
  \rput(2,-.5){\tiny $k$}
  \psline(2.7,-.3)(3.3,-.3)(3.3,-.7)(2.7,-.7)(2.7,-.3)
  \rput(3,-.5){\tiny $l$}
}
\rput(0,2){$=$}
\rput(0,2){
  \psline[linearc=.35](2,1.1)(2,.8)(1,.2)(1,-1)
  \psline[linewidth=0.05,border=3pt,linearc=.35](1,1.1)(1,.8)(2,.2)(2,-.3)
  \psline(1.6,-.3)(2.4,-.3)(2.4,-.7)(1.6,-.7)(1.6,-.3)
  \psline[linewidth=0.05](2,-.7)(2,-1)
  \rput(2,-.5){\tiny $k+l$}
}
\rput(4,2){$=\Delta_{k+l+1}$}
  \end{pspicture}
  \caption{First Part of Proof of Theorem~\ref{thm:Doperad}}
  \label{fig:operadic}
\end{figure}
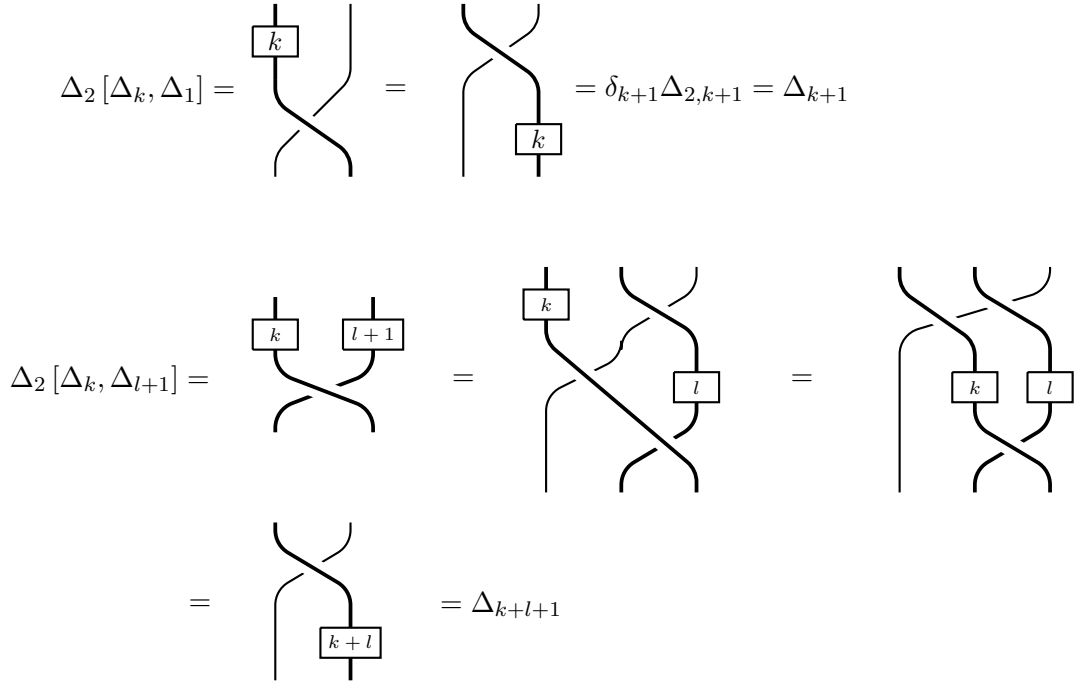

\begin{figure}[htbp]
  \centering
  \begin{pspicture}(-1,-3)(9.5,8.2)
    \rput(0,6){$\displaystyle{\Delta \left[ \Delta_{k_1},\ldots,\Delta_{k_m}, \Delta_{k_{m+1} + 1} \right] = }$}
\rput(2.9,7){
  \psline[linewidth=0.05](5.5,1.1)(5.5,.2)
  \psline[linewidth=0.05,linearc=.35](5.5,-.4)(5.5,-.8)(2,-1.9)(2,-2.2)
  \psline[border=3pt,linewidth=0.05,linearc=.35](2,-.3)(2,-.8)(5.5,-1.9)(5.5,-2.2)
  \psline(1,1.1)(1,.8)
  \psline(3,1.1)(3,.8)
  \psline(1,.4)(1,.1)
  \psline(3,.4)(3,.1)
  \psline(.7,.8)(1.3,.8)(1.3,.4)(.7,.4)(.7,.8)
  \psline(2.7,.8)(3.3,.8)(3.3,.4)(2.7,.4)(2.7,.8)
  \rput(1,.6){\tiny $k_1$}
  \rput(3,.6){\tiny $k_m$}
  \rput(2,.6){$\ldots$}
  \psline(.3,.1)(3.6,.1)(3.6,-.3)(.3,-.3)(.3,.1)
  \rput(2,-.1){\small$m$}
  \psline(4.7,.2)(6.3,.2)(6.3,-.4)(4.7,-.4)(4.7,.2)
  \rput(5.5,-.1){\small $k_{m+1}+1$}
}
\rput(-.5,2){$=$}
\rput(-1,3.5){
\rput(2,-.5){\small $\sum_1^m k_i$}
\psline[linewidth=0.05](2,.8)(2,-.15)
\psline[linearc=.35](4.2,.8)(4.2,.5)(3,-1)(3,-1.3)(2,-2.7)(2,-3)
\psline[linearc=.35,border=3pt,linewidth=0.05](3,.8)(3,.5)(4.2,-.1)(4.2,-.8)
\rput(4.2,-1){\small $k_{m+1}$}
\psline[linewidth=0.05,linearc=.35](4.2,-1.3)(4.2,-1.6)(3,-2.7)(3,-3)
\psline(3.7,-.8)(4.7,-.8)(4.7,-1.3)(3.7,-1.3)(3.7,-.8)
\psline[border=3pt,linewidth=0.05,linearc=.35](2,-.75)(2,-1)(4.2,-2.7)(4.2,-3)
\psline(1.35,-.15)(2.6,-.15)(2.6,-.75)(1.35,-.75)(1.35,-.15)
}
\rput(4.8,2){$=$}
\rput(4,3.2){
  \psline[linearc=.35](4.2,1.1)(4.2,.8)(2,-.6)(2,-2.8)
  \psline[linearc=.35,border=3pt,linewidth=0.05](3,1.1)(3,.8)(4.2,-.2)(4.2,-.8)
  \psline[linearc=.35,border=3pt,linewidth=0.05](2,1.1)(2,.8)(3,-.1)(3,-.8)
  \psline[linearc=.35,border=3pt,linewidth=0.05](4.2,-1.4)(4.2,-1.7)(3,-2.4)(3,-2.8)
  \psline[linearc=.35,border=3pt,linewidth=0.05](3,-1.5)(3,-1.8)(4.2,-2.5)(4.2,-2.8)
\psline(2.35,-.8)(3.6,-.8)(3.6,-1.5)(2.35,-1.5)(2.35,-.8)
\rput(3,-1.1){\small $\sum_1^m k_i$}
\psline(3.75,-.8)(4.7,-.8)(4.7,-1.4)(3.75,-1.4)(3.75,-.8)
\rput(4.2,-1.1){\small $k_{m+1}$}
}
\rput(-.5,-1.5){$=$}
\rput(0,-1.5){
\psline[linearc=.35](2,1.1)(2,.8)(.8,0)(.8,-1.2)
\psline[linearc=.35,border=3pt,linewidth=0.05](.8,1.1)(.8,.8)(2,0)(2,-.3)
\psline(1.2,-.3)(2.8,-.3)(2.8,-.9)(1.2,-.9)(1.2,-.3)
\rput(2,-.6){\small $\sum_1^{m+1} k_i$}
\psline[linewidth=0.05](2,-.9)(2,-1.2)
}
\rput(5,-1.5){$\displaystyle{=\Delta_{k_1+\cdots+k_{m+1}+1}}$}
  \end{pspicture}
  \caption{Second Part of Proof of Theorem~\ref{thm:Doperad}}
  \label{fig:operadic2}
\end{figure}
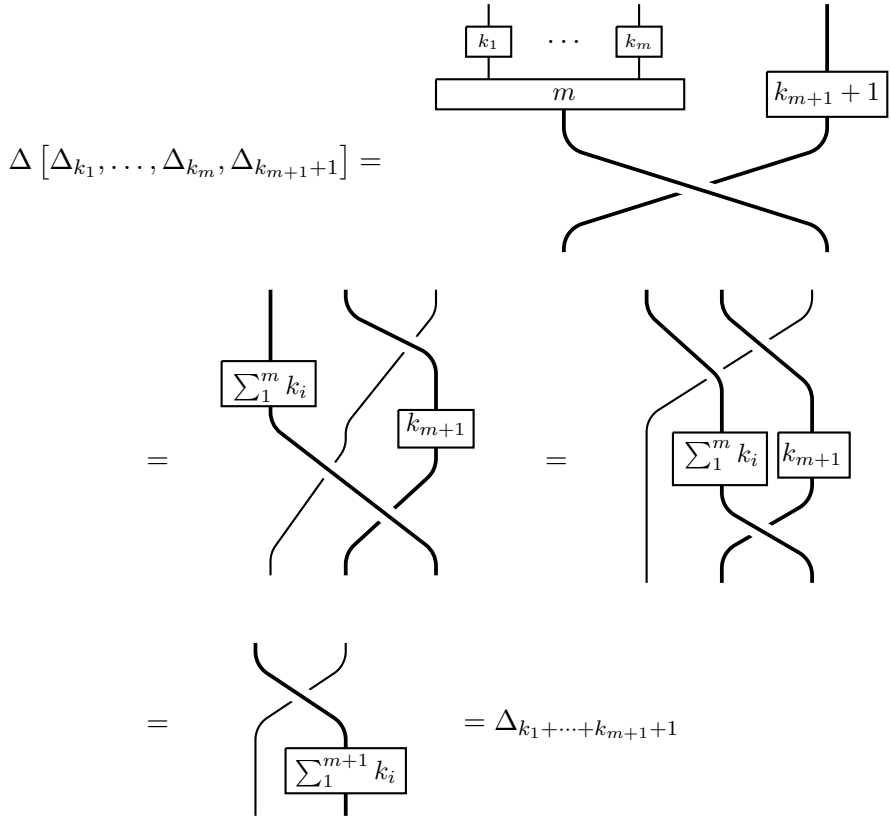
\end{proof}

For example Figure~\ref{fig:D4oper} shows the Garside element $\Delta_{4}$ as
$\Delta_2 \left[ \Delta_1,\Delta_3 \right]$, (its definition) in the left,
as $\Delta_2 \left[ \Delta_2, \Delta_2 \right]$ in the center, and as $\Delta_3 \left[ \Delta_1, \Delta_2, \Delta_1 \right]$ in the right.
\begin{figure}[htbp]
  \centering
  \begin{pspicture}(.7,-.3)(7,4)
    \rput(0,2){%
      \psset{unit=.7}
      \begin{pspicture}(4.5,6.3)
        \psline[linecolor=red,linearc=.35](4,6)(1,3)(1,1)
        \psline[linecolor=blue,linearc=.35,border=3pt](3,6)(4,5)(4,4)(2,2)(2,1)
        \psline[linecolor=blue,linearc=.35,border=3pt](2,6)(2,5)(4,3)(4,2)(3,1)
        \psline[linecolor=blue,linearc=.35,border=3pt](1,6)(1,4)(4,1)
      \end{pspicture}}
    \uput[-90](.2,.1){$\Delta_2 \left[ {\blue \Delta_3}, {\red \Delta_1} \right] $}
    \rput(4.5,1.5){%
      \psset{unit=.7}
      \begin{pspicture}(4.5,6.3)
        \psline[linearc=.35,linecolor=red](4,6)(1,3)(1,2)
        \psline[linearc=.35,linecolor=red,border=3pt](3,6)(4,5)(4,4)(2,2)
        \psline[linearc=.35,linecolor=blue,border=3pt](2,6)(1,5)(1,4)(3,2)
        \psline[linearc=.35,linecolor=blue,border=3pt](1,6)(4,3)(4,2)
      \end{pspicture}}
    \uput[-90](4.7,.1){$\Delta_2 \left[ {\blue \Delta_2}, {\red \Delta_2} \right] $}
    \rput(9,2){%
      \psset{unit=.7}
      \begin{pspicture}(4.5,6.3)
        \psline[linearc=.35,linecolor=green](4,6)(4,5)(1,2)(1,1)
        \psline[linearc=.35,linecolor=red,border=3pt](3,6)(2,5)(2,4)(3,3)(2,1)
        \psline[linearc=.35,linecolor=red,border=3pt](2,6)(4,4)(4,3)(3,1)
      \psline[linearc=.35,linecolor=blue,border=3pt](1,6)(1,3)(4,1)
      \end{pspicture}}
    \uput[-90](9,.1){$\Delta_3 \left[ {\blue \Delta_1}, {\red \Delta_2}, {\green \Delta_1} \right] $}
  \end{pspicture}
  \caption{$\Delta_4$ }
  \label{fig:D4oper}
\end{figure}



Using the operadic property of the Garside elements we can prove the following
generalization of Corollary~\ref{cor:exform}, that allow one to compute the
mind-body dual of a concatenation of factorizations ``piecewise''.

\begin{thm}
  \label{thm:dualconcat}
  For a factorization that is a concatenation of $k$ factorizations
  $\rho = \rho_1\,\rho_2\ldots\,\rho_k$, we have:
  $$ \rho^{*} = \rho_1^{*}\, \prescript{\mu(\rho_1)}{}{\left( \rho_2^{*} \right)}\,
     \prescript{\mu(\rho_1\ldots \rho_{k_1})}{}{\left( \rho_k^{*} \right)}$$
\end{thm}

An other corollary is the following:
\begin{cor}
  \label{cor:bubblesd}
  If $\rho = \tau_1, \tau_1, \tau_2, \tau_2, \ldots, \tau_k, \tau_k$
  is a factorization of the identity permutation into $k$ pairs of
  identical transpositions, then $\rho^{*} = \rho$.
\end{cor}
\begin{proof}
  This follows form the fact that $\Delta_{2k} = \Delta_k[\sigma_1, \ldots, \sigma_1]$.
  Each $\sigma_1$ fixes the corresponding pair and since the product of each pair
  is the identity permutation, the conjugations resulting from action of $\Delta_k$
  have no effect.
\end{proof}

We conclude this subsection with the following remark:
\begin{rem}
  \label{rem:convention}
  Recall that in Section~\ref{sec:duality} we made the convention
  $v^{*} = \overrightarrow{v}$ and we remarked that not much would
  change if we had made the convention $v^{*} = \overleftarrow{v}$
  instead. In our context, if we had made that convention then
  Formula~\eqref{eq:dudel} would read
  $$\Gamma^{*} = \left( \Gamma\Delta_m^{-1}  \right)^{\intercal}  $$
  instead.  Indeed, there is a straightforward analogue of
  Lemma~\ref{lem:daction} and the proof of Theorem~\ref{thm:Hurw}
  would go through almost verbatim.  For a different proof see
  Theorem~\ref{thm:Hurwbar} in the next subsection.
\end{rem}

\subsection{A closer look at the relation of the Hurwitz action and
  duality}
\label{sec:furtherhurdual}

So far we have seen two involutions on the set of our objects of
study: mind-body duality $x \mapsto x^{*}$ and reversion
$x \mapsto x^{\intercal}$.  These involutions are related to analogous
involutions on the braid group.  In this subsection we explore that
relation.

It's easy to check that the assignment $\sigma_i \mapsto \sigma_i^{-1}$
defines an (outer) automorphism $*\co B_m \to B_m$\footnote{Actually
  (see~\cite{DyerGrossman19811981}) it is the only non-trivial outer
  automorphism of $B_m$.} If $\beta\in B_m$ we will denote its image
under this automorphism by $\beta^{*}$.  Diagrammatically a diagram
for $\beta^{*}$ is obtained from a diagram of $\beta$ by reversing all
the crossings i.e. turning over-crossings to under-crossings and vice
versa.

We have the following relation of this automorphism of $B_n$ and mind-body
duality:

\begin{thm}
  \label{thm:dubr}
  Let $\beta\in B_m$ and $\Gamma$ an e-graph with $m$
  edges. Then the Hurwitz action has the following property:
  $$
  (\Gamma \beta)^{*} = \Gamma^{*}\beta^{*}
  $$
\end{thm}
\begin{proof}
  Let $\Gamma$ be an e-v-labeled graph, it suffices to prove that for
  all $i\le m-1$ we have:
  $$ \left( \Gamma \sigma_i \right)^{-1} = \Gamma^{*} \sigma_i^{-1} $$
  There are three cases, the edges $i$ and $i+1$ have two, one, or no
  vertices in common.

  In the first case $\Gamma\sigma_i = \Gamma$ so we need to prove that
  $\Gamma^{*}\sigma_i^{-1} = \Gamma^{*}$, i.e. the edges $i$ and $i+1$
  have two vertices in common in $\Gamma^{*}$ as well.  This is the
  case, because both trails, say $t_1$ and $t_2$, that contain the edge
  $i$ have to continue with $i+1$, thus in $\Gamma^{*}$ both edges $i$
  and $i+1$ connect $t_1$ to $t_{2}$.

  In the second case, in $\Gamma$ there are exactly three migts that
  contain edge $i$ or $i+1$, $t_0$ that contains both $i$ and $i+1$,
  $t_1$ that contains only $i$, and $t_2$ that contains only
  $i+1$. Then in $\Gamma^{*}$ edge $i$ connects $t_0$ to $t_1$, and
  edge $i+1$ connects $t_0$ to $t_{2}$; in particular $i$ and $i+1$
  are adjacent at $t_0$. After the action of $\sigma_i$, all migts
  except these three remain the same, while $t_0$ changes by loosing
  the edge $i+1$, $t_1$ remains the same except that the edge that was
  labeled $i$ is now labeled $i+1$, and$t_{2}$ changes by replacing
  the edge $i+1$ with two edges $i$ and $i+1$.  So in the dual of
  $\Gamma\sigma_i$ the vertex $t_2$ is connected to $t_0$ by an edge
  labeled $i$ and to $t_1$ by an edge labeled $i+1$.  See
  Figure~\ref{fig:dubr}.

\begin{figure}[htbp]
  \centering
  \psset{unit=1.5}
  \begin{pspicture}(-2,-5.5)(5,1)
    \rput(-1,0){
      \rput(-.5, 0){\rnode{a}{\psdot(0,0)}}
      \uput[135](-.5,0){$a$}
      \rput(1.5, 0){\rnode{b}{\psdot(0,0)}}
      \uput[45](1.5,0){$b$}
      \rput(.5,-1){\rnode{c}{\psdot(0,0)}}
      \uput[-90](.5,-1){$c$}
      \ncline{a}{c}
      \ncput*{\tiny $i$}
      \ncline{b}{c}
      \ncput*{\tiny $i+1$}
      \psline[ArrowInside=->, arrowsize=.125,linecolor=red,linearc=.15](-.3,.3)(.5,-.8)(1.2,.3)
      \uput[90](.5,-.6){\red \small $w_0$}
      \psline[ArrowInside=->, arrowsize=.125,linecolor=blue](.3,-1.2)(-1,0)
      \uput[180](-.35,-0.6){\blue \small $w_1$}
      \psline[ArrowInside=->, arrowsize=.125,linecolor=green](2,0)(.8,-1.2)
      \uput[0](1.4,-0.6){\green \small $w_2$}}
    \psline[arrowsize=.1]{->}(1.2,-.6)(2.2,-.6)
    \uput[90](1.775,-.6){$\sigma_{i}$}
    \rput(3.2,0){
      \rput(-.5, 0){\rnode{a}{\psdot(0,0)}}
      \uput[135](-.5,0){$a$}
      \rput(1.5, 0){\rnode{b}{\psdot(0,0)}}
      \uput[45](1.5,0){$b$}
      \rput(.5,-1){\rnode{c}{\psdot(0,0)}}
      \uput[-90](.5,-1){$c$}
      \ncline{a}{c}
      \ncput*{\tiny $i+1$}
      \ncline{a}{b}
      \ncput*{\tiny $i$}
      \psline[ArrowInside=->, arrowsize=.125,linecolor=green,linearc=.15](1.6,-.25)(0,-.25)(.9,-1.1)
      \uput[-45](.3,-.25){\green $w_2$}
      \psline[ArrowInside=->, arrowsize=.125,linecolor=red,linearc=.15](-.6,.3)(1.5,.3)
      \uput[90](.5,.3){\red $w_0$}
      \psline[ArrowInside=->, arrowsize=.125,linecolor=blue](.3,-1.2)(-1,0)
      \uput[180](-.35,-0.6){\blue \small $w_1$}
    }
    \rput(-1,-3.7){
      \rput(.5,0){\rnode{w0}{\psdot(0,0)}}
      \uput[90](.5,0){\red $w_0$}
      \rput(-.5,-1){\rnode{w1}{\psdot(0,0)}}
      \uput[225](-.5,-1){\blue $w_1$}
      \rput(1.5,-1){\rnode{w2}{\psdot(0,0)}}
      \uput[315](1.5,-1){\green $w_2$}
      \ncline{w0}{w1}
      \ncput*{\tiny $i$}
      \ncline{w0}{w2}
      \ncput*{\tiny $i+1$}}
    \psline[arrowsize=.1]{->}(1.2,-4.1)(2.35,-4.1)
    \uput[90](1.775,-4.1){$\sigma_{i}^{-1}$}
    \rput(3.2,-3.7){
      \rput(.5,0){\rnode{w0}{\psdot(0,0)}}
      \uput[90](.5,0){\red $w_0$}
      \rput(-.5,-1){\rnode{w1}{\psdot(0,0)}}
      \uput[180](-.5,-1){\blue $w_1$}
      \rput(1.5,-1){\rnode{w2}{\psdot(0,0)}}
      \uput[0](1.5,-1){\green $w_2$}
      \ncline{w1}{w2}
      \ncput*{\tiny $i+1$}
      \ncline{w0}{w2}
      \ncput*{\tiny $i$}}
    \psline[arrowsize=.11]{->}(-.5,-1.6)(-.5,-3.1)
    \uput[180](-.5,-2.35){$*$}
    \psline[arrowsize=.11]{->}(3.8,-1.6)(3.8,-3.1)
    \uput[0](3.8,-2.35){$*$}
  \end{pspicture}
  \caption{The second case in the proof of Theorem~\ref{thm:dubr}.}
  \label{fig:dubr}
\end{figure}
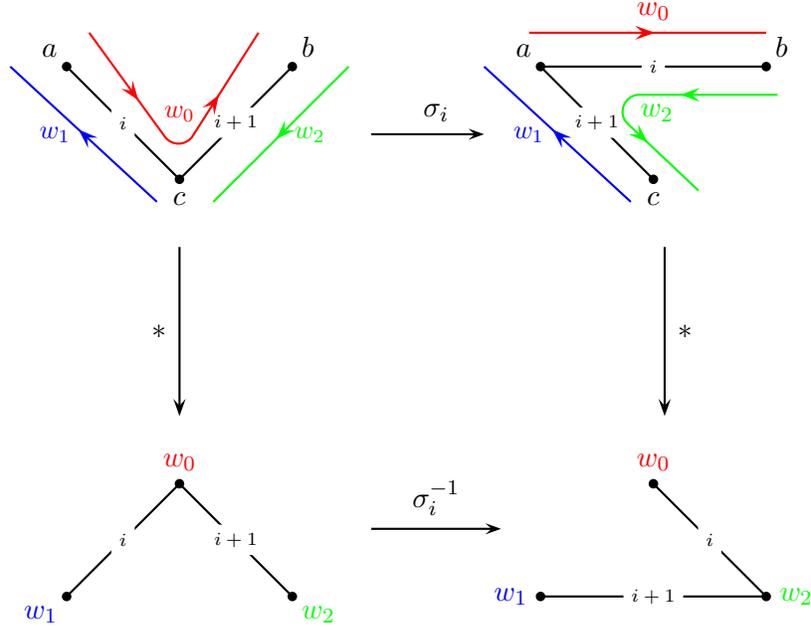

  Finally in the third case $\Gamma$ and $\Gamma \sigma_i$ have the
  same underlying graph and their labeling differs only in that the
  edges $i$ and $i+1$ have exchanged labels.  We need to prove that
  the same is true for $\Gamma^{*}$ and $\Gamma^{*}\sigma_i^{-1}$. This is
  the case because a migt that contains edge $i$ cannot contain edge
  $i+1$, since it starts with edges less or equal to $i$ and either
  ends in an endpoint of $i$ or continues with the largest edge greater
  than $i$ incident at an endpoint of $i$, which is greater that $i+1$.
  So the edges $i$ and $i+1$ cannot belong to the same migt.
\end{proof}

The reverse of a factorization is also related to an involution of
the braid group. Define the \emph{reverse} of a braid $\beta \in B_m$
to be the dual of $\beta$ conjugated by the Garside element,
i.e. $\beta^{\intercal} := (\beta^{*})^{\Delta_m}$.  Then we have:

\begin{thm}
  \label{thm:reversebr}
  For any factorization with $m$ monodromies we and any braid $\beta \in B_m$ we
  have:
  $$ \left( \rho \beta \right)^{\intercal} = \rho^{\intercal} \beta^{\intercal} $$
\end{thm}

\begin{proof}
 It suffices to
  prove this for the standard generators $\sigma_k$ of the braid group
  $B_m$. We have that for $1\le i \le n$ then $\rho^{\intercal}(i) =
  \rho(m+1-i)$. Now
  $$ \left( \rho\sigma_k \right)(i) =
  \begin{cases}
    \rho(i) & \text{if } i \ne k, k+1 \\
    \rho(k+1)^{\rho(k)} & \text{if } i = k \\
    \rho(k) & \text{if } i = k+1
  \end{cases}
  $$
  So that
  \begin{align*}
    \left( \rho \sigma_k \right)^{\intercal}(i) &= (\rho\sigma_k)(m+1-i) \\
    &= \begin{cases}
    \rho(m+1-i) & \text{if } m + 1 -i \ne k, k+1 \\
    \rho(k+1)^{\rho(k)} & \text{if } m+1-i = k\\
    \rho(k) & \text{if } m+1 -i = k + 1
  \end{cases}\\
    &= \begin{cases}
    \rho(m+1-i) & \text{if } i \ne m+1-k, m-k \\
    \rho(k+1)^{\rho(k)} & \text{if } i = m+1-k\\
    \rho(k) & \text{if } i = m-k
  \end{cases}
  \end{align*}
 Note that for the standard generators $\sigma_i\in B_m$ we
  have $\sigma_k^{\intercal} = \sigma_{n-k}^{-1}$.  So
  \begin{align*}
   \left(  \rho \sigma_k \right)^{\intercal}(i) &= \rho^{\intercal}\sigma_{n-k}^{-1}(i)\\
    &= \begin{cases}
      \rho^{\intercal}(i) & \text{if } i \ne m-k, m-k+1 \\
      \rho^{\intercal}(m-k+1) & \text{if } i = m-k \\
      \rho^{\intercal}(m-k)^{\rho^{\intercal}(m-k+1)} & \text{if } i = m-k+1
  \end{cases} \\
  &= \begin{cases}
    \rho(m+1-i) & \text{if } i \ne m-k, m-k+1 \\
    \rho(k) & \text{if } i = n-k \\
    \rho(k+1)^{\rho(k)} & \text{if } i = m-k+1
  \end{cases}
  \end{align*}
\end{proof}

We conclude this section by having a closer look at the action of the
Garside element $\Delta_m$.

\begin{prop}
  \label{prop:centralaction}
  Let $\rho$ be a factorization, then
  $$\rho \Delta_m^2 = \prescript{\mu(\rho)}{}\rho$$.
\end{prop}
\begin{proof}
  We will use the well known fact that $\Delta_m^2 = \lambda_m^m$,
  where $\lambda_m = \sigma_1\ldots \sigma_{m-1}$ (see
  Definition~\ref{defn:gar}).  It is easy to see that for an e-v-graph
  $\Gamma$, $\Gamma\lambda_m$ is obtained from $\Gamma$ by sliding all
  the edges along the edge labeled $1$ and then relabel all the edges
  according to $i\to i-1 \mod m$, or equivalently, interchanging the
  two vertex labels of the edge labeled $1$ and then relabeling the
  edges according to $i\mapsto i-1 \mod m$.  After $m$ iterations all
  the edges have their original labels, and each edge has interchanged
  its labels, in the order of the original e-labeling.  This means
  that the vertex labels have been relabeled according to
  $\mu(\Gamma)$.
\end{proof}

At the level of e-graphs the vertex labels are not important so it follows:
\begin{cor}
  \label{cor:ecentralaction} Action by $\Delta_m^2$ fixes all
  e-graphs, and therefore the action of $\Delta_m$ is an involution
  on the set of all e-graphs.
\end{cor}

Recall that in Section~\ref{sec:egraphdual} we made the convention
that in the case of an e-v-graph $\Gamma$ the vertex labeling of
$\Gamma^{*}$ is that the vertex that corresponds to the migt
$\overrightarrow{v}$ gets the same label as $v$.  With this definition
of duality we have that
$\left( \Gamma^{*} \right)^{\intercal} = \Gamma \Delta $. Let
$\bar{*}$ be the dual defined by the convention that the vertex of the
dual that corresponds to the migt $\overleftarrow{v}$ gets the same
label as $v$, in other words, $v^{\bar{*}} = \overleftarrow{v}$.  Then
we can prove the following:

\begin{thm}
  \label{thm:Hurwbar}
  For an e-v-labeled graph with $m$ edges, or factorization with $m$
  monodromies, $x$, we have:
$$
x^{\bar{*}} = \left( x\Delta_m^{-1} \right)^{\intercal}
$$
 In particular for a factorization $\rho = \tau_1\ldots \tau_m$ we have
 $$ \rho^{\bar{*}} = \tau_1^{\tau_2\tau_3\ldots\tau_m}, \tau_2^{\tau_3\ldots\tau_m},
     \ldots, \tau_{m-1}^{\tau_m}, \tau_m$$
\end{thm}

We will need the following Lemma in the proof:

\begin{lem}
  \label{lem:deltastar} We have:
  $$ \Delta_m^{\intercal} = \Delta_m^{*} = \Delta_m^{-1} $$
\end{lem}
\begin{proof}
  If $\Delta_m^{*} = \Delta_m^{-1}$ then
  $\Delta_m^{\intercal} = \left( \Delta_m^{-1} \right)^{\Delta_m} =
  \Delta_m^{-1}$. To prove the former we proceed by induction.
  For $m = 1,2$ it is clear.  Assuming that it has been proved
  for $m$ we have
  \begin{align*}
    \Delta_{m+1}^{*} &= \delta_{m+1}^{*} \Delta_{2,m+1}^{*}\\
                   &= \lambda_{m+1}^{-1} \Delta_{2,m+1}^{-1}\\
                   &= \left( \Delta_{2,m+1} \lambda_{m+1} \right)^{-1} \\
                   &= \Delta_{m+1}^{-1}
  \end{align*}
  where we used Items~\ref{item:induction} and~\ref{item:deltalambda}
  of Proposition~\ref{prop:garprop}, and the easily checked fact that
  $\delta_m^{*}=\lambda_m^{-1}$.
\end{proof}

\begin{proof}[Proof of Theorem~\ref{thm:Hurwbar}]
  Clearly the migts of $\Gamma^{\intercal}$, for an e-v-graph $\Gamma$
  are the inverses of the migts of $\Gamma$, and therefore
  $\Gamma^{\bar{*}}$ is the reverse of the mind-body dual of the reverse
  of $\Gamma$. So:
  \begin{align*}
    \Gamma^{\bar{*}} &= \left( \left( \Gamma^{\intercal} \right)^{*}
                       \right)^{\intercal}\\
                     &= \left( \left( \left( \Gamma^{\intercal} \right) \Delta_m \right)^{\intercal} \right)^{\intercal}\\
    &= \Gamma^{\intercal}\Delta_m \\
    &= \left( \Gamma \Delta^{\intercal} \right)^{\intercal} \\
    &= \left( \Gamma \Delta_m^{-1} \right)^{\intercal}
  \end{align*}
  The formula for the $\rho^{\bar{*}}$ is just the expanded form of this.
\end{proof}

\subsection{Loop Braid Group Action}
\label{sec:loopbraid}
The \emph{Loop Braid Group} $LB_m$ in $m$ strands is an extension of
the braid group and has been defined several times in the literature
under a variety of different names and points of view, we refer the
reader to~ \cite{Damiani2017} for a survey of the different
manifestations of these groups.  We will follow the spirit
of~\cite{FennRourkeRimanyi1997}: $LB_m$ is generated by $2(m-1)$
generators $\sigma_1,\ldots,\sigma_{m-1}, s_1,\ldots,s_{m-1}$, the
$\sigma_i$ generate a subgroup isomorphic to $B_m$ and the $s_i$ a
subgroup isomorphic to $\mathcal{S}_m$ ($s_i$ stands for the
transposition $(i\, i+1)$) and there are three types of additional
relations involving generators of both types:
$\sigma_i s_j = s_j \sigma_i$ for $\left| i-j \right| > 1$,
$s_is_{i+1}\sigma_i = \sigma_{i+1}s_is_{i+1}$, and
$\sigma_i \sigma_{i+1} s_i = s_{i+1}\sigma_i\sigma_{i+1}$.

$LB_m$ is isomorphic to a subgroup of the automorphism group of
$\mathsf{F}_m$, where the $\sigma_i$ acts like a braid, while $s_i$
interchanges the $i$-th and $(i+1)$-th generators.  It follows that
the Hurwitz action extends to an action of $LB_m$, where $s_i$ just
interchanges the $i$-th and $(i+1)$-th monodromy (or the labels of the
corresponding edges for an e-graph).

\begin{defn}
  \label{defn:permDn}
  The element of $LB_n$ corresponding to the permutation
  $\prod_{i=1}^{\lfloor m/2\rfloor} (i,m+1-i)$ is denoted by $D_m$.
  Notice that $D_m$ is the image of $\Delta_m$ under the standard
  surjection $B_m \to \mathcal{S}_m$.

  The \emph{dualizer} is the element $d_m := \Delta_m D_m \in LB_m$.
\end{defn}

It is clear that $\rho D_m = \rho^{\intercal}$, so we have the following theorem
justifying the name dualizer:

\begin{thm}
  \label{thm:dualizer}
  If $x$ is an e-graph with $m$ edges, or a factorization with $m$
  monodromies, we have:
  $$ x^{*} = x d_m $$
\end{thm}

\section{Properly Embedded Graphs}
\label{sec:pegs}

The reader may have noticed the close analogy of the mind-body dual
with the dual of a graph embedded in a (closed oriented) surface.  In
this section we elaborate on that analogy.  We refer the reader
to~\cite{gross1987topological} and~\cite{LandoZvonkin2004} for the
rich theory of graph embeddings and maps.  Most of the constructions
in this section are entirely analogous to the usual case of embeddings
in a closed surface, so some of the details are skipped trusting the
reader to supply them.

\begin{defn}
  \label{defn:peg}
  Let $F$ be an oriented surface with boundary and $\Gamma$ a graph.
  A \emph{proper embedding of $\Gamma$ into $F$} is an embedding
  $i\co \Gamma \to F$ such that:
  \begin{enumerate}
  \item The vertices of $\Gamma$ are mapped in the boundary of $F$,
    i.e.  $i(V) \subset \partial F$, and the interior of each edge of
    $\Gamma$ is mapped into the interior of $F$.
  \item $F\setminus i(\Gamma)$ is a disjoint union of simply connected
    domains, called the \emph{regions} of the embedding, and the
    interior of each region is homeomorphic to an open disc.
  \item $\partial F \setminus i(V)$ is a disjoint union of open
    intervals, called the \emph{arcs} of the embedding, and the closure
    of each region contains exactly one arc.
  \end{enumerate}

  A \emph{properly embedded graph} (\emph{peg} for short) is a graph
  endowed with a proper embedding into a surface.  We will abuse
  notation by not distinguishing a peg from its image, and we will use
  the same symbol to denote a peg, its underlying graph, or even the
  surface that the graph is embedded.

  An \emph{isomorphism of pegs} is an orientation preserving
  homeomorphism of their surfaces that restricts to a graph
  isomorphism on the images of the graphs.  Two pegs are called
  \emph{isomorphic} if there is an isomorphism between them\footnote{
  More nuanced notions of maps and equivalences between pegs will be
  considered in the planned work~\cite{Apostolakis2018b}}.
\end{defn}

We emphasize that, contrary to the usual convention in the theory of
graph embeddings, neither $\Gamma$ nor the surface $F$ is assumed
connected.

The prototype of a peg is a \emph{non-crossing tree}.  Non-crossing
trees are well studied in the literature, see for
example~\cite{DulPen1993}, and~\cite{Noy1998301}, from our perspective
a non-crossing tree is a tree properly embedded in a disk.  The left
side of Figure~\ref{fig:nctree} shows a non-crossing tree, and the
right side shows a unicycle\footnote{That is a connected graph with
  exactly one cycle.} properly embedded in an annulus (drawn with
thick black lines). See also the right side of
Figure~\ref{fig:pegofexample} that shows the graph of
Figure~\ref{fig:grassoc} properly embedded in a torus with a disk
removed.

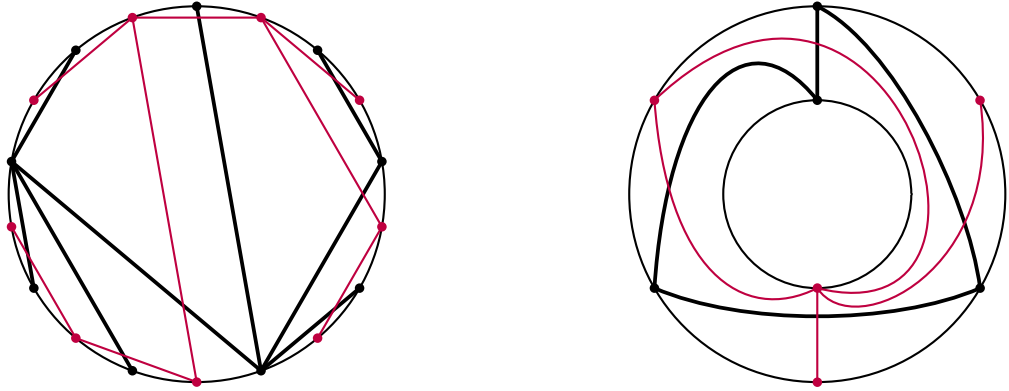
\begin{figure}[htbp]
  \centering
\psset{unit=2.5}
\begin{pspicture}(-1.2,-1.2)(5,1.2)
\pscircle(0,0){1}
 \rput(0, 1.0){\rnode{1}{\psdot(0,0)}}
 \rput(-0.6427876096865393, 0.7660444431189781){\rnode{2}{\psdot(0,0)}}
 \rput(-0.9848077530122081, 0.17364817766693025){\rnode{3}{\psdot(0,0)}}
 \rput(-0.8660254037844388, -0.4999999999999998){\rnode{4}{\psdot(0,0)}}
 \rput(-0.3420201433256685, -0.9396926207859084){\rnode{5}{\psdot(0,0)}}
 \rput(0.3420201433256682, -0.9396926207859085){\rnode{6}{\psdot(0,0)}}
 \rput(0.8660254037844384, -0.5000000000000004){\rnode{7}{\psdot(0,0)}}
 \rput(0.9848077530122081, 0.17364817766692991){\rnode{8}{\psdot(0,0)}}
 \rput(0.6427876096865397, 0.7660444431189778){\rnode{9}{\psdot(0,0)}}
 %
 \ncline[linewidth=0.02]{1}{6}
 \ncline[linewidth=0.02]{2}{3}
 \ncline[linewidth=0.02]{3}{4}
 \ncline[linewidth=0.02]{3}{5}
 \ncline[linewidth=0.02]{3}{6}
 \ncline[linewidth=0.02]{7}{6}
 \ncline[linewidth=0.02]{8}{6}
 \ncline[linewidth=0.02]{8}{9}

 \rput(-0.342020143325669, 0.939692620785908){\rnode{1}{\psdot[linecolor=purple](0,0)}}
 \rput(-0.866025403784439, 0.500000000000000){\rnode{2}{\psdot[linecolor=purple](0,0)}}
  \rput(-0.984807753012208, -0.173648177666930){\rnode{3}{\psdot[linecolor=purple](0,0)}}
 \rput(-0.642787609686540, -0.766044443118978){\rnode{4}{\psdot[linecolor=purple](0,0)}}
 \rput(0, -1.00000000000000){\rnode{5}{\psdot[linecolor=purple](0,0)}}
  \rput(0.642787609686539, -0.766044443118978){\rnode{6}{\psdot[linecolor=purple](0,0)}}
  \rput(0.984807753012208, -0.173648177666931){\rnode{7}{\psdot[linecolor=purple](0,0)}}
 \rput(0.866025403784439, 0.500000000000000){\rnode{8}{\psdot[linecolor=purple](0,0)}}
 \rput(0.342020143325669, 0.939692620785908){\rnode{9}{\psdot[linecolor=purple](0,0)}}

 \ncline[linecolor=purple]{1}{2}
 \ncline[linecolor=purple]{1}{5}
 \ncline[linecolor=purple]{1}{9}
 \ncline[linecolor=purple]{3}{4}
 \ncline[linecolor=purple]{4}{5}
 \ncline[linecolor=purple]{7}{6}
 \ncline[linecolor=purple]{7}{9}
 \ncline[linecolor=purple]{8}{9}

\rput(3.3,0){
  \pscircle(0,0){1}
  \pscircle(0,0){.5}
  \rput(0,1){\rnode{1}{\psdot(0,0)}}
  \rput(-0.866,-0.5){\rnode{2}{\psdot(0,0)}}
  \rput(0.866,-0.5){\rnode{3}{\psdot(0,0)}}
  \rput(0,.5){\rnode{4}{\psdot(0,0)}}
  \ncline[linewidth=0.02]{1}{4}
  \psbezier[linewidth=0.02](0,1)(.4,.8)(.8,0)(.866,-.5)
  \psbezier[linewidth=0.02](0,.5)(-.4,1)(-.8,.5)(-.866,-.5)
  \psbezier[linewidth=0.02](-.866,-.5)(-.4,-.7)(.4,-.7)(.866,-.5)

\psdots[linecolor=purple](-0.866, 0.5)(0,-1)(0.866,0.5)(0,-.5)

\psline[linecolor=purple](0,-1)(0,-.5)
\psbezier[linecolor=purple](0,-.5)(-.4,-.7)(-.8,-.4)(-.866,.5)
\psbezier[linecolor=purple](-.866,.5)(.4,1.7)(1.2,-.8)(0,-.5)
\psbezier[linecolor=purple](0,-.5)(.2,-.8)(1,-.4)(.866,.5)
}
\end{pspicture}
  \caption{Properly embedded tree (left) and unicycle (right).}
  \label{fig:nctree}
\end{figure}

If $\Gamma$ is properly embedded in $F$ then the vertices and arcs
endow the boundary with the structure of a $1$-complex, i.e. a graph.
The orientation of $\partial F$ further endows it with a digraph
structure, and clearly that digraph is the functional digraph of a
permutation.

\begin{defn}
  \label{defn:muofpeg} The \emph{monodromy digraph of a peg $\Gamma$}
  is the inverse of the digraph $\partial F$ described above.  The
  \emph{monodromy of $\Gamma$} is the permutation of $V$ defined by
  this digraph.  As usual we will use the same symbol $\mu(\Gamma)$
  to denote the monodromy of a peg or its functional digraph.

  Given a graph $\Gamma$ pegged in the oriented surface $F$, the
  orientation of $F$ induces a leo structure on $\Gamma$: for any
  vertex $v$ order the edges in $\nu(v)$ starting from the rightmost
  one and proceeding according to the orientation.
\end{defn}

The leo structure of the definition above means that a peg $\Gamma$
has a \emph{medial digraph}. That digraph can be considered actually
embedded (not properly) in $F$: simply put the vertex of $M$ in the
middle of the edge of $\Gamma$ that it corresponds to, and draw the
edge that connects two consecutive edges inside the region of $\Gamma$
in whose boundary they lie.

Given an e-graph (or more generally a leo) there is a natural way to
construct a peg, analogous to the way one defines a cellularly
embedded graph from a rotation scheme,
see~\cite{gross1987topological}, and the discussion in
Subsection~\ref{sec:cegrel}.  One obtains a sort of \emph{ribbon
  graph} with the vertices in the boundary instead of the interior.
We describe this procedure with more details below:

\begin{defn}
  \label{defn:pegofleo}
  The \emph{peg associated with a leo $\Gamma$} is obtained as
  follows: Thicken the star of each vertex, into a ribbon surface as
  shown in Figure~\ref{fig:thick} (for a vertex of degree $4$), so
  that the vertex is at the boundary and each edge is thickened into a
  ribbon with the ribbons arranged according to the local order around
  that vertex. Then glue, \emph{via an orientation reversing
    homeomorphism of the interval} the free ends of any pair of
  ribbons that correspond to the two half edges of the same edge of
  $\Gamma$ together to get a surface $F$ with $\Gamma$ embedded in it
  in such a way that the vertices are on the boundary.  We use
  $P \left( \Gamma \right)$ to denote the peg of a leo $\Gamma$.
  \end{defn}

  \begin{figure}[htbp]
    \centering
    \psset{unit=2}
    \begin{pspicture}(-2,-.8)(2,.8)
      \rput(-1.5,0){%
        \psset{unit=.6}
      \psline(0,0)(-1,1)
      \psline(0,0)(1,1)
      \psline(0,0)(1,-1)
      \psline(0,0)(-1,-1)
      \psdots(0,0)(-1,1)(1,1)(1,-1)(-1,-1)
      \psarc[linecolor=red,arrowsize=.2]{->}(0,0){.5}{-135}{135}}
    \rput(1.5,-.5){%
      \psset{unit=1.2}
      \pscustom[linecolor=white,fillstyle=solid,fillcolor=lightblue]{%
        \psline(-.86,0.2)(0,0)(-0.866025403784,0.5)(-.86,0.2)}
      \pscustom[linecolor=white,fillstyle=solid,fillcolor=lightblue]{%
        \psline(0,0)(-0.866025403784,0.5)
        \psline[linearc=.03](-.83,0.6)(-.2,.2)(-.6,.85)
        \psline(-0.5,0.866025403784)(0,0)}
      \pscustom[linecolor=white,fillstyle=solid,fillcolor=lightblue]{%
        \psline(0,0)(-0.5,0.866025403784)
        \psline[linearc=.03](-.4,.85)(0,.2)(.4,.85)
        \psline(0.5,0.866025403784)(0,0)}
      \pscustom[linecolor=white,fillstyle=solid,fillcolor=lightblue]{%
       \psline(0,0)(0.5,0.866025403784)
      \psline[linearc=.03](.6,.85)(.2,.2)(.83,0.6)
      \psline(0.866025403784,0.5)(0,0)}
      \pscustom[linecolor=white,fillstyle=solid,fillcolor=lightblue]{%
        \psline(0.866025403784,0.5)(0,0)(.86,0.2)
}
      \psline[linecolor=blue,linewidth=.02](-.86,0.2)(0,0)
      \psline[linecolor=blue,linewidth=.02](.86,0.2)(0,0)
      \psline[linecolor=blue,linearc=.03,linewidth=.02](-.83,0.6)(-.2,.2)(-.6,.85)
      \psline[linewidth=.02](0,0)(0.866025403784,0.5)
      \psline[linecolor=blue,linearc=.03,linewidth=.02](-.4,.85)(0,.2)(.4,.85)
      \psline[linewidth=.02](0,0)(0.5,0.866025403784)
      \psline[linecolor=blue,linearc=.03,linewidth=.02](.83,0.6)(.2,.2)(.6,.85)
      \psline[linewidth=.02](0,0)(-0.866025403784,0.5)
      \psline[linewidth=.02](0,0)(-0.5,0.866025403784)
      \psdot(0,0)
    \psarc[linecolor=red,arrowsize=.09]{->}(0,0){.17}{30}{150}}
    \end{pspicture}
    \caption{Thickening the star of a vertex}
    \label{fig:thick}
  \end{figure}
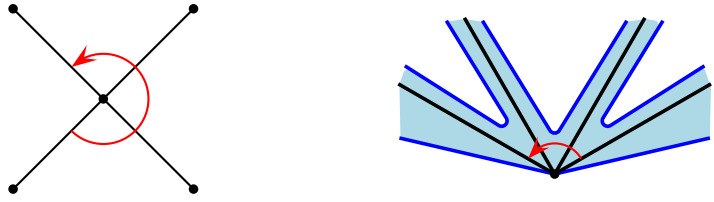

For example the left hand side of Figure~\ref{fig:pegofexample} shows
the thickening of the stars of the vertices of the e-graph of
Figure~\ref{fig:grassoc} while the right side shows them assembled
into a peg.

  \begin{figure}[htbp]
    \centering
    \begin{pspicture}(-6.8,-2.4)(6.2,2.4)
      \rput(-3.5,0){\includegraphics[scale=.3]{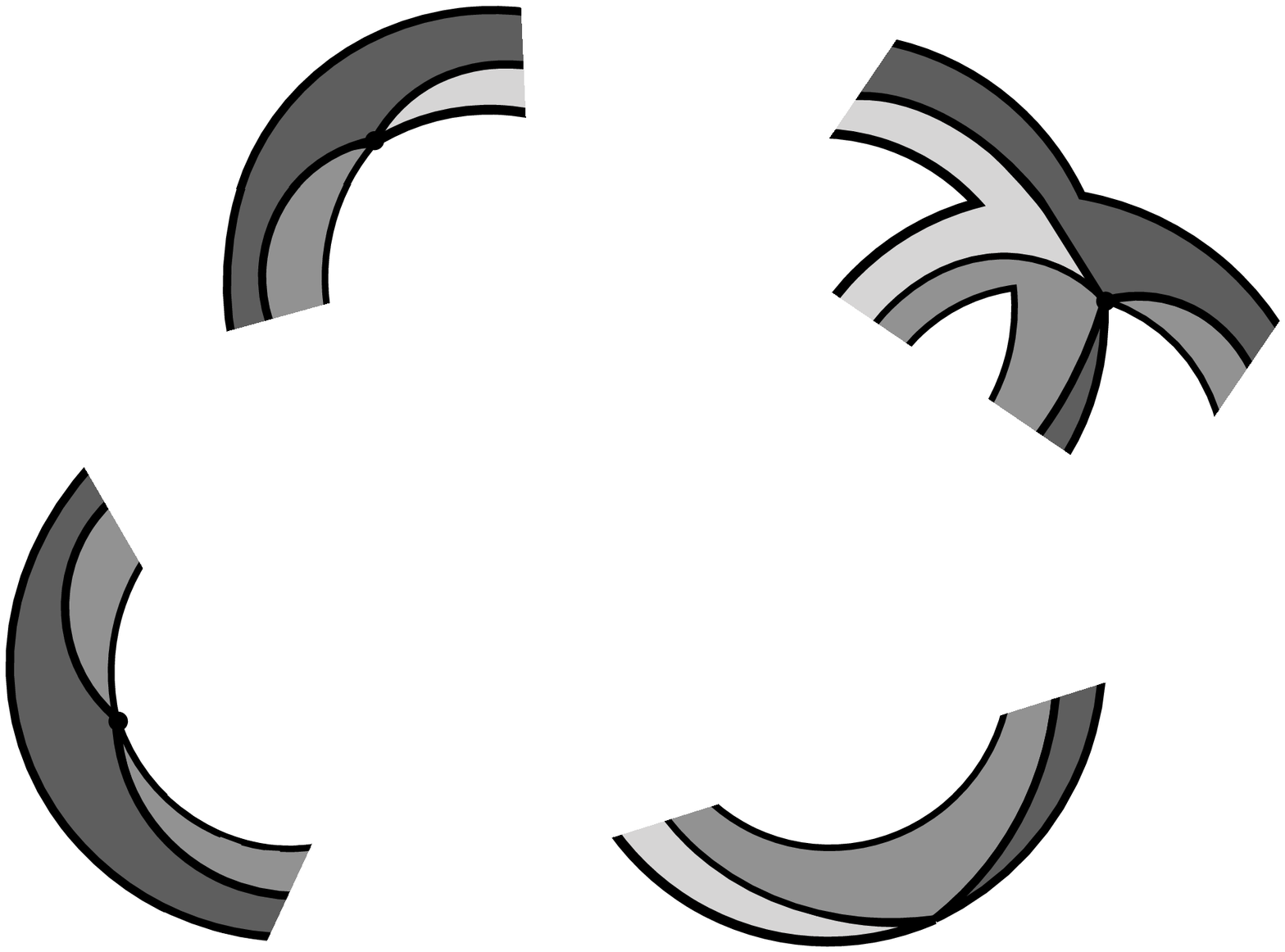}}
      \rput(3.5,0){\includegraphics[scale=.45]{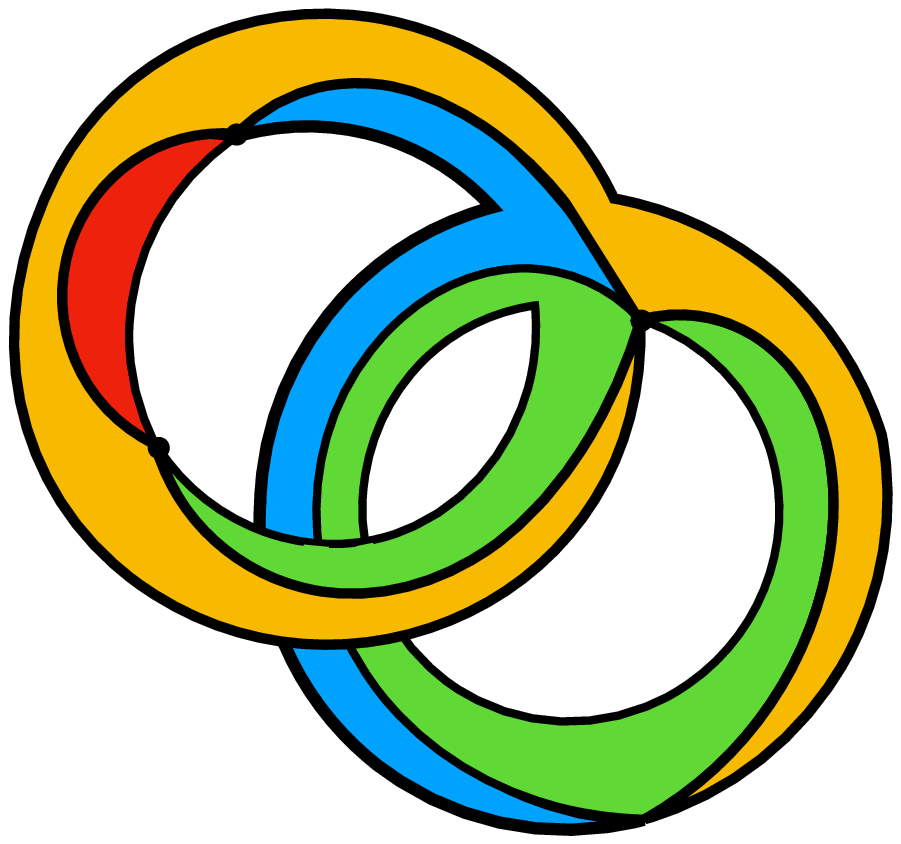}}
      \uput[0](-4.8,1.4){\small $1$}
      \uput[-90](-1.3,.8){\small $3$}
      \uput[0](-5.9,-1){\small $2$}
      \uput[-90](-2.2,-1.9){\small $4$}
      \uput[-135](-2.5,.8){\raisebox{.5pt}{\textcircled{\raisebox{-.9pt} {\small $4$}}}}
      \uput[-135](-1.7,.3){\raisebox{.5pt}{\textcircled{\raisebox{-.9pt} {\small $5$}}}}
      \uput[-45](-.8,.5){\raisebox{.5pt}{\textcircled{\raisebox{-.9pt} {\small $1$}}}}
      \uput[-135](-2.5,2.1){\raisebox{.5pt}{\textcircled{\raisebox{-.9pt} {\small $2$}}}}
      \uput[45](-4,1.7){\raisebox{.5pt}{\textcircled{\raisebox{-.9pt} {\small $2$}}}}
      \uput[90](-1.8,-1){\raisebox{.5pt}{\textcircled{\raisebox{-.9pt} {\small $1$}}}}
      \uput[90](-3.5,-1.5){\raisebox{.5pt}{\textcircled{\raisebox{-.9pt} {\small $4$}}}}
      \uput[0](-5.1,-1.8){\raisebox{.5pt}{\textcircled{\raisebox{-.9pt} {\small $5$}}}}
      \uput[45](-6,-.3){\raisebox{.5pt}{\textcircled{\raisebox{-.9pt} {\small $3$}}}}
      \uput[-90](-5.2,.8){\raisebox{.5pt}{\textcircled{\raisebox{-.9pt} {\small $3$}}}}
      \uput[-90](2.5,1.4){\blue \small $1$}
      \uput[0](2.1,.1){\red \small $2$}
      \uput[-90](4.55,.55){\green \small $3$}
      \uput[-90](4.6,-1.6){\orange \small $4$}
      \uput[0](5.4,-.5){\raisebox{.5pt}{\textcircled{\raisebox{-.9pt} {\small $1$}}}}
      \uput[90](3.8,1.7){\raisebox{.5pt}{\textcircled{\raisebox{-.9pt} {\small $2$}}}}
      \uput[0](2,.6){\raisebox{.5pt}{\textcircled{\raisebox{-.9pt} {\small $3$}}}}
      \uput[180](3.2,-1.6){\raisebox{.5pt}{\textcircled{\raisebox{-.9pt} {\small $4$}}}}
      \uput[0](4,-.6){\raisebox{.5pt}{\textcircled{\raisebox{-.9pt} {\small $5$}}}}
    \end{pspicture}
    \caption{The peg of the e-v-graph of Figure~\ref{fig:grassoc}}
    \label{fig:pegofexample}
  \end{figure}

Notice that the boundary of each region (colored with
different colors) is the union of a migt of $\Gamma$ and an arc of
$\partial P(\Gamma)$.  In fact one can easily see that this is
always the case:

\begin{lem}
  \label{lem:pegfromptdc}
  $P \left( \Gamma \right)$ is indeed a peg and is isomorphic to the
  peg obtained by attaching a half disk, that is a domain homeomorphic
  to $\left\{ (x,y) \in \mathbb{D}^{2} : y\ge 0\right\}$, along each
  migt of $\Gamma$, in such a way that the part of the boundary of the
  half disk that lies on the real axis is identified with the migt.
\end{lem}
\begin{proof}
  It's clear that the resulting complex is an orientable ribbon
  surface, with the vertices of the graph in the boundary, and the
  interior of the edges in the interior. To see that the regions are
  indeed open discs with exactly one arc in their boundary, notice
  that each thickened edge separates its ribbon in to two simply
  connected regions whose interior is an open disc, and when we glue
  the half edges together the four regions are glued together
  according to the migts of $\Gamma$ because the gluing is happening
  via an orientation reversing homeomorphism.
\end{proof}

We note that an alternative construction of $P(\Gamma)$ as the total
space of a branched covering over the disc is given in
Subsection~\ref{sec:brcov} below, see Theorem~\ref{thm:pegthroubrcov}.

We have the following general properties:
\begin{lem}
  \label{lem:eulermu} Let $\Gamma$ be an e-graph.  The following hold:
    \begin{enumerate}
    \item \label{item:b0peg} $P(\Gamma)$ has the same number of
      connected components as $\Gamma$.
    \item \label{item:euler}
      $\chi \left( P \left( \Gamma \right) \right) = \chi(\Gamma)$.
    \item \label{item:mupeg} The monodromy digraph of $P(\Gamma)$ is
      isomorphic to the monodromy digraph of $\Gamma$.
    \item \label{item:peggenus} If $\Gamma$ is connected and
      $\mu(\Gamma) \in \mathcal{S}_V$ has $b$ disjoint cycles then
      $$ \chi \left( P \left( \Gamma \right) \right) = \chi \left( \Gamma \right) + b$$
      and therefore genus of $P(\Gamma)$ is:
      $$ g = \frac{2+m-n-b}{2} $$
    \end{enumerate}
  \end{lem}

  \begin{proof}
    Clearly the surface of $P(\Gamma)$ can be homotopically collapsed
    to $\Gamma$ so Items~\ref{item:b0peg} and~\ref{item:euler} are
    immediate.  Item~\ref{item:mupeg} follows from the fact that the
    boundary of a region of $P(\Gamma)$ consists of a migt and an arc,
    and since $P(G)$ is oriented, the migt and the arc have opposite
    orientations.  Item~\ref{item:peggenus} is also clear, since the
    resulting surface has $b$ boundary components.
  \end{proof}

We remark that the notion of peg is more general than that of e-graph.
Indeed given any leo $\Gamma$, e-realizable or not, one can define the
peg $P(\Gamma)$.  For example the peg of Figure~\ref{fig:nonepeg} does
not come from an e-graph, in fact it is $P(\Gamma)$ for the
non-e-realizable leo of Example~\ref{exm:notrealizableptdc} (see
Figure~\ref{fig:nonrealizableptdc}).

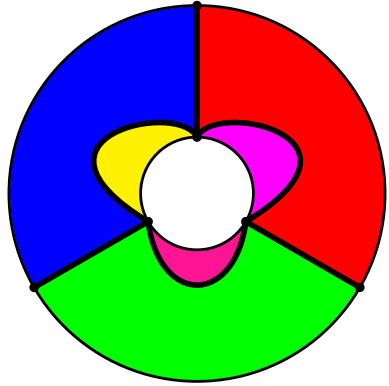
\begin{figure}[htbp]
  \centering
  \psset{unit=2.5}
  \begin{pspicture}(-1.1,-1.1)(1.1,1.1)
    \pscustom[linecolor=white,fillstyle=solid,fillcolor=blue]{%
      \psarc(0,0){1}{90}{210}
      \psline(-0.866025403784,-0.5)(-0.259807621135,-.15)
      \psbezier(-0.259807621135,-0.15)(-1,.3)(-.1,.5)(0,0.3)
      \psline(0,.3)(0,1)}
    \pscustom[linecolor=white,fillstyle=solid,fillcolor=red]{%
      \psarc(0,0){1}{-30}{90}
      \psline(0,1)(0,.3)
      \psbezier(0,0.3)(.1,.5)(1,.3)(0.2598,-0.15)
      \psline(0.866025403784,-0.5)(0.259807621135,-.15)}
    \pscustom[linecolor=white,fillstyle=solid,fillcolor=green]{%
      \psarc(0,0){1}{210}{330}
      \psline(0.866025403784,-0.5)(0.259807621135,-.15)
      \psbezier(0.259807621135,-0.15)(.2,-.6)(-.2,-.6)(-0.259807621135,-0.15)
      \psline(-0.259807621135,-0.15)(-0.866,-0.5)}
    \pscustom[linecolor=white,fillstyle=solid,fillcolor=yellow]{%
      \psarc(0,0){.3}{90}{210}
      \psbezier(-0.2598,-0.15)(-1,.3)(-.1,.5)(0,0.3)}
    \pscustom[linecolor=white,fillstyle=solid,fillcolor=deeppink]{%
      \psarc(0,0){.3}{210}{330}
      \psbezier(0.2598,-0.15)(.2,-.6)(-.2,-.6)(-0.2598,-0.15)}
    \pscustom[linecolor=white,fillstyle=solid,fillcolor=magenta]{%
      \psarc(0,0){.3}{330}{90}
      \psbezier(0,0.3)(.1,.5)(1,.3)(0.2598,-0.15)}
    \pscircle[linewidth=0.015](0,0){1}
    \pscircle[linewidth=0.015](0,0){.3}
  \rput(0,1){\rnode{1}{\psdot(0,0)}}
  \rput(-0.866,-0.5){\rnode{2}{\psdot(0,0)}}
  \rput(0.866,-0.5){\rnode{3}{\psdot(0,0)}}
  \rput(0,.3){\rnode{a}{\psdot(0,0)}}
  \rput(-0.2598,-0.15){\rnode{b}{\psdot(0,0)}}
  \rput(0.2598,-0.15){\rnode{c}{\psdot(0,0)}}
  \ncline[linewidth=.03]{1}{a}
  \ncline[linewidth=.03]{2}{b}
  \ncline[linewidth=.03]{3}{c}
  \psbezier[plotpoints=300,linewidth=.03](0,0.3)(-.1,.5)(-1,.3)(-0.2598,-0.15)
  \psbezier[plotpoints=300,linewidth=.03](0,0.3)(.1,.5)(1,.3)(0.2598,-0.15)
  \psbezier[plotpoints=300,linewidth=.03](0.2598,-0.15)(.2,-.6)(-.2,-.6)(-0.2598,-0.15)
  \end{pspicture}
  \caption{A peg that doesn't come from an e-graph}
  \label{fig:nonepeg}
\end{figure}

\begin{defn}
  \label{defn:epeg}
  An e-peg is a peg that is $P(\Gamma)$ for some e-graph $\Gamma$.
\end{defn}

\subsection{The dual of a peg}
\label{sec:dualpeg}

We can now interpret mind-body duality in terms of pegs.  Given a peg
there is a natural way to define its dual analogous to the way one
defines the dual of a Cellularly Embedded Graph.

\begin{defn}
  \label{defn:dualpeg}
  The \emph{dual peg} of a peg $\Gamma$ is the peg $\Gamma^{*}$
  obtained as follows:
  \begin{itemize}
  \item If $\Gamma$ is embedded in the surface $F$ then $\Gamma^{*}$
    is embedded is $F^{\intercal}$, that is, $F$ endowed with the
    opposite orientation.
  \item The vertices of $\Gamma^{*}$ are in one-to-one correspondence
    with the regions of $\Gamma$; when we draw $\Gamma^{*}$ we place
    its vertices on the arcs of the corresponding regions.
  \item The edges of $\Gamma^{*}$ are in one-to-one correspondence
    with the edges of $\Gamma$, the edge $e^{*}$ that corresponds to
    the edge $e$ connects the vertices of $\Gamma^{*}$ that correspond
    to the two regions of $\Gamma$ that $e$ lies in the boundary of.
  \end{itemize}
\end{defn}

For example the duals of the pegs in Figure~\ref{fig:nctree} are shown
with lighter purple lines.  The reason we consider $\Gamma^{*}$ to be
embedded in the oppositely oriented surface than that of $\Gamma$ will
become clear later in this section.

This notion of duality for non-crossing trees is implicit
in~\cite{GouldenYong}, although they do not explicitly consider the
dual as a non-crossing tree.  An explicit notion of dual for
non-crossing trees was defined in~\cite{Hernando1999}, however
properly speaking the notion defined there, though closely related to
ours, is not a real duality since it fails to be involutory; rather it
reflects the action of the Garside element (see
Section~\ref{sec:braid}).  See also Section~\ref{sec:futurenc}.

Since the boundary of each region of $\Gamma$ contains exactly one
arc and exactly one migt of $\Gamma$, and two regions share an edge if
and only if the corresponding migts do, the following proposition is
clear:
\begin{prop}
  \label{prop:dualpeg} If $\Gamma$ is an e-graph then, the peg that
  corresponds to $\Gamma^{*}$ is the dual of the peg that corresponds
  to $\Gamma$, i.e.
  $P \left( \Gamma^{*} \right) = P \left( \Gamma \right)^{*}$.
\end{prop}

\subsection{Relation with Cellularly Embedded Graphs}
\label{sec:cegrel}
There is a direct connection with the usual cellular embeddings to a
closed surface.  We will use the abbreviation \emph{ceg} to refer to a
cellularly embedded graph in a closed surface.

As noted the construction of a peg from a leo, given in the previous
subsection is directly analogous to the construction of a ceg from a
rotation system. Using that we can see that any peg can be
``completed'' to a ceg.
\begin{defn}
  \label{defn:hat}
  The \emph{closure} of a compact oriented surface with boundary $F$ is
  the surface $\widebar{F}$ obtained by $F$ by attaching $2$-cells along its
  boundary components.

  If $i\co \Gamma \to F$ is a peg then its \emph{closure} is the ceg
  $\bar{i}\co \Gamma \to \widebar{F}$ obtained by composing $i$ with
  the natural inclusion $F \hookrightarrow \widebar{F}$.  When we
  abuse notation and refer to a peg by its underlying graph $\Gamma$
  we will further abuse the notation by denoting its closure by
  $\widebar{\Gamma}$.
\end{defn}

As we have already remarked, the construction of the peg of an e-graph
is very similar to the construction of a ceg given a rotation system.
In fact the following is clear:
\begin{lem}
  \label{lem:rotofceg} The rotation system of $\widebar{\Gamma}$ is
  obtained by the leo of $\Gamma$ by ``completing'' each local edge
  order to a cyclic order.
\end{lem}




A natural question that arises at this point is what cegs are
completions of an e-peg?  An obvious necessary condition is that there
should be at least as many vertices as regions, but as we will see
shortly this is not sufficient, for example we will see that the ceg
in Figure~\ref{fig:unpeg} does not come from an e-graph.  In order to
answer this question we introduce the concept of the \emph{medial
  digraph} of a ceg.

The medial graph $\mathcal{M}(\Gamma)$ of a ceg $\Gamma$ is defined
(see for example~\cite{Archdeacon1992}) as the graph that has vertices
the edges of $\Gamma$, and an edge connecting any two vertices that
correspond to a pair of consecutive edges in the rotation system given
by the embedding.  Since we consider embeddings into \emph{oriented}
surfaces we observe that $\mathcal{M}(\Gamma)$ has a natural
orientation given by the cyclic order, so we will refer to
$\mathcal{M}(\Gamma)$ as the \emph{medial digraph} of the ceg
$\Gamma$.

The medial digraph $M$ of a ceg $\Gamma$ is embedded in the same
surface that the original graph is.  Indeed one can place the vertex
of the medial digraph in the middle of the corresponding edge of the
graph, and draw the edge connecting the vertices that correspond to
two consecutive edges inside the region that has the corresponding
edges in its boundary. If $\Gamma$ has $n$ vertices then the edges of
$\mathcal{M}(\Gamma)$ can be colored with $n$ colors corresponding to
the vertices of $\Gamma$, where all the edges corresponding to a
rotation around a vertex are colored by the color of that
vertex. Notice that if $e$ is an edge then at the corresponding vertex
of $\mathcal{M}$ there are edges of two different colors, say blue and
red and the cyclic order induced from the embedding is red-in,
blue-out, blue-in, red-out.

The regions of the medial digraph correspond to either vertices, or
regions of the original graph, and their boundaries consist of
(oriented) cycles.  The boundary cycle of a region that corresponds to
a vertex of the original graph is \emph{monochromatic}, and the
boundary cycle of a region that corresponds to a region of the
original graph is \emph{properly edge colored}, that is any two
adjacent edges are colored with different colors\footnote{Recall that
  we consider loopless graphs.  For graphs with loops this observation
  is not necessarily true.}.  These concepts are illustrated in
Figure~\ref{fig:sdk4} for the standard genus $0$ embedding of the
complete graph on four vertices.

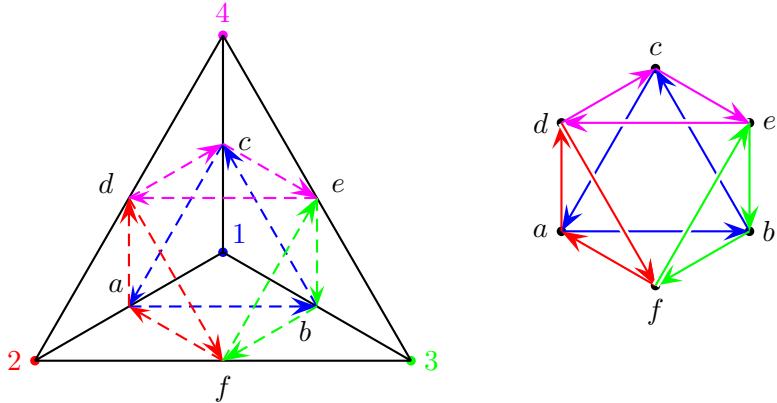
\begin{figure}[htbp]
  \centering
  \psset{unit=5}
  \begin{pspicture}(0,-.05)(2,1)
    \rput(0,0){
    \pnode(.5,0.288675134594813){0}
    \psdot[linecolor=blue](.5,0.288675134594813)
    \uput[45](.5,0.288675134594813){\blue $1$}
    \pnode(0,0){1}
    \psdot[linecolor=red](0,0)
    \uput[180](0,0){\red $2$}
    \pnode(1,0){2}
    \psdot[linecolor=green](1,0)
    \uput[0](1,0){\green $3$}
    \pnode(.5, 0.866025403784439){3}
    \psdot[linecolor=magenta](.5, 0.866025403784439)
    \uput[90](.5, 0.866025403784439){\magenta $4$}
    \ncline{0}{1}
    \nbput{$a$}
    \ncline{0}{2}
    \nbput{$b$}
    \ncline{0}{3}
    \nbput{$c$}
    \ncline{1}{2}
    \nbput{$f$}
    \ncline{1}{3}
    \naput{$d$}
    \ncline{2}{3}
    \nbput{$e$}
    \pnode(0.25,0.1443375672974){a1}
    \pnode(0.75,0.43301270189221){bc}
    \pnode(0.5,0.57735026918962){c1}
    \pnode(0.75,0.1443375672974){b1}
    \pnode(0.5,0){ab}
    \pnode(0.25,0.43301270189221){ac}
    \ncline[arrowsize=.05,linestyle=dashed,linecolor=blue]{->}{a1}{b1}
    \ncline[arrowsize=.05,linestyle=dashed,linecolor=blue]{->}{b1}{c1}
    \ncline[arrowsize=.04,linestyle=dashed,linecolor=blue]{->}{c1}{a1}
    \ncline[arrowsize=.05,linestyle=dashed,linecolor=red]{->}{a1}{ac}
    \ncline[arrowsize=.05,linestyle=dashed,linecolor=red]{->}{ac}{ab}
    \ncline[arrowsize=.04,linestyle=dashed]{->,linecolor=red}{ab}{a1}
    \ncline[arrowsize=.05,linestyle=dashed,linecolor=green]{->}{b1}{ab}
    \ncline[arrowsize=.05,linestyle=dashed,linecolor=green]{->}{ab}{bc}
    \ncline[arrowsize=.04,linestyle=dashed,linecolor=green]{->}{bc}{b1}
    \ncline[arrowsize=.05,linestyle=dashed,linecolor=magenta]{->}{ac}{c1}
    \ncline[arrowsize=.05,linestyle=dashed,linecolor=magenta]{->}{c1}{bc}
    \ncline[arrowsize=.04,linestyle=dashed,linecolor=magenta]{->}{bc}{ac}}
 \rput(1.15,.2){
   \pnode(0.25,0.1443375672974){a1}
   \psdot(0.25,0.1443375672974)
   \uput[180](0.25,0.1443375672974){$a$}
    \pnode(0.75,0.43301270189221){bc}
    \psdot(0.75,0.43301270189221)
    \uput[0](0.75,0.43301270189221){$e$}
    \pnode(0.5,0.57735026918962){c1}
    \psdot(0.5,0.57735026918962)
    \uput[90](0.5,0.57735026918962){$c$}
    \pnode(0.75,0.1443375672974){b1}
    \psdot(0.75,0.1443375672974)
    \uput[0](0.75,0.1443375672974){$b$}
    \pnode(0.5,0){ab}
    \psdot(.5,0)
    \uput[-90](.5,0){$f$}
    \pnode(0.25,0.43301270189221){ac}
    \psdot(0.25,0.43301270189221)
    \uput[180](0.25,0.43301270189221){$d$}
    \ncline[arrowsize=.05,linecolor=blue]{->}{a1}{b1}
    \ncline[arrowsize=.05,linecolor=blue]{->}{b1}{c1}
    \ncline[arrowsize=.05,linecolor=blue]{->}{c1}{a1}
    \ncline[arrowsize=.05,linecolor=magenta,,border=1pt]{->}{bc}{ac}
    \ncline[arrowsize=.05,linecolor=red]{->}{ab}{a1}
    \ncline[arrowsize=.05,linecolor=red]{->}{a1}{ac}
    \ncline[arrowsize=.05,linecolor=red,border=1pt]{->}{ac}{ab}
    \ncline[arrowsize=.05,linecolor=green]{->}{b1}{ab}
    \ncline[arrowsize=.05,linecolor=green,border=1pt]{->}{ab}{bc}
    \ncline[arrowsize=.05,linecolor=green]{->}{bc}{b1}
    \ncline[arrowsize=.05,linecolor=magenta]{->}{ac}{c1}
    \ncline[arrowsize=.05,linecolor=magenta]{->}{c1}{bc}}
  \end{pspicture}
  \caption{An embedded $K_4$ and its medial digraph}
  \label{fig:sdk4}
\end{figure}

The medial digraph of the dual of a ceg is isomorphic to the medial
digraph of the original graph, or its inverse depending on whether we
consider the dual ceg embedded in the oppositely oriented, or the same
surface as the original.  \emph{We make the convention that for a
  graph $\Gamma$ embedded in the closed surface $F$, it's dual
  $\Gamma^{*}$ is embedded in $F^{\intercal}$.}

If $\Gamma$ is a peg then the medial digraph of the ceg
$\widebar{\Gamma}$ is obtained from $\mathcal{M}\left( \Gamma \right)$
by adding an edge (of the appropriate color) from the last to the
first vertex of every chain in the PCD of $\Gamma$.  Removing those
edges from $\mathcal{M}\left( \widebar{\Gamma} \right)$, makes it
into a dag.

\begin{defn}
\label{defn:fas}
  A \emph{Feedback Arc Set} (FAS) in a digraph $D$ is a set of edges
  whose removal converts $D$ into a dag (see for
  example~\cite{bang2002diagraphs}, chapter 10.)

  A set of edges in an edge-colored digraph is called \emph{diverse}
  if it contains a representative of each color.
\end{defn}

Clearly a ceg $\Gamma$ that comes from a factorization, i.e. it is the
completion of an e-peg, has a diverse FAS, since each vertex of
$\Gamma$ gives a monochromatic cycle in
$\mathcal{M} \left( \Gamma \right)$. With these concepts in place we
can now state the following:

\begin{prop}
  \label{prop:real}
  The ceg $\Gamma$, with $n$ vertices, comes from a factorization
  exactly when its medial digraph has a FAS of cardinality $n$.
\end{prop}

\begin{proof}
  If $\Gamma$ is an e-peg, then $\mathcal{M}\left( \Gamma \right)$ is
  a dag contained in, and having exactly $n$ edges less than,
  $\mathcal{M}\left( \widebar{\Gamma} \right)$.

  Conversely, since the medial digraph
  $\mathcal{M}\left( \Gamma \right)$ of a ceg has $n$ monochromatic
  cycles, a FAS $S$ with $n$ edges will contain an edge in the cycle of
  each vertex. Removing $S$ therefore will give a leo structure on
  $\Gamma$, whose medial digraph is a dag, and whose completion will be
  the the rotation system of $\Gamma$.
\end{proof}

As an example consider the genus $0$ embedding of $K_4$ in
Figure~\ref{fig:sdk4}.  The set of edges $\{de, fa, bd, eb\}$ is a
diverse FAS for its medial digraph.  Taking a topological sort of the
resulting dag give us the factorization
$(1,2),(1,3),(2,4),(1,4),(2,3),(3,4)$ of $\id$, see
Figure~\ref{fig:sdk4el}.

\begin{figure}[htbp]
  \centering
  \psset{unit=5}
  \begin{pspicture}(2,1)
    \rput(0,.25){
   \pnode(0.25,0.1443375672974){a1}
   \psdot(0.25,0.1443375672974)
   \uput[180](0.25,0.1443375672974){$1$}
    \pnode(0.75,0.43301270189221){bc}
    \psdot(0.75,0.43301270189221)
    \uput[0](0.75,0.43301270189221){$6$}
    \pnode(0.5,0.57735026918962){c1}
    \psdot(0.5,0.57735026918962)
    \uput[90](0.5,0.57735026918962){$4$}
    \pnode(0.75,0.1443375672974){b1}
    \psdot(0.75,0.1443375672974)
    \uput[0](0.75,0.1443375672974){$2$}
    \pnode(0.5,0){ab}
    \psdot(.5,0)
    \uput[-90](.5,0){$5$}
    \pnode(0.25,0.43301270189221){ac}
    \psdot(0.25,0.43301270189221)
    \uput[180](0.25,0.43301270189221){$3$}
    \ncline[arrowsize=.05,linecolor=blue]{->}{a1}{b1}
    \ncline[arrowsize=.05,linecolor=blue]{->}{b1}{c1}
    \ncline[arrowsize=.05,linecolor=red]{->}{a1}{ac}
    \ncline[arrowsize=.05,linecolor=red,border=1pt]{->}{ac}{ab}
    \ncline[arrowsize=.05,linecolor=green]{->}{b1}{ab}
    \ncline[arrowsize=.05,linecolor=green,border=1pt]{->}{ab}{bc}
    \ncline[arrowsize=.05,linecolor=magenta]{->}{ac}{c1}
    \ncline[arrowsize=.05,linecolor=magenta]{->}{c1}{bc}}
  \rput(1.12,0){
    \pnode(.5,0.288675134594813){0}
    \psdot[linecolor=blue](.5,0.288675134594813)
    \uput[45](.5,0.288675134594813){\blue $1$}
    \pnode(0,0){1}
    \psdot[linecolor=red](0,0)
    \uput[180](0,0){\red $2$}
    \pnode(1,0){2}
    \psdot[linecolor=green](1,0)
    \uput[0](1,0){\green $3$}
    \pnode(.5, 0.866025403784439){3}
    \psdot[linecolor=magenta](.5, 0.866025403784439)
    \uput[90](.5, 0.866025403784439){\magenta $4$}
    \ncline{0}{1}
    \ncput*{\red $1$}
    \ncline{0}{2}
    \ncput*{\red $2$}
    \ncline{0}{3}
    \ncput*{\red $4$}
    \ncline{1}{2}
    \ncput*{\red $5$}
    \ncline{1}{3}
    \ncput*{\red $3$}
    \ncline{2}{3}
    \ncput*{\red $6$}}
  \end{pspicture}
  \caption{An e-labeling for the embedded $K_4$ in Figure~\ref{fig:sdk4}}
  \label{fig:sdk4el}
\end{figure}
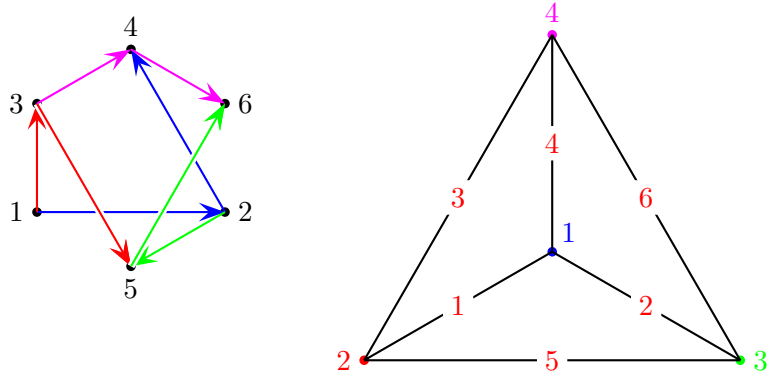

As an example of an non-peggable ceg consider the plane graph in
Figure~\ref{fig:unpeg}.  It's clear that its colored digraph does
not admit a diverse FAS so it does not come from an e-graph.

\begin{figure}[htbp]
  \centering
    \psset{unit=2}
  \begin{pspicture}(-.5,-.8)(5,1)
   \rput(1,0){%
    \pnode(-1,0){0}
    \psdot[linecolor=green](-1,0)
    \pnode(0,0){1}
    \psdot[linecolor=red](0,0)
    \pnode(1,0){2}
    \psdot[linecolor=blue](1,0)
    \ncline{0}{1}
    \ncput*{$a$}
    \ncline{1}{2}
    \ncput*{$c$}
    \ncarc[arcangle=45]{1}{2}
    \ncput*{$b$}
    \ncarc[arcangle=-45]{1}{2}
    \ncput*{$d$}}
    %
    \psset{unit=.6, arrowsize = .2}
     \rput(7.5,0){%
     \pnode(-1.0,0.0){a}
     \psdot(-1.0,0.0)
     \uput[180](-1.0,0.0){$a$}
     \pnode(0,1.0){b}
     \psdot(0,1.0)
     \uput[90](0,1.0){$b$}
     \pnode(0.0,0.0){c}
     \psdot(0.0,0.0)
     \uput[0](0.0,0.0){$c$}
     \pnode(0,-1.0){d}
     \psdot(0,-1.0)
     \uput[-90](0,-1.0){$d$}
     \nccircle[nodesep=3pt,angle=90,linecolor=green]{->}{a}{.7cm}
     \ncline[linecolor=red]{->}{b}{a}
     \ncarc[arcangle=25,linecolor=red]{->}{c}{b}
     \ncarc[arcangle=25, linecolor=red]{->}{d}{c}
     \ncline[linecolor=red]{->}{a}{d}
     \ncarc[arcangle=25, linecolor=blue]{->}{b}{c}
     \ncarc[arcangle=25, linecolor=blue]{->}{c}{d}
     \ncarc[arcangle=-85, linecolor=blue]{->}{d}{b}
}
  \end{pspicture}
  \caption{A non-peggable ceg and it's medial digraph}
  \label{fig:unpeg}
\end{figure}
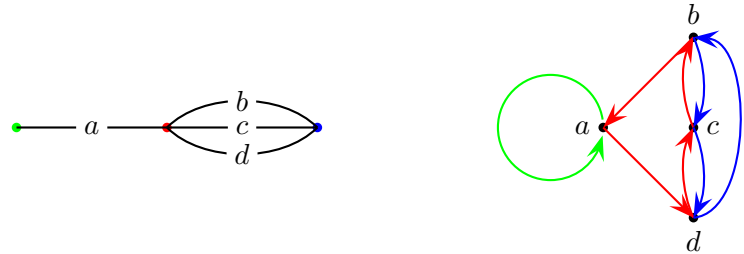

In general, if $\Gamma$ is peg then $\widebar{\Gamma^{*}}$ and
$\left( \widebar{\Gamma} \right)^{*}$ are not isomorphic as graphs.
For example if $\rho = (1\,3), (1\,2), (1\,3)$ then
$\rho^{*} = (1\,3), (2\,3), (1\,2)$, and the reader can easily check
that the duals of the completions of the corresponding pegs, which are
planar, are different.

However one can easily see the following:
\begin{thm}
  \label{thm:dualpegceg}
  If for a peg $\Gamma$ we have $\mu \left( \Gamma \right) = \id$ then
  $\left( \widebar{\Gamma} \right)^{*} = \widebar{\Gamma^{*}}$.
\end{thm}
\begin{proof}
  If $\mu \left( \Gamma \right) = \id$ then each boundary component of
  $P(\Gamma)$ contains exactly one point, and the closure of each
  region is a ``pinched'' annulus, and so the regions of
  $\widebar{\Gamma}$ are in one-to-one correspondence with the regions
  of $\Gamma$, and this correspondence obviously preserves incidence
  relations between regions and edges.  Since the vertices, edges, and
  regions of $\Gamma$ and $\widebar{\Gamma}$ are in one-to-one
  correspondence and those correspondences preserve incidence
  relations, it follows that
  $ \left( \,\widebar{\Gamma}\, \right)^{*} = \widebar{\Gamma^{*}}$.
\end{proof}

\begin{exm}\label{exm:sdk45}
  It is known that the complete graph $K_n$ admits a
  self-dual\footnote{In the sense that the underlying graph of the
    dual ceg is also complete.} embedding into a closed oriented
  surface, if and only if, $n \equiv 0 \text{ or } 1 \pmod{4}$,
  (see~\cite{Pengelley1975}).  These are exactly the degrees for which
  the complete graph has an even number of edges, and so it's possible
  to give $K_n$ an e-labeling with monodromy equal to the identity.
  We can ask then, whether the theory of e-pegs we developed can be
  used to prove this result. In this work we will give a simple proof
  for the cases $n=4$ and $n=5$ to illustrate the basic ideas.

  For $n=4$, start with the self-dual factorization (see
  Corollary~\ref{cor:bubblesd})
  $$\rho_{0;4} := (1\,2),(1\,2), (2\,3), (2\,3), (3\,4), (3\,4)$$
  and notice that if $\beta = \sigma_2 \sigma_4 \sigma_3^{-1}$ then
  $\rho_0\beta = (1\,2), (1\,3), (2\,4), (1\,4), (2\,3), (3\,4)$,
  whose associated graph is complete is the complete e-v-graph on the
  right side of Figure~\ref{fig:sdk4el}.  On the other hand we also
  have that
  $\rho_0\beta^{*} = (1\,2), (2\,3), (1\,4), (1\,3), (2\,4), (3\,4)$,
  whose associated graph is also complete.  Using
  Theorems~\ref{thm:dualpegceg} and~\ref{thm:dubr} we conclude that
  the e-labeling of $K_4$ in Figure~\ref{fig:sdk4el} gives a self-dual
  embedding of $K_{4}$ into the sphere.

  For $n=5$, one can start with the self-dual factorization
$$
\rho_{0;5} := (1\,2),(1\,2), (2\,3), (2\,3), (3\,4), (3\,4), (4\,5),
(4\,5), (1\,5), (1\,5)
$$
and observe that if
$\beta =
\delta_{2,10}\sigma_3^{-1}\sigma_5^{-1}\sigma_7^{-1}\sigma_9^{-1} $
then
$$
\rho_{0;5} \beta = (1\,2), (2\,5), (2\,3), (1\,3), (3\,4), (2\,4),
(4\,5), (3\,5), (1\,5), (1\,4)
$$
and
$$
\rho_{0;5} \beta^{*} = (1\,2), (1\,5), (3\,5), (2\,5), (2\,4), (2\,3),
(1\,3), (3\,4), (4\,5), (1\,4)
$$
both factorizations with complete associated e-v-graphs, thus giving
self-dual embeddings of $K_{5}$ into a torus.
\end{exm}

\begin{rem}\label{rem:sdk6}
  We also remark that even though $K_{6}$ does not admit self-dual
  embeddings into a closed surface, it can be self-dually pegged in
  surfaces with boundary.  Indeed the factorization
\begin{align*}
\rho_1 := & (1,2), (3,5), (1,3), (4,6), (2,4), (1,4), (5,6), (1,6), \\
          & (2,3), (2,5), (1,5), (3,4), (4,5), (2,6), (3,6)
\end{align*}
pegs $K_6$ into a torus with three boundary components, and the
factorization
\begin{align*}
\rho_2 := &(1,2), (3,6), (1,3), (4,5), (4,6), (2,4), (2,3), (1,5), \\
          &  (1,4), (5,6), (3,4), (2,5), (3,5), (1,6), (2,6)
\end{align*}
pegs it, self-dually, into a genus $2$ surface with one boundary
component.
\end{rem}

In general we can use mind-body duality to get self-dual embeddings of
graphs into closed surfaces by gluing a pair of dual pegs along their
common boundary.

\begin{defn}
  \label{defn:bdryconnsum} Let $P_1, P_2$ be two pegs, with
  $\mu\left( P_2 \right) = \mu \left( P_1 \right)^{-1}$ and
  $f\co \partial P_1 \to \partial P_2$ an orientation reversing
  homeomorphism, that maps the vertices of $P_1$ to the vertices of
  $P_{2}$.  The \emph{boundary connected sum of $P_1$ and $P_2$, with
    respect to $f$} is the ceg $P_1 \#_f P_2$ defined as follows:
  \begin{itemize}
  \item The surface of $P_1 \#_f P_2$ is the boundary connected sum of
    the surfaces of $P_1$ and $P_2$ with respect to $f$.
  \item $P_1 \#_f P_2$ has the same vertices as $P_1$ and
    $P_2$, and edges the union of the edges of $P_1$ and the edges
    of $P_2$.
  \end{itemize}
\end{defn}

Clearly $P_1 \#_f P_2$ is a ceg and each of its regions is obtained by
gluing a region of $P_1$ with the region of $P_2$ that has the same
boundary arc, along their common boundary.  Now we can prove:

 \begin{thm}
   \label{thm:bdryconnsum}
   For every peg $P$, the boundary connected sum $P \#_{\id} P^{*}$
   is self-dually embedded.
 \end{thm}
 \begin{proof}
   Let $C$ be the boundary connected sum. First since the orientations
   of $P$ and $P^{*}$ are opposite $\id$ is orientation reversing so
   $C$ is defined. Now observe that each region of $C$ is obtained by
   gluing a region of $P$ and a region of $P^{*}$ along their boundary
   arc. If we choose the vertex dual to a region to lie along that
   common arc, we see that $C^{*} = P^{*} \#_{\id} P$.
 \end{proof}

 \begin{exm}
   \label{exm:sddoublek6}
   Consider $\overrightarrow{K}_6$ the complete digraph with six
   vertices, any two of which are connected by a pair of opposite
   edges.  Theorem~\ref{thm:bdryconnsum} and Remark~\ref{rem:sdk6}
   imply that $\overrightarrow{K}_{6}$ admits self-dual embeddings
   into a surface of genus $4$.  Furthermore this is a digraph
   embedding in the sense of~\cite{BCMMC2002}, i.e. the boundary of
   every region is a directed cycle.
 \end{exm}

 We end this subsection by remarking that the theory of pegs is a
 refinement of the theory of cegs, more attuned to the graph
 theoretical properties of the graph. For example if a graph is
 cellularly embeddable into a closed surface then so is any graph
 homeomorphic to it. This is not the case with pegs. Indeed we
 have the following:

\begin{prop}
  \label{prop:pegsub} Any ceg has a subdivision that is the closure of
  an e-peg.
\end{prop}
\begin{proof}
  Let $\Gamma$ be a ceg.  Subdividing an edge of $\Gamma$, adds a
  pair of opposite arcs with a new color to the medial digraph of
  $\Gamma$.  Chose a FAS $S$ for $\mathcal{M}(\Gamma)$, and let $S'$
  be a minimal diverse subset of $S$ (it's clear that any FAS contains
  a diverse set).  Now for every arc $a$ in $S\setminus S'$ subdivide
  the edge of $\Gamma$ that corresponds to the beginning vertex of
  $a$, twice, and then replace, $a$ with a pair of the new edges, each
  going in opposite direction.  The resulting set $S''$ is a diverse
  FAS for the medial digraph of the subdivided ceg. See
  Figure~\ref{fig:subdiv}

  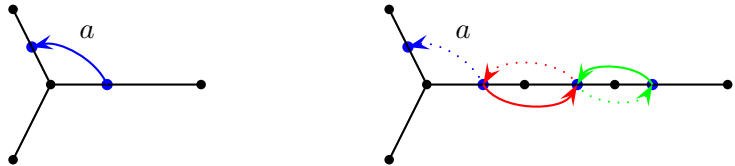
\begin{figure}[htbp]
  \centering
  \begin{pspicture}(8,2)
    \rput(.5,1){\rnode{v}{\psdot(0,0)}}
    \rput(0,2){\rnode{u}{\psdot(0,0)}}
    \rput(0,0){\rnode{w}{\psdot(0,0)}}
    \rput(2.5,1){\rnode{x}{\psdot(0,0)}}
    \rput(1.25,1){\rnode{m1}{\psdot[linecolor=blue,dotsize=.15](0,0)}}
    \rput(.25,1.5){\rnode{m2}{\psdot[linecolor=blue,dotsize=.15](0,0)}}
    \ncline{x}{v}
    \ncline{u}{v}
    \ncline{w}{v}
    \ncarc[arcangle=-40,arrowsize=.2,linecolor=blue]{->}{m1}{m2}
    \nbput{$a$}
    \rput(5,0){
    \rput(.5,1){\rnode{v}{\psdot(0,0)}}
    \rput(0,2){\rnode{u}{\psdot(0,0)}}
    \rput(0,0){\rnode{w}{\psdot(0,0)}}
    \rput(4.5,1){\rnode{x}{\psdot(0,0)}}
    \rput(1.25,1){\rnode{m1}{\psdot[linecolor=blue,dotsize=.15](0,0)}}
    \rput(2.5,1){\rnode{m3}{\psdot[linecolor=blue,dotsize=.15](0,0)}}
    \rput(3.5,1){\rnode{m4}{\psdot[linecolor=blue,dotsize=.15](0,0)}}
    \rput(.25,1.5){\rnode{m2}{\psdot[linecolor=blue,dotsize=.15](0,0)}}
    \ncline{x}{v}
    \ncline{u}{v}
    \ncline{w}{v}
    \ncarc[arcangle=-40,arrowsize=.2,linecolor=blue,linestyle=dotted]{->}{m1}{m2}
    \nbput{$a$}
    \ncarc[linestyle = dotted,arcangle=-60,arrowsize=.2,linecolor=red]{->}{m3}{m1}
    \ncarc[arcangle=60,arrowsize=.2,linecolor=red]{<-}{m3}{m1}
    \ncarc[linestyle = dotted, arcangle=-60,arrowsize=.2,linecolor=green]{->}{m3}{m4}
    \ncarc[arcangle=60,arrowsize=.2,linecolor=green]{<-}{m3}{m4}}
  \psdots(6.8,1)(8,1)
  \end{pspicture}
  \caption{Subdividing to get a diverse FAS}
  \label{fig:subdiv}
\end{figure}
\end{proof}

\subsection{On the genus and number of boundary components of e-pegs}
\label{sec:chib}

For an e-graph $\Gamma$, we see from Item~\ref{item:peggenus} of
Lemma~\ref{lem:eulermu} that the Euler characteristic of
$P \left( \Gamma \right)$ is determined by the number $b$ of disjoint
cycles of $\mu \left( \Gamma \right)$ and, of course, the Euler
characteristic of $\Gamma$.  A natural question that arises is: given
a graph $\Gamma$ what can we say about the values of $b$ that arise
from the different edge-labelings of $\Gamma$? The following
proposition provides two obvious necessary, but not in general
sufficient, conditions.

\begin{prop}
  \label{prop:bnecess}
  For every edge-labeling of $\Gamma$, the number $b$ of disjoint
  cycles of $\mu \left( \Gamma \right)$, or equivalently the number of
  boundary components of $P \left( \Gamma \right)$, satisfies:
  $$ b \le n, \quad \text{ and } \quad b \equiv \chi \left( \Gamma \right)\pmod{2}$$
\end{prop}

\begin{proof}
  Every boundary component of $P \left( \Gamma \right)$ contains at
  least one vertex of $\Gamma$, so $b \le n$.  Also, since the genus
  of a closed surface is an integer, by Item~\ref{item:peggenus} of
  Lemma~\ref{lem:eulermu} we infer that $b$ has the same parity as
  $\chi \left( \Gamma \right)$.
\end{proof}

To see that these conditions are not sufficient consider the graph
$\Gamma$ shown on the left side of Figure~\ref{fig:unpeg}: we have
$n=3$, and $\chi \left( \Gamma \right) = -1$ so the value $b=3$
satisfies both conditions, but, as one can easily see, for every
edge-labeling of $\Gamma$ the monodromy is a $3$-cycle, so that
$b=1$.

\begin{ques}
  \label{ques:chib}
  For what class of graphs are the conditions of
  Proposition~\ref{prop:bnecess} sufficient?
\end{ques}

We don't know the complete answer but we can prove that complete
graphs belong in that class:

\begin{thm}
    \label{thm:bsconv}
    For every complete graph $K_n$ all values of $b$ allowed by
    Proposition~\ref{prop:bnecess} occur.  Actually all possible
    conjugacy classes of $\mu$ consistent with
    Proposition~\ref{prop:bnecess} occur.
\end{thm}

We will use the following lemmata in the proof of
Theorem~\ref{thm:bsconv}.  The proofs of the lemmata are
straightforward.

\begin{lem}
  \label{lem:prom}
  Let $\Gamma$ be an e-graph of size $m$ with $\mu(\Gamma)$ a
  $d$-cycle, and let $\Gamma'$ be the graph obtained from $\Gamma$ by
  adding one new vertex $v$ and connecting it by an edge labeled $m+1$
  to a vertex $u$ that is moved by the cycle $\mu(\Gamma)$.  Then
  $\mu(\Gamma')$ is a $(d+1)$-cycle; namely if
  $\mu(\Gamma) = (\ldots,u',u,u'',\ldots)$ then
  $\mu(\Gamma') = (\ldots,u',v,u,u'',\ldots)$.
\end{lem}

\begin{defn}
  \label{def:muinv} Let $v$ be a vertex, and $e$ an edge of the
  e-graph $\Gamma$ not incident to $v$.  Then a \emph{$T$-operation
    from the vertex $v$ on the edge $e$} is the following
  modification: If $i$ is the label of $e$, increase all edge-labels
  greater than $i$ by $2$, and change $i$ to $i+1$.  Then add two new
  edges labeled $i$ and $i+2$ connecting $v$ to the endpoints of $e$.
  See Figure~\ref{fig:T}.
\end{defn}
\begin{figure}[htbp]
  \centering
  \begin{pspicture}(8,3)
    \pnode[0,1.5]{v}
    \psdot(v)
    \uput[180](v){$v$}
    \pnode[2,.5]{u}
    \psdot(u)
    \pnode[2,2.5]{w}
    \psdot(w)
    \ncline{u}{w}
    \ncput*{\blue $\small i$}
    \pnode[6,1.5]{v1}
    \psdot(v1)
    \uput[180](v1){$v$}
    \pnode[8,.5]{u1}
    \psdot(u1)
    \pnode[8,2.5]{w1}
    \psdot(w1)
    \ncline{u1}{w1}
    \ncput*{\blue $\small i+1$}
    \ncline{v1}{u1}
    \ncput*{\blue $\small i$}
    \ncline{v1}{w1}
    \ncput*{\blue $\small i+2$}
    \psline[linewidth=.08,arrowsize=.3]{->}(3,1.5)(5,1.5)
    \uput[90](4,1.5){$T$}
  \end{pspicture}
  \caption{The $T$-operation}
  \label{fig:T}
\end{figure}
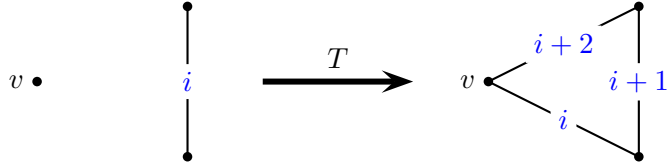
\begin{lem}
  \label{lem:muinv}
  The operation $T$ doesn't change the monodromy of the graph.
\end{lem}
\begin{proof}[Proof of Theorem~\ref{thm:bsconv}]
  We proceed by induction on $n$.  The theorem is obvious for
  $n=1,2,3$ and proven in Figure~\ref{fig:d=4} for $n=4$.  Assume
  then the theorem proven for all values less than $n$.

  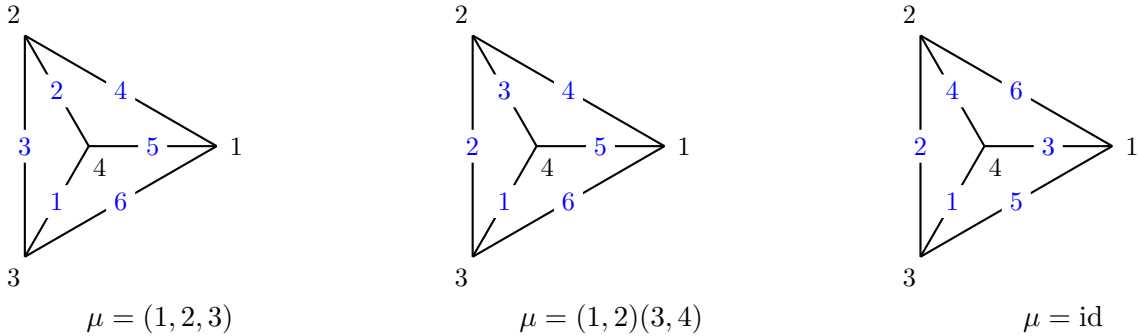
\begin{figure}[htbp]
    \centering
    \begin{pspicture}(-2.5,2.8)(12,7.5)
      \rput(-2,5.3){ \psset{unit=1.7}
        \begin{pspicture}(-1.1,-1.1)(1.1,1.1)
          \pnode[1.0,0.0]{1}
          \uput[0](1,0){\small $1$}
          \pnode[-0.5,0.86602540378444]{2}
          \uput[120](-0.5,0.86602540378444){\small $2$}
          \pnode[-0.5,-0.86602540378444]{3}
          \uput[-120](-0.5,-0.86602540378444){\small $3$}
          \pnode[0,0]{4}
          \uput[-60](0,0){\small $4$}
          \ncline{1}{2}
          \ncput*{\blue \small $4$}
          \ncline{1}{3}
          \ncput*{\blue \small $6$}
          \ncline{1}{4}
          \ncput*{\blue \small $5$}
          \ncline{2}{3}
          \ncput*{\blue \small $3$}
          \ncline{2}{4}
          \ncput*{\blue \small $2$}
          \ncline{3}{4}
          \ncput*{\blue \small $1$}
        \end{pspicture}}
      \rput(-1,3){$\mu=(1,2,3)$}
      \rput(4,5.3){ \psset{unit=1.7}
        \begin{pspicture}(-1.1,-1.1)(1.1,1.1)
          \pnode[1.0,0.0]{1}
          \uput[0](1,0){\small $1$}
          \pnode[-0.5,0.86602540378444]{2}
          \uput[120](-0.5,0.86602540378444){\small $2$}
          \pnode[-0.5,-0.86602540378444]{3}
          \uput[-120](-0.5,-0.86602540378444){\small $3$}
          \pnode[0,0]{4}
          \uput[-60](0,0){\small $4$}
          \ncline{1}{2}
          \ncput*{\blue \small $4$}
          \ncline{1}{3}
          \ncput*{\blue \small $6$}
          \ncline{1}{4}
          \ncput*{\blue \small $5$}
          \ncline{2}{3}
          \ncput*{\blue \small $2$}
          \ncline{2}{4}
          \ncput*{\blue \small $3$}
          \ncline{3}{4}
          \ncput*{\blue \small $1$}
        \end{pspicture}
      }
      \rput(5,3){$\mu = (1,2)(3,4)$}
      \rput(10,5.30){\psset{unit=1.7}
        \begin{pspicture}(-1.1,-1.1)(1.1,1.1)
          \pnode[1.0,0.0]{1}
          \uput[0](1,0){\small $1$}
          \pnode[-0.5,0.86602540378444]{2}
          \uput[120](-0.5,0.86602540378444){\small $2$}
          \pnode[-0.5,-0.86602540378444]{3}
          \uput[-120](-0.5,-0.86602540378444){\small $3$}
          \pnode[0,0]{4}
          \uput[-60](0,0){\small $4$}
          \ncline{1}{2}
          \ncput*{\blue \small $6$}
          \ncline{1}{3}
          \ncput*{\blue \small $5$}
          \ncline{1}{4}
          \ncput*{\blue \small $3$}
          \ncline{2}{3}
          \ncput*{\blue \small $2$}
          \ncline{2}{4}
          \ncput*{\blue \small $4$}
          \ncline{3}{4}
          \ncput*{\blue \small $1$}
        \end{pspicture}
      }
      \rput(11,3){$\mu = \id$}
    \end{pspicture}
    \caption{Proof of Theorem~\ref{thm:bsconv} for $n=4$}
    \label{fig:d=4}
  \end{figure}

  If $n \equiv 0 \pmod{4}$, then $\mu$ has to induce a partition into
  an even number of parts $b$.  If $b=n$ then $\mu = \id$.  To
  construct an e-labeling of $K_n$ with $\mu(K_n) = \id$, start with
  an e-labeling of $K_{n-4}$ with $\mu(K_{n-4}) = \id$ and an
  e-labeling of $K_4$ using labels,
  $\binom{n}{2}-6,\ldots,\binom{n}{2}$ with $\mu(K_4) = \id$. Then
  chose a partition of the vertices of $K_{n-4}$ into
  $\frac{1}{2} \binom{n-4}{2}$ pairs, and apply a $T$ operation from
  each vertex of $K_4$ on each of the edges of $K_{n-4}$ determined by
  those pairs.  The result is a $K_n$ with monodromy equal to $\id$.

  If $b<n$, and the corresponding partition is $k_1+\cdots + k_b =
  n$, with $k_b>1$, chose an e-labeled $K_{n-1}$ with monodromy of
  type $(k_1,\ldots,k_b-1)$ and add a new vertex $v$. Choose a
  vertex $u$ of $K_{n-1}$ that belongs in the $(k_{b}-1)$-cycle and
  partition the rest of the vertices of $K_{n-1}$ into pairs. Now
  add an edge connecting $v$ and $u$ and label it
  $\binom{b-1}{2}+1$, resulting in a graph with monodromy of type
  $(k_1,\ldots,k_b)$.  Then apply $T$ operations from $v$ to the
  edges determined by the partition into pairs of the remaining
  vertices of $K_{n-1}$.  At the end we get an e-labeled $K_d$
  whose monodromy has type $(k_1,\ldots,k_b)$.

  If $n \equiv 1 \pmod{4}$, then the possible values of $b$ are
  $1,3,\ldots,n$.  Let $n=k_1+\cdots+k_b$ be the partition induced
  by $\mu$. If $k_1$ is odd, by the inductive hypothesis we can find
  an e-labeled $K_{n-1}$ with monodromy of type
  $(k_2,\ldots,k_b,1,\ldots,1)$, where there are $k_1-1$ ones (if
  $k_1 = 1$ there are no ones).  Add a new vertex and connect it to
  the $k_1-1$ fixed points of $\mu(K_{n-1})$ with edges labeled
  $\binom{n-1}{2}+1,\ldots, \binom{n-1}{2} + k_1-1$.  The resulting
  graph has monodromy of type $(k_1,\ldots,k_b)$.  At this stage,
  there are an even number of vertices of $K_{n-1}$ not connected to
  the new vertex, so we can partition them into pairs and apply $T$
  moves from the new vertex to get an e-labeled $K_n$ with the
  desired monodromy.

  If $k_1$ is even, then we start with an e-labeled $K_{n-1}$ with
  monodromy of type $(k_2,\ldots,k_b-1,1,\ldots,1)$, where there are
  $k_1$ ones.  Add a new vertex $v$ and connect it to one of the
  vertices of the $(k_b-1)$-cycle with an edge labeled
  $\binom{n-1}{2}+1$, and to $k_1-1$ of the fixed points of
  $\mu(K_{n-1})$ by edges labeled
  $\binom{n-1}{2}+2,\ldots,\binom{n-1}{2}+k$.  The resulting graph
  has monodromy of type $(k_1,\ldots,k_b)$, and there are an even
  number of vertices not connected to the new vertex. Then we can
  use $T$ moves to get an e-labeled $K_n$ with the desired
  monodromy.

  If $n \equiv 2 \pmod{4}$, then $b=1,\ldots,n-1$.  If
  $n = k_1 + \ldots + k_b$ is a partition of $n$ then $b_K > 1$ and we
  can choose a $K_{n-1}$ with monodromy of type $(k_1,\ldots,k_b-1)$.
  Add a new vertex and connect it to a vertex of the $(k_b-1)$-cycle
  by a edge labeled $\binom{n-1}{2}+1$.  The result is a graph with
  monodromy of type $(k_1,\ldots,k_b)$ and an even number of edges
  unconnected to the new vertex.  So we can apply $T$-moves from the
  new vertex to get an e-labeled $K_n$ with the desired monodromy.

  If $n \equiv 3 \pmod{4}$, then $b=2,\ldots,n-1$.  Let
  $n = k_1 +\ldots+k_b$ be a partition of $n$.  If $k_1$ is odd,
  choose a $K_{n-1}$ with monodromy of type
  $(k_2,\ldots,k_b,1,\ldots,1)$, where there are $k_1-1$ ones.  Add a
  new vertex and connect it to the fixed points of $\mu(K_{n-1})$ by
  edges labeled $\binom{n-1}{2}+1,\ldots,\binom{n-1}{2}+k_1-1$.  This
  gives a graph with monodromy of the right type and an even number of
  vertices not connected to the new vertex.  So we can use $T$-moves
  to complete the proof.

  If $k_1$ is even, choose a $K_{n-1}$ with monodromy of type
  $(k_2,\ldots,k_b-1, 1,\ldots,1)$, where there are $k_1$ ones.  Add
  a new vertex and connect it to a vertex on the $k_b-1$-cycle by an
  edge labeled $\binom{n-1}{2} + 1$, and to $k_1-1$ fixed points of
  $\mu(K_{n-1})$ with edges labeled, $\binom{n-1}{2}+2, \ldots,
  \binom{n-1}{2}+k$. This gives a graph with the right monodromy and
  an even number of edges not connected to the new vertex.  So again
  we can use $T$-moves to complete the proof.  This completes the
  inductive step and the proof of the theorem.
\end{proof}

Item~\ref{item:peggenus} of Lemma~\ref{lem:eulermu} implies that in
general peggable embeddings have genera in the upper part of the genus
range of a graph.  These and related topics will be addressed in more
detail in the planned work~\cite{Apostolakis2018b}.

\subsection{Branched coverings interpretation}
\label{sec:brcov}

In this subsection we provide an alternative construction of the peg
associated with a factorization via the theory of branched coverings
of the two-dimensional disk $\mathbb{D}^{2}$.

Recall that the $B_m$-action on the free group comes from the fact
that $B_m$ is the \emph{mapping class group relative to the boundary}
of a $\mathbb{D}_{m}^2$, a disk with $m$ punctures, while the
fundamental group $\pi_1(\mathbb{D}_m^2)$ is a free group with $m$
generators, see for example~\cite{Birman1974},
\cite{TuraevKassel2008}, and~\cite{FBMprim2012}.  A factorization
$\rho$ can be thought of as a representation
$\pi_1(\mathbb{D}_m^2) \to \mathcal{S}_n$ and therefore gives a
covering of $\mathbb{D}^2$ branched over $m$ points, and the Hurwitz
action can be thought of as an action of $B_m$ to the set of
(equivalence classes of) branched coverings over the disk
$\mathbb{D}^2$.  For details on branched coverings
see~\cite{BernsEdm1979} and~\cite{Apos2003}, the later describes the
Hurwitz action in detail.

A free generating set of $\pi_1$ can be represented by a \emph{Hurwitz
  system} i.e. an ordered system of arcs connecting each branching
point to the basepoint of the disk, which we take to be on the
boundary circle, and whose interiors are pairwise disjoint.  An arc in
the system represents the loop that starts at the basepoint, follows
the arc to a small neighborhood of the puncture, goes once around the
puncture counterclockwise and returns to the basepoint along the arc.
To be concrete we consider the unit disk in $\mathbb{R}^{2}$, the
basepoint $b$ to be $(0,-1)$, while the $m$ branching points to be
equally spaced along the interval $[-1,1]$, and we take the
\emph{standard Hurwitz system} $h_0$, to be the $m$ straight line
segments connecting the basepoint to the branching points, this is
shown in the left side of Figure~\ref{fig:Hursys} in black. The $B_m$
action on $\mathbb{F}_m$ determines a left action on the set of
(isotopy classes of) Hurwitz systems, for example the image of $h_0$
under the action of the Garside element $\Delta_m$ is shown in the
left side of Figure~\ref{fig:Hursys} in green.

\begin{figure}[htbp]
  \centering
  \begin{pspicture}(-6,-2.5)(6,2.5)
    \rput(0,0){\includegraphics[scale=.55]{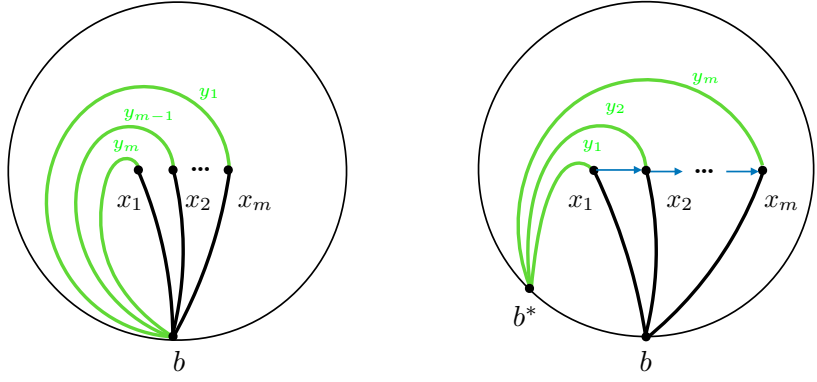}}
    \uput[-90](-3,-2.2){$b$}
    \uput[-90](3.2,-2.2){$b$}
    \uput[-90](1.6,-1.6){$b^{*}$}
    \uput[180](-3.3,-.4){\small $x_1$}
    \uput[0](-3.1,-.4){\small $x_2$}
    \uput[0](-2.4,-.4){\small $x_m$}
    \uput[90](-3.7,.1){\green \tiny $y_m$}
    \uput[90](-3.4,.5){\green \tiny $y_{m-1}$}
    \uput[90](-2.6,.8){\green \tiny $y_1$}
    \uput[180](2.7,-.4){\small $x_1$}
    \uput[0](3.3,-.4){\small $x_2$}
    \uput[0](4.6,-.4){\small $x_m$}
    \uput[90](2.5,0.1){\green \tiny $y_1$}
    \uput[90](2.8,.6){\green \tiny $y_2$}
    \uput[90](4,1){\green \tiny $y_m$}
  \end{pspicture}
  \caption{The standard Hurwitz system and its image under $\Delta$ (left), and its dual (right)}
  \label{fig:Hursys}
\end{figure}

A factorization in $\mathcal{S}_n$ gives a \emph{simple} branched
covering, that is a branched covering where the preimage of each
branching point has only one singular point and $n-2$ regular points
called \emph{pseudosingular}.  There is an explicit model for the
simple branched covering corresponding to a factorization $\rho$ (see
for example~\cite{BernsEdm1979} and~\cite{Apos2003}) and we briefly
recall that construction.  First choose a \emph{cut system} for
$\mathbb{D}_m^2$, consisting of $m$ segments connecting each branch
point to the boundary, to be explicit we take the $m$ vertical segments
in the upper half disk, and ``cut'' the disk open along this cut
system. Then take $n$ labeled copies of the disk (the sheets of the
covering) and glue them together along the cuts and according to the
monodromy sequence, that is, for $i=1,\ldots,m$ if the $i$-th
monodromy of $\rho$ is $(k\,l)$ we glue the $i$-th cuts of the $k$-th
and $l$-th sheet together, and ``sew''back together the $i$-th cut of
any other sheet.  The surface resulting from all these gluings is the
total space of the covering.  This construction is illustrated in
Figure~\ref{fig:deg3brcov} for the factorization $(1\,2),(2\,3)$.

\begin{defn}
  \label{defn:essentpreim}
  The \emph{essential preimage} of a Hurwitz system is defined to be
  the union of all the preimages of the arcs that contain a
  singular point.

  The \emph{reverse} of a Hurwitz system $h$ is the Hurwitz system
  $h^{\intercal}$ that has the same arcs as $h$ but in reverse order:
  the $i$-th arc of $h^{\intercal}$ is the $(m+1-i)$-th arc of $h$.

  The \emph{dual} of a Hurwitz system $h$ is the Hurwitz system
  $h^{*} := \left( \Delta_m h \right)^{\intercal}$ with basepoint slightly
  to the left of the basepoint of $h$.
\end{defn}

For example the dual of the standard Hurwitz system is shown in green
in the right side of Figure~\ref{fig:Hursys}.

We can now prove the following theorem:
\begin{thm}
  \label{thm:brpeg}
  Let $p\co F \to D^2$ be a simple branched covering. Then for any
  Hurwitz system $h$, the essential preimage of $h$ is a graph pegged
  in the total space of the covering.  Moreover the dual peg is the
  essential preimage of the dual Hurwitz system $h^{*}$.
\end{thm}

\begin{proof}
  Let $x$ be an arc of $h$, then $p^{-1}(x)$ consists of $n$ arcs, one
  in each sheet of the covering, Of these preimages, only two are
  essential, and they meet at the same singular point.  So $x$
  contributes to the essential preimage an arc connecting two points
  in the boundary of the total space, namely the preimages of $b$ in
  the two sheets that are glued together in the cut that corresponds
  to $x$. It follows that the essential preimage of $h$ is a graph
  $\Gamma$ embedded in $F$ with all the vertices in the boundary.  To
  see that $\Gamma$ is indeed properly embedded we first observe that
  the complement of the full preimage of $h$ consists of $n$ disjoint
  domains with one arc in their boundary, and interior homeomorphic to
  an open disc. Indeed, if we remove $h$ from $D^2$ we are left with
  a contractible set, with exactly one arc on the boundary and its
  interior homeomorphic to a disc. That set has $n$ homeomorphic
  preimages, that constitute the complement of the whole preimage of
  $h$ in $F$. A component of the complement of $\Gamma$ is obtained
  from a component of the complement of the full preimage, by possibly
  inserting some semi-open arcs, and it is easily seen that this
  still results in a contractible set with exactly one arc in its
  boundary and interior homeomorphic to an open disc.
  
  The fact that the essential preimage of $h^{*}$ is $\Gamma^{*}$ follows 
  from Theorem~\ref{thm:Hurw}.
\end{proof}

From the explicit construction of the branched covering
$p\co F \to \mathbb{D}^2$, associated with a factorization $\rho$
described above one can easily see that the essential preimage of the
standard Hurwitz system $h_0$ is a graph isomorphic to $\Gamma(\rho)$.
Indeed if the $k$-th monodromy of $\rho$ is $(i\,j)$, then the whole
preimage of the $k$-th arc of $h$ consists of $n-2$ ``short'' arcs
connecting the $n-2$ preimages of the basepoint $x_0$ in the sheets
with labels different than $i,j$ to the $n-2$ points in the
pseudo-singular locus above $x_{k}$, and one ``long'' arc that
consists of the two preimages of the arc that start at the sheets
labeled $i$ and $j$ and are glued together at the singular point above
$x_k$, see Figure~\ref{fig:deg3brcov}, for an example when $m=2$ and
$n=3$.  As a corollary then we have   an alternative
construction of the peg associated with an e-graph $\Gamma$, and its
mind-body dual:

\begin{thm}
  \label{thm:pegthroubrcov}
  For a factorization $\rho$, the peg $P \left( \Gamma(\rho) \right)$ is 
  the essential preimage of the standard Hurwitz system in the branched
  covering determined by $\rho$.  Furthermore it's dual peg is the 
  essential preimage of $h_0^{*}$.
\end{thm}

An example of the theorem is shown in Figure~\ref{fig:deg3brcov} for
the factorization $(1,2), (2,3)$ and its mind-body dual $(1,2), (1,3)$.

\begin{figure}[htbp]
  \centering
  \begin{pspicture}(-6.5,-4.4)(7,4.5)
    \rput(0,0){\includegraphics[scale=.6]{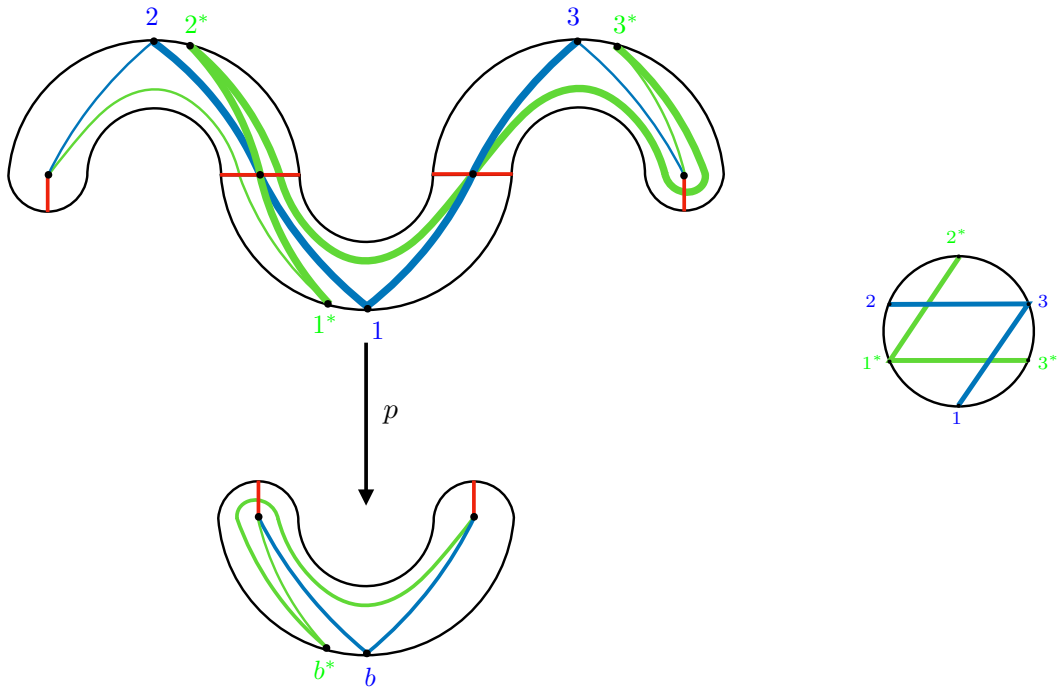}}
    \uput[0](-1.9,-.8){$p$}
    \uput[-90](-1.9,-4){\blue \small $b$}
    \uput[-90](-2.5,-3.9){\green \small $b^{*}$}
    \uput[-90](-1.8,.6){\blue \small $1$}
    \uput[-90](-2.5,.7){\green \small $1^{*}$}
    \uput[90](-4.8,4.2){\blue \small $2$}
    \uput[90](-4.2,4.1){\green \small $2^{*}$}
    \uput[90](.8,4.2){\blue \small $3$}
    \uput[90](1.5,4.1){\green \small $3^{*}$}
    \uput[-90](5.9,-.6){\blue \tiny $1$}
    \uput[180](5.1,-.1){\green \tiny $1^{*}$}
    \uput[180](5,.7){\blue \tiny $2$}
    \uput[90](5.9,1.3){\green \tiny $2^{*}$}
    \uput[0](6.8,.7){\blue \tiny $3$}
    \uput[0](6.8,-.1){\green \tiny $3^{*}$}
  \end{pspicture}
  \caption{The branched covering $(1\,2), (2\,3)$ and its dual $(1\,2),(1\,3)$.}
  \label{fig:deg3brcov}
\end{figure}

\begin{rem}
  \label{rem:medialbr}
  The medial digraph of the peg $P(\Gamma)$, for an e-graph $\Gamma$
  can also be interpreted via the associated branched covering.
  Consider the $n-1$ oriented intervals connecting $x_i$ to $x_{i+1}$,
  for $i=1,\ldots,n-1$, shown in blue on the right side of
  Figure~\ref{fig:Hursys}.  Clearly the subset of the preimage of
  those intervals, consisting only of those arcs that connect two
  singular points is a digraph isomorphic to the medial digraph of
  $\Gamma$.
 \end{rem}

The relation of e-labeled graphs with branched covering was observed
in~\cite{Arnold1996} (see also~\cite{LandoZvonkin2004}).  However they
consider branched coverings over the sphere $S^{2}$, by adding an
additional branched point with monodromy $\mu(\Gamma)^{-1}$, so that
in effect they obtain the ceg $\widebar{\Gamma}$.

If $\rho$ is a factorization of the identity permutation, it
determines not only a branched covering $p$ of the $2$-disk but also a
branched cover $\bar{p}$ of the sphere $2$-sphere $S^2$.  This is so
because the fundamental group of $S^2$ with $m$ punctures has a
presentation with $m$ generators $x_1, \ldots, x_m$, corresponding to
loops going around each puncture, and a single relation
$x_1\ldots x_m = \id$.  In that case the essential preimage of the
Hurwitz system under $\bar{p}$ is a ceg, and it is easy to see that it
is the completion of the peg obtained as the essential preimage under
$p$.  One can also easily see the following:

\begin{thm}
  \label{thm:brbdryconn}
  Let $\rho_1$ and $\rho_{2}$ be factorizations with
  $\mu(\rho_2) = \mu(\rho_1)^{-1}$, and $P_1$, $P_2$ the corresponding
  pegs.  Then $P_1\#_{\id} P_2$ is the ceg obtained from the
  concatenation of the factorizations $\rho_1\,\rho_{2}$ interpreted
  as a branched cover of the $2$-sphere.
\end{thm}

\section{On Cycles and Trees}
\label{sec:trees}

The research on this paper started by the author reading~\cite{GouldenYong}.
That paper provides a ``structural'' bijection $\phi$ from the set of minimal
factorizations of an $n$-cycle to the set of trees on $[n]$.  In this section
we give an exposition of that and related topics in light of the present work.

The set of all cyclic permutations on $[n]$ is denoted by
$\mathcal{C}_n$ and we choose the \emph{standard cyclic permutation}
to be $\zeta_0 := (n,n-1,\ldots,1)$.  $\mathcal{E}_n$ stands for the
set of edge-labeled trees (\emph{e-trees} for short), with $n$
vertices, $\mathcal{E}_n^{*}$ for the set of rooted e-trees with
$n$-vertices, $\mathcal{V}_n$ for the set of vertex labeled trees
(\emph{v-trees} for short) with $n$ vertices, and $\mathcal{L}_n$ for
the set of edge and vertex labeled trees (\emph{e-v-trees} for short).

We denote the set of all factorizations in $\mathcal{S}_n$ of length
$m$ by $\mathcal{F}_m (n)$. For a permutation
$\pi \in \mathcal{S}_{n}$ and $m\in \mathbb{N}$, the set of all
factorizations of $\pi$ as a product of $m$ transpositions is denoted
by $\mathcal{F}_m^{\pi}$.

\subsection{Bijections between $\mathcal{F}_n^{\zeta_0}$ and $\mathcal{V}_n$}
\label{sec:gy}

In~\cite{Denes1959}, D\'enes proved that

\begin{thm}
  \label{thm:Denes1}
  The graph of a factorization $\rho \in \mathcal{F}_{n-1}(n)$ is a
  tree if and only if $\mu(\rho) \in \mathcal{C}_{n}$.
\end{thm}
\begin{proof}
  If the graph $\Gamma(\rho)$ is a tree, then by Item~\ref{item:euler}
  of Lemma~\ref{lem:eulermu} we have that $P(\Gamma)$ is an orientable
  surface with Euler characteristic $1$, and therefore a disk.  It
  follows that all vertices of $\Gamma$ lie in a circle and therefore
  $\mu(\rho)$ is an $n$-cycle.

  Conversely, if $\mu(\rho)$ is an $n$-cycle, $P(\rho)$ is connected and therefore
  $\Gamma(\rho)$ is a connected graph with Euler characteristic $n - (n-1) = 1$.
  Therefore $\Gamma(\rho)$ is a tree.
\end{proof}

Using this result one can then establish the following:
\begin{thm} \cite{Denes1959}
  \label{thm:Denes2}
   For any $\zeta \in \mathcal{C}_n$ there is a bijection
   $f_{\zeta}\co \mathcal{F}^{\zeta} \to \mathcal{E}_n^{*}$.
\end{thm}
\begin{proof}
  Given $\rho \in \mathcal{F}^{\zeta}$ we obtain $f(\rho)$ by taking the
  corresponding graph in $\mathcal{L}_n$, declaring, say, $1$ to be the
  root and forgetting the v-labels.  To go back, starting from a rooted
  e-tree $t$, by Theorem~\ref{thm:Denes1}, $\mu(t)$ is an $n$-cycle in
  $\tau \in \mathcal{S}_{V}$ and we can label the vertices of $t$ so that the
  root is labeled $1$ and arranging so that the label of $\tau(v)$ is the image
  under $\zeta$  of the label of $v$ for all vertices $v$.
\end{proof}

Putting all these $(n-1)!$ bijections together one obtains:
\begin{thm} \cite{Denes1959}
  \label{thm:Denes3}
  There is a bijection
  $D\co \mathcal{C}_n \times \mathcal{E}_n^{*} \to \mathcal{L}_n$.
\end{thm}

\begin{cor}
  \label{cor:enumln}
   $\left| \mathcal{F}_{n-1}^{\zeta_0} \right| = n^{n-2}$
\end{cor}
\begin{proof}
  By Cayley's result, see e.g.~\cite{Moon1970}, $L_n$ has cardinality
  $(n-1)! n^{n-2}$ and since $\mathcal{C}_n$ has cardinality $(n-1)!$
  it follows by Theorem~\ref{thm:Denes3} that $\mathcal{E}_n^{*}$ has
  cardinality $n^{n-2}$ and therefore so does $\mathcal{F}_{n-1}^{\zeta_0}$
  by Theorem~\ref{thm:Denes2}.
\end{proof}

The fibers of $D$ are rather complicated, given a tree
$t \in \mathcal{V}_n$ there is an e-labeling of $t$ that makes it
being in the image of $\left\{ \zeta \right\}\times \mathcal{E}_n$ if
and only if $t$ is non-crossing with respect to $\zeta$.  If the
degree sequence of $t$ is $d_1, \ldots, d_n$ then there are
$\left( d_1\right)!\ldots \left( d_n\right)!$ such cycles
$\zeta$, see~\cite{Eden1962} and~\cite{DulPen1993}.  It follows that
one can not extract a bijection
$\mathcal{F}_{n-1}^{\zeta_0} \to \mathcal{V}_n$ from $D$, and D\'enes
in~\cite{Denes1959} posed the problem of finding such an explicit
bijection.

Based on the observations above, Moszkowski, in~\cite{Moszkowski1989},
realized that in order to solve the problem one has to \emph{delabel}
the vertices of the trees, and provided the following solution:

\begin{thm}
  \label{thm:slindingbij}
   There is a bijection $S\co \mathcal{F}_{n-1}^{\zeta_0} \to \mathcal{V}_n$.
\end{thm}
\begin{proof}
  $S$ is the composition of $f_{\zeta_0}$ with a bijection
  $\mathcal{E}_n \to \mathcal{V}_n$ defined by labeling the root of the
  tree $1$, increasing all e-labels by $1$ and then sliding each
  e-label to the vertex of that edge that is further away from
  the root.  Starting from a v-tree, we can recover the rooted e-tree, by
  declaring the vertex $1$ to be the root, and decreasing the vertex labels
  by $1$ and then sliding each v-label to its incident edge that is closest
  to the root; therefore $S$ is a bijection.
\end{proof}

We remark that the description above comes from~\cite{GouldenYong}, and is
also contained in~\cite{Poulalhon1997}.

Goulden and Yong in~\cite{GouldenYong} constructed a new bijection
$\phi\co \mathcal{F}_n^{\zeta_0} \to \mathcal{V}_n$, which with our
notation is defined by the following diagram:

    $$
  \begin{psmatrix}[mnode=R,colsep=2cm,rowsep=2cm]
    \mathcal{F}_n^{\zeta_0} & \mathcal{E}_n^{*}\\
    \mathcal{V}_n & \mathcal{E}_n^{*}
  \end{psmatrix}
  \psset{nodesep=0.3cm}
  \ncLine[arrowsize=.2]{->}{1,1}{1,2}
  \Aput{f_{\zeta_0}}
  \ncLine[arrowsize=.2,linestyle=dashed]{->}{1,1}{2,1}
  \Bput{\phi}
  \ncLine[arrowsize=.2]{->}{1,2}{2,2}
  \Aput{*}
  \ncLine[arrowsize=.2]{->}{2,2}{2,1}
  \Bput{S}
  $$

where $*\co \mathcal{E}_n^{*} \to \mathcal{E}_n^{*}$ is the
mind-body dual.  This bijection enjoys two ``structural'' properties
which we now explain.

For a transposition $\tau$ define it's
\emph{difference index} to be the cyclic distance of its moved
points, in other words, if $\tau = (s,t),\quad s<t$ then
$\delta(\tau) = \mathrm{min}\left\{ t-s, n-t+s \right\}$, and for
a factorization $\rho = \tau_1,\ldots, \tau_n$ define its \emph{difference
distribution} to be $\left(d_1,\ldots,d_n \right)$, where $d_i$ is the
number of elements in $\rho$ with difference index $i$.

For an edge $e$ of a tree $t\in \mathcal{V}_n$ define its
\emph{edge-deletion index} $\varepsilon(e)$ to be the minimum of the
orders of the two trees that result from $t$ after we delete $e$, and
the \emph{edge-deletion distribution} of $t$ to be $(a_1,\ldots,a_n)$
where $a_i$ is the number of edges of $t$ with edge-deletion index
$i$.

For a factorization $\rho \in \mathcal{F}_{n-1}^{\zeta_0}$ define its
\emph{degree distribution}, $d(\rho)$to be the degree distribution of the
associated e-v-tree. For a vertex $i$ of a v-tree $t\in \mathcal{V}_n$
define its \emph{maximal minimally increasing path} to be the path
obtained by starting at $i$ follow the edge that leads to the smallest
of its neighbors and then keep going to the smaller neighbor that is
larger than the vertex we are in, for as long as such neighbors exist.
The \emph{path-length distribution} of $t$ is the sequence
$l(t) = (l_1,\ldots, l_n)$ where $l_i$ is the number of vertices
of $t$ that have maximal minimal increasing path of length $i$.

\begin{thm}
  \label{thm:gy}
   The bijection $\phi\co \mathcal{F}_n^{\zeta_0} \to \mathcal{V}_n$ satisfies:
   \begin{enumerate}
   \item \label{item:deletionind} $\delta(\rho) = \varepsilon(\phi(\rho))$
   \item \label{item:degree} $d(\rho) = l(\phi(\rho))$
   \end{enumerate}
\end{thm}
\begin{proof}
  Both properties follow from the properties of mind-body duality.
  The first one from the fact that if $\tau$ has difference index $k$
  then $e^{*}$, the dual of the corresponding edge $e$ of the associated
  tree, will have deletion index $k$. The second one follows from the
  fact that the maximal minimally increasing paths are the image of
  migts under $S$.
\end{proof}

Item~\ref{item:deletionind} was observed in~\cite{GouldenYong},
Item~\ref{item:degree} is not explicitly stated there, although it is
implicit in the discussion.

\subsection{Mind-body duality for rooted e-graphs and flagged PCDs}
\label{sec:rootdual}

It turns out that for e-v-trees, mind-body duality at the level of
factorizations can be described via the functions $f_{\zeta}$
(see~\ref{thm:Denes2}) and mind-body duality at the level of rooted e-trees.
Mind-body duality can be extended to \emph{rooted} e-graphs in an
obvious way: if $\Gamma$ is a rooted e-graph with root $v_0$ then we
can take the root of $\Gamma^{*}$ to be $v_0^{*}$, i.e. the trail that
starts at $v_0$. Alternatively, the Hurwitz action of the (loop) braid
group extends to rooted e-graphs, the root just stays the
same, and we can use Theorem~\ref{thm:dualizer} to define the dual
of a rooted e-graph.  It's straightforward to check that these two
approaches define the same notion.  It's also easy to see that the
following diagram commutes, where $*$ stands for mind-body
duality of the relevant sets:

$$
\begin{psmatrix}[mnode=R,colsep=2cm,rowsep=2cm]
  \mathcal{F}_{n-1}^{\zeta_0} & \mathcal{E}_n^{*}\\
  \mathcal{F}_{n-1}^{\zeta_0^{-1}} & \mathcal{E}_n^{*}
\end{psmatrix}
\psset{nodesep=0.3cm}
\ncLine[arrowsize=.2]{->}{1,1}{1,2}
\Aput{f_{\zeta_0}}
\ncLine[arrowsize=.2]{->}{1,1}{2,1}
\Bput{*}
\ncLine[arrowsize=.2]{->}{1,2}{2,2}
\Aput{*}
\ncLine[arrowsize=.2]{->}{2,2}{2,1}
\Bput{f_{\zeta_0^{-1}}^{-1}}
$$

\begin{rem}
  \label{rem:rootcentral}
We note that one could use $f_{\zeta_0}^{-1}$ for the bottom arrow
to define duality between factorizations of the standard cycle,
i.e. one could define a duality
$\mathcal{F}_{n-1}^{\zeta_0} \to \mathcal{F}_{n-1}^{\zeta_0}$, to be 
the conjugate $f_{\zeta_0}^{-1}\circ * \circ f_{\zeta_0}$.  This
observation will be used in~\cite{Apostolakis2018a} to define a
``true'' duality for non-crossing trees, and study it's properties.
See also Remark~\ref{rem:sdevtrees} below.
\end{rem}

Clearly choosing a root for an e-graph, or more generally a leo, is
equivalent to choosing one of the chains of the PCD of its medial
digraph, and we can translate mind-body duality of rooted e-graphs (or
leos) in terms of PCDs with a distinguished chain.  we formalize this
in the following definition.

\begin{defn}
  \label{defn:flagged}  A \emph{flagged PCD} on a binary digraph $M$ is 
  a PCD $\mathcal{C}$ on $M$ endowed with a distinguished chain $f\in \mathcal{C}$
  called its \emph{flag}.

  For a chain $c$ in a PCD we use the notation $\alpha(f)$ (respectively
  $\omega(f)$) to denote the first (respectively last) vertex of $c$.

  The \emph{mind-body dual} of a flagged PCD
  $\left( \mathcal{C}, f \right)$ is the flagged PCD
  $\left( \mathcal{C}^{*}, f^{*} \right)$ where $f^{*}$ is defined as
  follows: $\alpha(f^{*}) = \alpha(f)$ and if $f$ is the only chain
  that starts at $\alpha(f)$ then $f^{*}$ is the only chain of
  $\mathcal{C}^{*}$ that starts at $\alpha(f)$, otherwise the first
  edge of $f^{*}$ is the outgoing edge incident at $\alpha(f)$ that
  does not belong to $f$, if no such edge exist then $f^{*}$ is a
  trivial chain, see Figure~\ref{fig:dualflag} where the flags of 
  the relevant PCDs are shown in red.
\end{defn}

\begin{figure}[ht]
  \centering
  \psset{arrowsize=0.15}
  \begin{pspicture}(-2.3,-.8)(11,5)
    \psdot(0,-.5)
    \psline(0,.5)(0,-.5)
    \psdot[linecolor=red](-.2,-.7)
    \psline[linecolor=blue](-.2,-.5)(-.2,.5)
    \psline{<->}(.5,0)(1.5,0)
    \uput[90](1,0){$\tiny *$}
    \rput(2,0){%
      \psdot(0,-.5)
    \psline(0,-.5)(0,.5)
    \psdot[linecolor=blue](.2,-.7)
    \psline[linecolor=red](.2,-.5)(.2,.5)}
    \rput(5.8,0){%
    \psdot(0,-.5)
     \psline(-.6,.5)(0,-.5)(.6,.5)      
     \psline[linecolor=red](-.8,.5)(-.2,-.5)
     \psline[linecolor=blue](.8,.5)(.2,-.5)}
    \psline{<->}(7,0)(8,0)
    \uput[90](7.5,0){$\tiny *$}
    \rput(9.2,0){%
    \psdot(0,-.5)
     \psline(-.6,.5)(0,-.5)(.6,.5)      
     \psline[linecolor=blue](-.8,.5)(-.2,-.5)
     \psline[linecolor=red](.8,.5)(.2,-.5)}
   \rput(-3,2.25){%
     \psdot(0,1)
     \psline(0,0)(0,1)
    \psdot[linecolor=red](-.2,1.2)
    \psline[linecolor=blue](-.2,0)(-.2,1)}
    \psline{<->}(-2.5,2.75)(-1.5,2.75)
    \uput[90](-2,2.75){$\tiny *$}
   \rput(-1,2.25){%
     \psdot(0,1)
     \psline(0,0)(0,1)
    \psdot[linecolor=red](-.2,1.2)
    \psline[linecolor=blue](-.2,0)(-.2,1)}
  \rput(1.5,2){%
  \psdot(0,1)
  \psline(0,2)(0,0)
  \psline[linecolor=blue](.2,2)(.2,0)
  \psdot[linecolor=red](-.2,1)}
  \psline{<->}(2,3)(3,3)
  \uput[90](2.5,3){$\tiny *$}
  \rput(3.5,2){%
  \psdot(0,1)
  \psline(0,2)(0,0)
  \psline[linecolor=blue](.2,.9)(.2,0)
  \psline[linecolor=red](.2,1.1)(.2,2)}
  \rput(7,2){%
    \psdot(0,1)
    \psline(-.6,2)(0,1)(0,0)      
    \psline(.6,2)(0,1)
    \psline[linearc=.25,linecolor=blue](-.2,0)(-.2,1)(-.8,2)
    \psline[linecolor=red](.8,2)(.2,1)}
  \psline{<->}(8,3)(9,3)
  \uput[90](8.5,3){$\tiny *$}
  \rput(10,2){%
    \psdot(0,1)
    \psline(-.6,2)(0,1)(0,0)      
    \psline(.6,2)(0,1)
    \psline[linecolor=red](-.8,2)(-.2,1)
    \psline[linearc=.25,linecolor=blue](.8,2)(.2,1)(.2,0)}
  \end{pspicture}
  \caption{The flag of the dual of a flagged PCD}
  \label{fig:dualflag}
\end{figure}
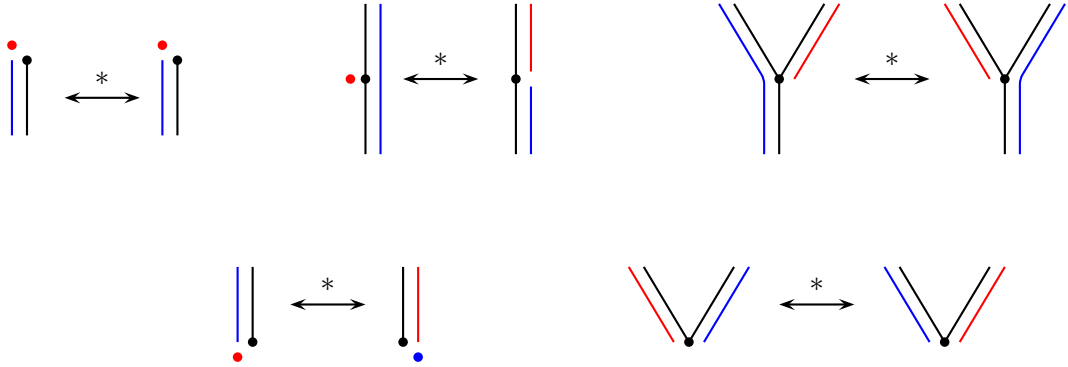

By Lemma~\ref{lem:leochi} the underlying graph of the medial digraph of
an e-tree is a tree.

\begin{defn}
  \label{defn:ditree}
   A \emph{ditree} is a digraph whose underlying graph is a tree.
\end{defn}

We finish this subsection by mentioning that the enumeration of
factorizations, or equivalently simple branched coverings, is a well
established area of research with connections to Geometric Topology,
Algebraic Geometry, and Mathematical Physics, and there are many
results with bijective proofs for various ``Hurwitz numbers''.  See
for example~\cite{DPS2014}.  The author hopes that the notion of
mind-body duality introduced in this paper will help provide explicit
bijections explaining known enumerative coincidences, as well
discovering new ones.

\subsection{Self-dual e-trees}
\label{sec:zigzag}

In every context where an interesting concept of duality is defined, a
natural question that arises is whether there are any self-dual
objects.  The question for general graphs will be studied in further
projects, in this subsection we concentrate on trees. In the context
of mind-body duality it is obvious that there are no self-dual
factorizations or e-v-trees\footnote{with the trivial exceptions of
  $n=1,2$ where the unique objects are obviously self-dual.} since the
monodromy of the dual is the inverse of the monodromy of the original
object\footnote{See however Remark~\ref{rem:rootcentral}
and Remark~\ref{rem:sdevtrees} at the end of this section}.
For e-trees the question is meaningful and has an interesting answer.

\begin{defn}
  \label{defn:sdetree}
   An e-tree $t$ is called \emph{self-dual} if $t^{*} = t$.
\end{defn}

\begin{defn}
  \label{defn:zigzag}

   For an integer $n$, the \emph{updown ditree with $n$ vertices} is the ditree
   with vertices $x_1,\ldots,x_n$ and an edge from $x_{2i-1}$ to $x_{2i}$, and
   an edge from $x_{2i}$ to $x_{2i+1}$ for each $i=1,\ldots, \lfloor n/2 \rfloor$.

   The \emph{downup ditree with $n$ vertices} is the ditree with
   vertices $x_1,\ldots,x_n$ and an edge from $x_{2i}$ to $x_{2i-1}$,
   and an edge from $x_{2i+1}$ to $x_{2i}$ for each
   $i=1,\ldots, \lfloor n/2 \rfloor$.

   A \emph{zigzag ditree} is an updown or downup ditree.
\end{defn}

For even $n$ the updown and downup ditrees are isomorphic, while for
odd $n$ there are two (inverse to each other) zigzag ditrees.  The
zigzag ditrees are the Hasse diagrams of the zigzag (or \emph{fence})
posets.  See~\cite{StanleyVolI}, page 157, Exercise 23 in Chapter 3.

With this definitions in place, we can now prove:
\begin{thm}
  \label{thm:zigzag}
   An e-tree $t$ is self-dual if and only if its medial ditree $\mathcal{M}(t)$
   is a zigzag ditree.
\end{thm}
\begin{proof} The medial ditree of an e-tree is a topsorted binary ditree with
a PCD. It is clear that the set of zigzag ditrees coincides with the set
of binary ditrees with no internal vertices.  So we need to prove that an e-tree
is self-dual if and only if its medial ditree has no internal vertices.

By Definition~\ref{defn:dualofpcd}, it follows that if
$\mathcal{M}(t)$ has no internal vertices then $\mathcal{M}(t)$ is
self-dual and hence, by Theorem~\ref{thm:dualofpcd}, $t$ is self dual.

Conversely, since a PCD and its dual, differ at every internal vertex, and
$\mathcal{M}(t)$, being topsorted, has labeled vertices, it follows that if
$\mathcal{M}(t)$ has internal vertices then $t$ is not self-dual.
\end{proof}

It is well known that the number of topological sorts of a zigzag
ditree with $n$ vertices is given by the $n$-th Euler up/down number.
This is sequence \href{https://oeis.org/A000111}{A000111} in The
On-Line Encyclopedia of Integer Sequences~\cite{oeis}. This sequence
enumerates (among other things) the set of \emph{alternating
  permutations}, see~\cite{Andre1881}.

The fact that \href{https://oeis.org/A000111}{A000111} enumerates the
set of topological sorts of a given zigzag ditree is not enough to
conclude that it also enumerates self-dual e-trees, because of the
presence of automorphisms.  Indeed, while for even $n$ the zigzag
ditree has no non-trivial automorphisms, for odd $n$ there is a
non-trivial automorphism of order $2$.  However for odd $n$ there are
two, inverse to each other, zigzag ditrees and that introduces a factor
of $2$ that compensates.  So we have:

\begin{cor}
  \label{cor:zigzag}
   The number of self-dual e-trees with $n$ vertices is equal to the $(n-1)$-th
   Euler up/down number.
\end{cor}

The self-dual e-trees for $n=3,4,5$ are shown in
Figure~\ref{fig:sdetrees}, on the left side we have the zigzag
ditree(s), on the center all possible topsorted zigzag ditrees, and on
the right the corresponding e-graphs.

\begin{figure}[htbp]
  \centering
  \psset{unit=1.2}
\begin{pspicture}(-1.5,-1.6)(10,15.6)
  \rput(-1,-.5){$\mathbf{n=3}$}
  \rput(-1,3){$\mathbf{n=4}$}
  \rput(-1,10.75){$\mathbf{n=5}$}
  \rput(.5,-1){\rnode{0}{\psdot(0,0)}}
  \rput(.5,0){\rnode{1}{\psdot(0,0)}}
  \ncline[arrowsize=.2]{->}{0}{1}
  \rput(4,-1){%
    \rput(0,0){\rnode{0}{\psdot(0,0)}}
    \uput[-90](0,0){$1$}
    \rput(0,1){\rnode{1}{\psdot(0,0)}}
    \uput[90](0,1){$2$}
    \ncline[arrowsize=.15]{->}{0}{1}}
  \rput(8,-.75){%
    \psset{unit=.9}
    \rput(0,1){\rnode{1}{\psdot(0,0)}}
    \rput(-0.866025,-0.5){\rnode{2}{\psdot(0,0)}}
    \rput(0.866025,-0.5){\rnode{3}{\psdot(0,0)}}
    \ncline{1}{2}
    \ncput*{\blue {\small $2$}}
    \ncline{2}{3}
    \ncput*{\blue {\small $1$}}}
  \rput(0,1.5){\rnode{0}{\psdot(0,0)}}
  \rput(.5,2.5){\rnode{1}{\psdot(0,0)}}
  \rput(1,1.5){\rnode{2}{\psdot(0,0)}}
  \ncline[arrowsize=.15]{->}{0}{1}
  \ncline[arrowsize=.15]{->}{2}{1}

  \rput(3.5,-.5){
    \rput(0,2){\rnode{0}{\psdot(0,0)}}
    \uput[-90](0,2){$1$}
    \rput(.5,3){\rnode{1}{\psdot(0,0)}}
    \uput[90](.5,3){$3$}
    \rput(1,2){\rnode{2}{\psdot(0,0)}}
    \uput[-90](1,2){$2$}
  \ncline[arrowsize=.15]{->}{0}{1}
  \ncline[arrowsize=.15]{->}{2}{1}}

\rput(8,4){%
\psset{unit=.8}
\rput(0, 1.0){\rnode{1}{\psdot(0,0)}}
\rput(-1.0, 0){\rnode{2}{\psdot(0,0)}}
\rput(0, -1.0){\rnode{3}{\psdot(0,0)}}
\rput(1.0, 0){\rnode{4}{\psdot(0,0)}}
\ncline{1}{2}
\ncput*{\blue {\small $2$}}
\ncline{1}{3}
\ncput*{\blue {\small $1$}}
\ncline{3}{4}
\ncput*{\blue {\small $3$}}}

\rput(8,2){%
\psset{unit=.8}
\rput(0, 1.0){\rnode{1}{\psdot(0,0)}}
\rput(-1.0, 0){\rnode{2}{\psdot(0,0)}}
\rput(0, -1.0){\rnode{3}{\psdot(0,0)}}
\rput(1.0, 0){\rnode{4}{\psdot(0,0)}}
\ncline{1}{2}
\ncput*{\blue {\small $1$}}
\ncline{1}{3}
\ncput*{\blue {\small $3$}}
\ncline{3}{4}
\ncput*{\blue {\small $2$}}}

\rput(0,1.5){
  \rput(0,3){\rnode{0}{\psdot(0,0)}}
  \rput(.5,2){\rnode{1}{\psdot(0,0)}}
  \rput(1,3){\rnode{2}{\psdot(0,0)}}
  \ncline[arrowsize=.15]{->}{1}{0}
  \ncline[arrowsize=.15]{->}{1}{2}

  \rput(3.5,0){
  \rput(0,3){\rnode{0}{\psdot(0,0)}}
  \uput[90](0,3){$2$}
  \rput(.5,2){\rnode{1}{\psdot(0,0)}}
  \uput[-90](.5,2){$1$}
  \rput(1,3){\rnode{2}{\psdot(0,0)}}
  \uput[90](1,3){$3$}
  \ncline[arrowsize=.15]{->}{1}{0}
  \ncline[arrowsize=.15]{->}{1}{2}}}
\rput(0,1.7){%
  \rput(0,8.5){\rnode{0}{\psdot(0,0)}}
  \rput(.5,9.5){\rnode{1}{\psdot(0,0)}}
  \rput(1,8.5){\rnode{2}{\psdot(0,0)}}
   \rput(1.5,9.5){\rnode{3}{\psdot(0,0)}}
  \ncline[arrowsize=.15]{->}{0}{1}
  \ncline[arrowsize=.15]{->}{2}{1}
  \ncline[arrowsize=.15]{->}{2}{3}}

  \rput(3,6.25){\rnode{0}{\psdot(0,0)}}
  \uput[-90](3,6.25){$1$}
  \rput(3.5,7.25){\rnode{1}{\psdot(0,0)}}
  \uput[90](3.5,7.25){$3$}
  \rput(4,6.25){\rnode{2}{\psdot(0,0)}}
  \uput[-90](4,6.25){$2$}
   \rput(4.5,7.25){\rnode{3}{\psdot(0,0)}}
  \uput[90](4.5,7.25){$4$}
  \ncline[arrowsize=.15]{->}{0}{1}
  \ncline[arrowsize=.15]{->}{2}{1}
  \ncline[arrowsize=.15]{->}{2}{3}

\rput(0.5,2){%
  \rput(2.5,6.25){\rnode{0}{\psdot(0,0)}}
  \uput[-90](2.5,6.25){$1$}
  \rput(3,7.25){\rnode{1}{\psdot(0,0)}}
  \uput[90](3,7.25){$4$}
  \rput(3.5,6.25){\rnode{2}{\psdot(0,0)}}
  \uput[-90](3.5,6.25){$2$}
   \rput(4,7.25){\rnode{3}{\psdot(0,0)}}
  \uput[90](4,7.25){$3$}
  \ncline[arrowsize=.15]{->}{0}{1}
  \ncline[arrowsize=.15]{->}{2}{1}
  \ncline[arrowsize=.15]{->}{2}{3}}

\rput(0.5,4){%
  \rput(2.5,6.25){\rnode{0}{\psdot(0,0)}}
  \uput[-90](2.5,6.25){$2$}
  \rput(3,7.25){\rnode{1}{\psdot(0,0)}}
  \uput[90](3,7.25){$3$}
  \rput(3.5,6.25){\rnode{2}{\psdot(0,0)}}
  \uput[-90](3.5,6.25){$1$}
   \rput(4,7.25){\rnode{3}{\psdot(0,0)}}
  \uput[90](4,7.25){$4$}
  \ncline[arrowsize=.15]{->}{0}{1}
  \ncline[arrowsize=.15]{->}{2}{1}
  \ncline[arrowsize=.15]{->}{2}{3}}

\rput(0.5,6){%
  \rput(2.5,6.25){\rnode{0}{\psdot(0,0)}}
  \uput[-90](2.5,6.25){$2$}
  \rput(3,7.25){\rnode{1}{\psdot(0,0)}}
  \uput[90](3,7.25){$4$}
  \rput(3.5,6.25){\rnode{2}{\psdot(0,0)}}
  \uput[-90](3.5,6.25){$1$}
   \rput(4,7.25){\rnode{3}{\psdot(0,0)}}
  \uput[90](4,7.25){$3$}
  \ncline[arrowsize=.15]{->}{0}{1}
  \ncline[arrowsize=.15]{->}{2}{1}
  \ncline[arrowsize=.15]{->}{2}{3}}

\rput(0.5,8){%
  \rput(2.5,6.25){\rnode{0}{\psdot(0,0)}}
  \uput[-90](2.5,6.25){$3$}
  \rput(3,7.25){\rnode{1}{\psdot(0,0)}}
  \uput[90](3,7.25){$4$}
  \rput(3.5,6.25){\rnode{2}{\psdot(0,0)}}
  \uput[-90](3.5,6.25){$1$}
   \rput(4,7.25){\rnode{3}{\psdot(0,0)}}
  \uput[90](4,7.25){$2$}
  \ncline[arrowsize=.15]{->}{0}{1}
  \ncline[arrowsize=.15]{->}{2}{1}
  \ncline[arrowsize=.15]{->}{2}{3}}

\rput(8,6.6){%
\psset{unit=.8}
\rput(0, 1.0){\rnode{1}{\psdot(0,0)}}
\rput(-0.9510565162951535, 0.3090169943749475){\rnode{2}{\psdot(0,0)}}
\rput(-0.5877852522924732, -0.8090169943749473){\rnode{3}{\psdot(0,0)}}
\rput(0.5877852522924729, -0.8090169943749476){\rnode{4}{\psdot(0,0)}}
\rput(0.9510565162951536, 0.3090169943749472){\rnode{5}{\psdot(0,0)}}
\ncline{2}{3}
\ncput*{\blue {\small $1$}}
\ncline{1}{3}
\ncput*{\blue {\small $3$}}
\ncline{1}{4}
\ncput*{\blue {\small $2$}}
\ncline{4}{5}
\ncput*{\blue {\small $4$}}}

\rput(8,8.6){%
\psset{unit=.8}
\rput(0, 1.0){\rnode{1}{\psdot(0,0)}}
\rput(-0.9510565162951535, 0.3090169943749475){\rnode{2}{\psdot(0,0)}}
\rput(-0.5877852522924732, -0.8090169943749473){\rnode{3}{\psdot(0,0)}}
\rput(0.5877852522924729, -0.8090169943749476){\rnode{4}{\psdot(0,0)}}
\rput(0.9510565162951536, 0.3090169943749472){\rnode{5}{\psdot(0,0)}}
\ncline{2}{3}
\ncput*{\blue {\small $1$}}
\ncline{1}{3}
\ncput*{\blue {\small $4$}}
\ncline{1}{4}
\ncput*{\blue {\small $2$}}
\ncline{4}{5}
\ncput*{\blue {\small $3$}}}

\rput(8,10.6){%
\psset{unit=.8}
\rput(0, 1.0){\rnode{1}{\psdot(0,0)}}
\rput(-0.9510565162951535, 0.3090169943749475){\rnode{2}{\psdot(0,0)}}
\rput(-0.5877852522924732, -0.8090169943749473){\rnode{3}{\psdot(0,0)}}
\rput(0.5877852522924729, -0.8090169943749476){\rnode{4}{\psdot(0,0)}}
\rput(0.9510565162951536, 0.3090169943749472){\rnode{5}{\psdot(0,0)}}
\ncline{2}{3}
\ncput*{\blue {\small $2$}}
\ncline{1}{3}
\ncput*{\blue {\small $3$}}
\ncline{1}{4}
\ncput*{\blue {\small $1$}}
\ncline{4}{5}
\ncput*{\blue {\small $4$}}}

\rput(8,12.6){%
\psset{unit=.8}
\rput(0, 1.0){\rnode{1}{\psdot(0,0)}}
\rput(-0.9510565162951535, 0.3090169943749475){\rnode{2}{\psdot(0,0)}}
\rput(-0.5877852522924732, -0.8090169943749473){\rnode{3}{\psdot(0,0)}}
\rput(0.5877852522924729, -0.8090169943749476){\rnode{4}{\psdot(0,0)}}
\rput(0.9510565162951536, 0.3090169943749472){\rnode{5}{\psdot(0,0)}}
\ncline{2}{3}
\ncput*{\blue {\small $2$}}
\ncline{1}{3}
\ncput*{\blue {\small $4$}}
\ncline{1}{4}
\ncput*{\blue {\small $1$}}
\ncline{4}{5}
\ncput*{\blue {\small $3$}}}

\rput(8,14.6){%
\psset{unit=.8}
\rput(0, 1.0){\rnode{1}{\psdot(0,0)}}
\rput(-0.9510565162951535, 0.3090169943749475){\rnode{2}{\psdot(0,0)}}
\rput(-0.5877852522924732, -0.8090169943749473){\rnode{3}{\psdot(0,0)}}
\rput(0.5877852522924729, -0.8090169943749476){\rnode{4}{\psdot(0,0)}}
\rput(0.9510565162951536, 0.3090169943749472){\rnode{5}{\psdot(0,0)}}
\ncline{2}{3}
\ncput*{\blue {\small $3$}}
\ncline{1}{3}
\ncput*{\blue {\small $4$}}
\ncline{1}{4}
\ncput*{\blue {\small $1$}}
\ncline{4}{5}
\ncput*{\blue {\small $2$}}}
\end{pspicture}
  \caption{Self dual e-trees}
  \label{fig:sdetrees}
\end{figure}
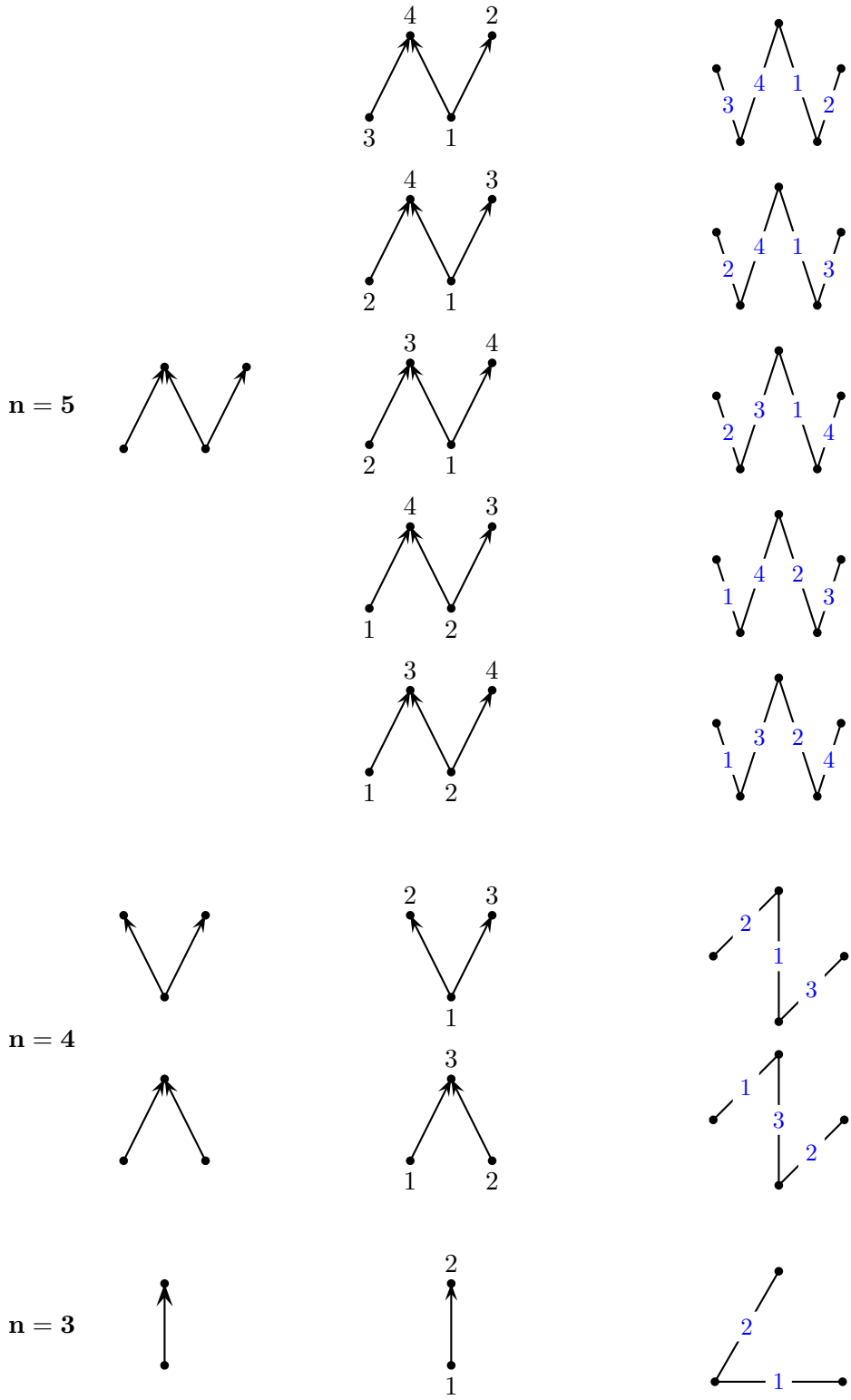

One can also ask if there are any self-dual rooted e-trees and if so,
how many.  The answer turns out to be again the Euler up/down numbers.
To see this notice that it follows from the discussion in
Section~\ref{sec:rootdual} (see Definition~\ref{defn:flagged}) that 
in order for a flagged PCD to be self-dual it is necessary that it's
flag is a trivial chain whose only vertex is a maximal leaf.  Therefore
an rooted e-tree is self-dual if and only if it's medial ditree
is a zigzag ditree, and the flag of its PCD is a maximal leaf.  For 
even $m$ each of the updown and downup ditrees with $m$ vertices
has exactly one maximal leaf, while for odd $n$ the updown ditree
has has no maximal leaf, while the downup ditree has exactly two.
So we also have the following corollary:

\begin{cor}
  \label{cor:zigzagrooted}
  The number of rooted self-dual e-trees with $n$ vertices is also
  equal to the $(n-1)$-th Euler up/down number.
\end{cor}

\begin{rem}
  \label{rem:sdevtrees} This means that if we use the alternative
  duality of Remark~\ref{rem:rootcentral} there are self-dual
  e-v-trees, since that duality is simply a conjugate of the duality
  of rooted e-trees.
\end{rem}

\section{Future directions}
\label{sec:future}
As we mentioned in the introduction, we expect that the main
application of this work will be in finding new, as well as
explaining already known results about Hurwitz numbers.  Pegs are more
attuned to the graph theoretical properties of the graph than cegs.

We conclude by listing a few further works that will use the theory
developed in this paper.

\subsection{Almost minimal factorizations of cycles}
\label{sec:unicycles}
The next class of e-graphs after trees is the class of e-unicycles,
i.e. e-graphs with a unique cycle.  Using Item~\ref{item:euler} of
Lemma~\ref{lem:eulermu}, and the classification of surfaces, we see
that the surface of the peg of such an e-unicycle is an annulus.  It
follows that the monodromy of an e-unicycle is a product of two disjoint
cycles, this was also observed in~\cite{Arnold1996} using a branched
covering argument.  The mind-body dual of an e-unicycle is also an
e-unicycle and so one obtains interesting ``structural'' bijections
between different classes of e-unicycles.

A particularly simple case is the case of e-unicycles whose monodromy
has a fixed point.  The set of these unicycles, with $n+1$ vertices
has cardinality $n^{n}$ and a question posed in~\cite{GouldenYong}, is
(or rather can be interpreted to be) whether mind-body duality can be
used to ``explain'' this simple counting.  It turns out that the migt
of the fixed vertex of such a unicycle is (the closed trail
corresponding to) the unique cycle, and so one of the ``structural''
properties of mind-body duality is that it takes the neighborhood of
the fixed point to the unique cycle, and that fact can be used to
provide bijections between various subsets of the set of those
e-unicycles.

The factorization that corresponds to an e-unicycle with a fixed
vertex expresses an $n$-cycle as the product of $n+1$ transpositions,
while the minimum number of transpositions needed is $n-1$, so we call
such factorizations \emph{almost minimal}.  Those and related topics
will be studied in~\cite{Apostolakis2018d}.

\subsection{Duality for non-crossing trees}
\label{sec:futurenc}

As we've mentioned a notion of duality for non-crossing trees has been
defined in~\cite{Hernando1999}, however in the context that it was
defined (vertex-labeled non-crossing trees), that duality is not
involutory: the dual of the dual is not the original, it becomes
involutory, and coincides with the peg duality as we defined it in
Definition~\ref{defn:dualpeg}, only if we descend to the level of
unlabeled non-crossing trees.  That ``duality'' is closely related
with the action of the Garside element (see Section~\ref{sec:dualhurw})
and it's periodic with period a multiple of~$n$.

In~\cite{Apostolakis2018a} we use the idea of
Remark~\ref{rem:rootcentral} to define a ``true'' duality for
(labeled) non-crossing trees, ask the question ``how many self-dual
non-crossing trees are there?'' and get an interesting answer.

\subsection{Duality for increasing trees}
\label{sec:increasefut}

The class of increasing trees is well studied in the literature, for
example see~\cite{VarietiesIncreasing1992}, these are rooted v-trees
in which the children of every vertex have labels greater than the
vertex.  It follows that the root is labeled $1$, and we can apply the
inverse of the sliding operation $\mathcal{E}_n^{*} \to \mathcal{V}_n$
defined in the proof of Theorem~\ref{thm:slindingbij} to convert the
class of increasing trees to a class of rooted e-trees that turns out
to be closed under the mind-body duality.  So one can define a duality
in the set of increasing trees and study its properties.  This will be
done in~\cite{Apostolakis2018c} where interesting bijection are
obtained for several classes of increasing trees.

We mention that the set of (topsorted) medial ditrees of increasing
trees consists of those binary ditrees that have exactly one minimum,
and that set is obviously in bijection with binary increasing trees, which
in turn are in bijection (see e.g.~\cite{Donaghey1975}) with the set of
alternating permutations.  So the Euler up/down numbers appear again!

\subsection{General theory of Properly Embedded Graphs}
\label{sec:pegsgen}

The focus of this paper is on e-graphs and factorizations, and we
developed enough of the theory of pegs to be able to treat this case.
However there is a more general theory of pegs, that treats the case
of pegs whose medial digraph is not a dag, as well as the case of
graphs properly embedded in non-orientable surfaces. One can even
consider \emph{semi-pegs}, where some of the vertices lie in the
boundary, and some in the interior of the surface. While most of the
ingredients for such a theory are already contained, or have been
hinted on, in this work, there are a few new ingredients needed
for such an extension.  We plan to pursue this in a future
work~\cite{Apostolakis2018b}.

\bibliographystyle{plain}
\bibliography{brcov}

\begin{thebibliography}{10}

\bibitem{Andre1881}
D\'esir\'e Andr\'e.
\newblock Sur les permutations altern\'ees.
\newblock {\em Journal de Mathématiques Pures et Appliqu\'ees}, 7:167--184,
  1881.

\bibitem{Apos2003}
N.~Apostolakis.
\newblock On $4$--fold covering moves.
\newblock {\em Algebraic and Geometric Topology}, 3:117--145, 2003.

\bibitem{Apostolakis2018c}
N.~Apostolakis.
\newblock A duality for increasing trees.
\newblock {\em In Preparation}, 2018.

\bibitem{Apostolakis2018a}
N.~{Apostolakis}.
\newblock {Non-crossing trees, quadrangular dissections, ternary trees, and
  duality preserving bijections}.
\newblock {\em ArXiv e-prints}, July 2018.

\bibitem{Apostolakis2018d}
N.~Apostolakis.
\newblock On almost minimal factorizations of a cycle.
\newblock {\em In Preparation}, 2018.

\bibitem{Apostolakis2018b}
N.~Apostolakis.
\newblock On graphs properly embedded on bounded surfaces.
\newblock {\em In Preparation}, 2018.

\bibitem{Archdeacon1992}
Dan Archdeacon.
\newblock The medial graph and voltage-current duality.
\newblock {\em Discrete Math.}, 104(2):111--141, 1992.

\bibitem{Arnold1996}
V.I. Arnold.
\newblock Topological classification of trigonometric polynomials and
  combinatorics of graphs with an equal number of vertices and edges.
\newblock {\em Functional Analysis and Its Applications}, 30(1):1--14, 1996.

\bibitem{bang2002diagraphs}
J.~Bang-Jensen and G.~Gutin.
\newblock {\em Diagraphs: Theory, Algorithms and Applications}.
\newblock Digraphs: Theory, Algorithms and Applications. Springer, 2002.

\bibitem{VarietiesIncreasing1992}
Fran\c{c}ois Bergeron, Philippe Flajolet, and Bruno Salvy.
\newblock Varieties of increasing trees.
\newblock In {\em C{AAP} '92 ({R}ennes, 1992)}, volume 581 of {\em Lecture
  Notes in Comput. Sci.}, pages 24--48. Springer, Berlin, 1992.

\bibitem{BernsEdm1979}
Israel Berstein and Allan~L. Edmonds.
\newblock On the construction of branched coverings of low-dimensional
  manifolds.
\newblock {\em Trans. Amer. Math. Soc.}, 247:87--124, 1979.

\bibitem{Birman1974}
Joan~S. Birman.
\newblock {\em Braids,links, and the mapping class groups}.
\newblock Number~82 in Annals of Mathematics Studies. Princeton University
  Press, Princeton New Jersey, 1972.

\bibitem{Bondy1990}
J.~A. Bondy.
\newblock Perfect path double covers of graphs.
\newblock {\em J. Graph Theory}, 14(2):259--272, 1990.

\bibitem{BCMMC2002}
C.~Paul Bonnington, Marston Conder, Margaret Morton, and Patricia McKenna.
\newblock Embedding digraphs on orientable surfaces.
\newblock {\em J. Combin. Theory Ser. B}, 85(1):1--20, 2002.

\bibitem{CataWaj1991}
Fabrizio Catanese and Bronislaw Wajnryb.
\newblock The fundamental group of generic polynomials.
\newblock {\em Topology}, 30(4):641--651, 1991.

\bibitem{Damiani2017}
Celeste Damiani.
\newblock A journey through loop braid groups.
\newblock {\em Expo. Math.}, 35(3):252--285, 2017.

\bibitem{Denes1959}
J{\'o}zsef D{\'e}nes.
\newblock The representation of a permutation as the product of a minimal
  number of transpositions, and its connection with the theory of graphs.
\newblock {\em Magyar Tud. Akad. Mat. Kutat\'o Int. K\"ozl.}, 4:63--71, 1959.

\bibitem{sagemath}
The~Sage Developers.
\newblock {\em {S}ageMath, the {S}age {M}athematics {S}oftware {S}ystem
  ({V}ersion 8.0)}, 2017.
\newblock {\tt http://www.sagemath.org}.

\bibitem{Donaghey1975}
Robert Donaghey.
\newblock Alternating permutations and binary increasing trees.
\newblock {\em J. Combinatorial Theory Ser. A}, 18:141--148, 1975.

\bibitem{DPS2014}
E.~{Duchi}, D.~{Poulalhon}, and G.~{Schaeffer}.
\newblock {Bijections for simple and double Hurwitz numbers}.
\newblock {\em ArXiv e-prints}, October 2014.

\bibitem{DulPen1993}
Serge Dulucq and Jean-Guy Penaud.
\newblock Cordes, arbres et permutations.
\newblock {\em Discrete Math.}, 117(1-3):89--105, 1993.

\bibitem{DyerGrossman19811981}
Joan~L. Dyer and Edna~K. Grossman.
\newblock The automorphism groups of the braid groups.
\newblock {\em Amer. J. Math.}, 103(6):1151--1169, 1981.

\bibitem{Eden1962}
Murray Eden and M.~P. Sch{\"u}tzenberger.
\newblock Remark on a theorem of {D}\'enes.
\newblock {\em Magyar Tud. Akad. Mat. Kutat\'o Int. K\"ozl.}, 7:353--355, 1962.

\bibitem{EvansHuang2014}
Ron Evans and Lihua Huang.
\newblock Mind switches in {\it {f}uturama} and {\it {s}targate}.
\newblock {\em Math. Mag.}, 87(4):252--262, 2014.

\bibitem{FBMprim2012}
Benson Farb and Dan Margalit.
\newblock {\em A primer on mapping class groups}, volume~49 of {\em Princeton
  Mathematical Series}.
\newblock Princeton University Press, Princeton, NJ, 2012.

\bibitem{FennRourkeRimanyi1997}
Roger Fenn, Rich{\'a}rd Rim{\'a}nyi, and Colin Rourke.
\newblock The braid-permutation group.
\newblock {\em Topology}, 36(1):123--135, 1997.

\bibitem{GAP4}
The GAP~Group.
\newblock {\em {GAP -- Groups, Algorithms, and Programming, Version 4.7.8}},
  2015.

\bibitem{GouldenYong}
Ian Goulden and Alexander Yong.
\newblock Tree-like properties of cycle factorizations.
\newblock {\em J. Combin. Theory Ser. A}, 98(1):106--117, 2002.

\bibitem{gross1987topological}
J.L. Gross and T.W. Tucker.
\newblock {\em Topological Graph Theory}.
\newblock Dover Books on Mathematics Series. Dover Publications, 1987.

\bibitem{Hernando1999}
M.C Herando.
\newblock {\em Complejidad de Estructuras Geom{\'e}tricas y Combinatorias}.
\newblock PhD thesis, Universitat Polit{\`e}cnica de Catalunya, 1999.

\bibitem{Hurwitz1902}
A.~Hurwitz.
\newblock Ueber die anzahl der riemann'schen flächen mit gegebenen
  verzweigungspunkten.
\newblock {\em Mathematische Annalen}, 55:53--66, 1902.

\bibitem{TuraevKassel2008}
Christian Kassel and Vladimir Turaev.
\newblock {\em Braid groups}, volume 247 of {\em Graduate Texts in
  Mathematics}.
\newblock Springer, New York, 2008.
\newblock With the graphical assistance of Olivier Dodane.

\bibitem{LandoZvonkin2004}
Sergei~K. Lando and Alexander~K. Zvonkin.
\newblock {\em Graphs on surfaces and their applications}, volume 141 of {\em
  Encyclopaedia of Mathematical Sciences}.
\newblock Springer-Verlag, Berlin, 2004.
\newblock With an appendix by Don B. Zagier, Low-Dimensional Topology, II.

\bibitem{MarklShniderStasheff2002}
Martin Markl, Steve Shnider, and Jim Stasheff.
\newblock {\em Operads in algebra, topology and physics}, volume~96 of {\em
  Mathematical Surveys and Monographs}.
\newblock American Mathematical Society, Providence, RI, 2002.

\bibitem{Moon1970}
J.~W. Moon.
\newblock {\em Counting labelled trees}, volume 1969 of {\em From lectures
  delivered to the Twelfth Biennial Seminar of the Canadian Mathematical
  Congress (Vancouver}.
\newblock Canadian Mathematical Congress, Montreal, Que., 1970.

\bibitem{Moszkowski1989}
Paul Moszkowski.
\newblock A solution to a problem of {D}\'enes: a bijection between trees and
  factorizations of cyclic permutations.
\newblock {\em European J. Combin.}, 10(1):13--16, 1989.

\bibitem{Noy1998301}
Marc Noy.
\newblock Enumeration of noncrossing trees on a circle.
\newblock {\em Discrete Mathematics}, 180(1–3):301 -- 313, 1998.
\newblock Proceedings of the 7th Conference on Formal Power Series and
  Algebraic Combinatorics.

\bibitem{Pengelley1975}
David Pengelley.
\newblock Self-dual orientable embeddings of {$K\sb{n}$}.
\newblock {\em J. Combinatorial Theory Ser. B}, 18:46--52, 1975.

\bibitem{Poulalhon1997}
D.~Poulalhon.
\newblock {\em Graphes et d{\'e}compositions de permutations}.
\newblock Number LIX in M{\'e}moire de DEA. {\'E}cole Polytechnique, 1997.

\bibitem{oeis}
N.~J.~A. Sloane, Editor.
\newblock {\em The On-Line Encyclopedia of Integer Sequences}.
\newblock published electronically at \url{https://oeis.org}.

\bibitem{StanleyVolI}
Richard~P. Stanley.
\newblock {\em Enumerative combinatorics. {V}olume 1}, volume~49 of {\em
  Cambridge Studies in Advanced Mathematics}.
\newblock Cambridge University Press, Cambridge, second edition, 2012.

\end{thebibliography}

\end{document}
